\documentclass{amsart}
\usepackage{amssymb}
\usepackage[vcentermath]{youngtab}
\usepackage[dvips]{graphicx}
\usepackage{amsmath,amsthm}
\usepackage[all]{xy}
\usepackage{xcolor}
\usepackage{tabmac}
\usepackage{verbatim}

\providecommand{\U}[1]{\protect\rule{.1in}{.1in}}
\newtheorem{theorem}{Theorem}[section]

\newtheorem{conjecture}[theorem]{Conjecture}
\newtheorem{corollary}[theorem]{Corollary}

\newtheorem{lemma}[theorem]{Lemma}

\newtheorem{proposition}[theorem]{Proposition}
\theoremstyle{definition}

\newtheorem{definition}[theorem]{Definition}
\newtheorem{notation}[theorem]{Notation}
\newtheorem{example}[theorem]{Example}
\newtheorem{remark}[theorem]{Remark}

\numberwithin{equation}{section}

\newcommand{\abs}[1]{|#1|}

\newcommand{\al}{\alpha}
\newcommand{\Aut}{\mathrm{Aut}}
\newcommand{\AutAff}{\widehat{\mathrm{Aut}}}
\newcommand{\as}{\mathbf{a}}
\newcommand{\ba}{\overline{b}}
\newcommand{\bamin}{\overline{b}_{\mathrm{min}}}
\newcommand{\Bh}{\widehat{B}}
\newcommand{\Bhp}{L}
\newcommand{\bla}{{\overline{\la}}}

\newcommand{\bs}{\mathbf{b}}

\newcommand{\CC}{\mathcal{C}}
\newcommand{\CCH}{\mathcal{C}_h}
\newcommand{\CCHinf}{\mathcal{C}_h^\infty}
\newcommand{\CCinf}{\mathcal{C}^\infty}
\newcommand{\card}{\mathrm{card}}
\newcommand{\di}{\diamondsuit}
\newcommand{\Db}{\overline{D}}
\newcommand{\dm}{\delta^-}

\newcommand{\geh}{\mathfrak{g}}

\newcommand{\gl}{\mathfrak{gl}}
\newcommand{\Hb}{\overline{H}}

\newcommand{\heh}{\mathfrak{h}}
\newcommand{\hmax}{\widehat{\max}}
\newcommand{\Hom}{\mathrm{Hom}}
\newcommand{\hw}{\mathrm{hw}}
\newcommand{\ip}[2]{\langle#1\,,\,#2\rangle}
\newcommand{\Image}{\mathrm{Im}}
\newcommand\la{\lambda}
\newcommand{\lam}{\la^-}

\newcommand{\lamin}{\la_{\mathrm{min}}}

\newcommand\La{\Lambda}
\newcommand\lev{\text{lev}}
\newcommand{\mba}{\mathbf{a}}
\newcommand\ol{\overline}
\newcommand\ot{\otimes}
\newcommand{\Pa}{\mathcal{P}}

\newcommand{\Pvd}{\mathcal{P}^{(1,1)}}
\newcommand{\R}{{\mathbb{R}}}
\newcommand{\rk}{\mathrm{rank}}
\newcommand{\rows}{\mathrm{rows}}
\newcommand{\rowtab}{\mathrm{rowtab}}
\newcommand{\Sbox}{S_{\tableau[pby]{\\}}}
\newcommand{\spl}{\mathbb{S}}
\newcommand{\SR}{\mathrm{S}}
\newcommand{\SSbox}{\spl_{\tableau[pby]{\\}}}
\newcommand{\SSS}{\mathfrak{S}}
\newcommand{\ta}{\mathbf{\tilde{a}}}

\newcommand{\tM}{\tilde{M}}
\newcommand{\tomega}{\tilde{\omega}}
\newcommand{\tops}{\mathrm{tops}}
\newcommand{\tW}{\tilde{W}}
\newcommand\veps{\varepsilon}
\newcommand{\vn}{\varnothing}
\newcommand\vphi{\varphi}
\newcommand{\wt}{\mathrm{wt}}
\newcommand{\wtaf}{\widehat{\wt}}
\newcommand{\Xb}{\overline{X}}
\newcommand{\Z}{\mathbb{Z}}

\newcommand{\BA}{\mathbf{1}}
\newcommand{\BB}{\mathbf{2}}
\newcommand{\BC}{\mathbf{3}}
\newcommand{\BD}{\mathbf{4}}

\newcommand{\ichi}{\ol{1}}
\newcommand{\nii}{\ol{2}}
\newcommand{\kuu}{{}}
\newcommand{\yon}{\ol{4}}
\newcommand{\go}{\ol{5}}
\newcommand{\roku}{\ol{6}}
\newcommand{\nana}{\ol{7}}
\newcommand{\hachi}{\ol{8}}
\newcommand{\kyu}{\ol{9}}

\begin{document}

\title{Affine crystals, one-dimensional sums and parabolic Lusztig $q$-analogues}

\author{C\'edric Lecouvey}
\address{LMPT, Universit\'e Fran\c{c}ois Rabelais, UFR Sciences, 
Parc Grammont, 37200 Tours, France}

\author{Masato Okado}
\address{Department of Mathematical Science, Graduate School of Engineering Science, 
Osaka University, Toyonaka, Osaka 560-8531, Japan}

\author{Mark Shimozono}
\address{Department of Mathematics,
Virginia Tech,
Blacksburg, VA 24061-0123 USA}

\date{}

\maketitle

\begin{abstract}
This paper is concerned with one-dimensional sums in classical affine types.
We prove a conjecture of \cite{SZ} by showing they all decompose in terms of
one-dimensional sums related to affine type $A$ provided the rank of the root
system considered is sufficiently large. As a consequence, any one-dimensional sum
associated to a classical affine root system with sufficiently large rank can
be regarded as a parabolic Lusztig $q$-analogue.\
\end{abstract}

\tableofcontents

\section{Introduction}

\label{S:intro}
Consider $\lambda $ and $\mu $ two partitions with at most $n$
parts. Schur duality asserts 
that the Kostka number $K_{\lambda ,\mu }$ counts both the dimension of the
weight space $\mu $ in the irreducible $\mathfrak{sl}_n$ representation $V(\lambda )$ of
highest weight $\lambda $ and the multiplicity of $V(\lambda )$ in the
tensor product $S^{\mu _{1}}(V)\otimes \cdots \otimes S^{\mu _{n}}(V)$ of
the symmetric powers of the vector representation. Using the Weyl character formula,
the Kostka numbers may be expressed in terms of the Kostant partition
function. The $q$-deformation of this partition function gives rise to the
Kostka polynomials.
The Kostka polynomials are Kazhdan-Lusztig polynomials for the affine Weyl group and thus their coefficients are
nonnegative integers, being dimensions of stalks of intersection cohomology sheaves
on Schubert varieties in the affine flag variety. They
also admit a nice combinatorial description in terms of the Lascoux-Sch\"{u}tzenberger 
charge statistic on semistandard tableaux. 

The Kostka polynomials also appear in the representation theory of the
quantum affine algebra $U_{q}(\widehat{\mathfrak{sl}}_{n})$. This was
established by Nakayashiki and Yamada \cite{NY} by relating the charge
statistic to the energy function, a fundamental grading defined on tensor
products of Kashiwara crystals associated to Kirillov-Reshetikhin modules.\
Their result can be regarded as a quantum analogue of Schur duality. It
is also worth mentioning that the energy function naturally appears in
solvable lattice models in statistical physics. 

The aim of this paper is to establish a generalization 
of the connection observed in \cite{NY}. On the weight multiplicity side,
we consider parabolic Lusztig $q$-analogues. These are polynomials 
which quantize the branching coefficients given by the restriction
of an irreducible representation of a simple Lie algebra $\mathfrak{g}_0$
to a Levi subalgebra. In the case that the Levi is the Cartan subalgebra,
these are Lusztig's $q$-analogue of weight multiplicity, and in the further special
case that $\mathfrak{g}_0=\mathfrak{sl}_n$ they are Kostka polynomials.
We consider \textit{stable} parabolic Lusztig $q$-analogues, which are defined when
$\mathfrak{g}_0$ is of classical type and the weights $\lambda$ and $\mu$
do not involve spin weights and stay away from a certain hyperplane.
The stable parabolic Lusztig $q$-analogues have a well-defined large rank limit.

On the other side we consider tensor products of Kirillov-Reshetikhin modules,
which afford the action of the quantum enveloping algebra associated to an affine 
algebra $\mathfrak{g}$. Their restriction
to the canonical simple Lie subalgebra $\mathfrak{g}_0$ has a natural grading
by the energy function, and taking isotypic components, we obtain polynomials
called one-dimensional sums. A stable one-dimensional sum is one 
in which $\mathfrak{g}_0$ is of classical type and the tensor factors
do not involve spin weights. They are so named because they are stable in the large rank limit.

Our key result is Theorem \ref{Th_dec_X} (previously conjectured
in \cite{SZ}) giving the decomposition of the one-dimensional sums for any
classical affine type in terms of those of affine type $A$. It then suffices
to observe that this decomposition is the same as the decomposition of the
stable parabolic Lusztig $q$-analogues obtained in \cite{Lec}.

\bigskip

Let us give a more detailed description of our results. For an affine Lie 
algebra $\geh$ with classical subalgebra $\geh_0$, there is a
finite-dimensional $U'_q(\geh)$-module with crystal graph
given by the tensor product of Kirillov-Reshetikhin (KR) crystals
\begin{align} \label{E:defB}
  B = B^{r_1,s_1} \otimes \dotsm \otimes B^{r_p,s_p}.
\end{align}
A KR crystal $B^{r,s}$ is indexed by a pair $(r,s)\in I_0 \times \Z_{>0}$ where
$I = \{0,1,\dotsc,n\}$ is the affine Dynkin node set and $I_j=I\setminus\{j\}$ for $j\in I$.
The crystal graph $B$ has a $I_0$-equivariant grading by the coenergy function $\Db_B:B\to\Z_{\ge0}$.
Given a dominant $\geh_0$-weight $\la$, the one-dimensional (1-d) sum $\Xb_{\la,B}(q)$ is the
graded multiplicity of the irreducible highest weight $I_0$-crystal $B(\la)$ in $B$.

Throughout the paper we shall assume that $\geh$ belongs to one of the nonexceptional families of affine root systems.
Fix the sequence $((r_1,s_1),\dotsc,(r_p,s_p))$ representing $B$
and the sequence $(d_1,d_2,\dotsc,d_n)$ such that $\la = \sum_{i\in I_0} d_i \omega_i$ where $\omega_i$
is the $i$-th fundamental weight of $\geh_0$. Throughout the paper $r\in I_0$ is called a spin node
if $r=n$ when $\geh_0=B_n,C_n$ and $r=n-1,n$ when $\geh_0=D_n$. 
In order to take a large rank limit of the 1-d sum $\Xb_{\la,B}(q)$,
we assume that no spin weights appear: none of the $r_i$ are spin nodes, and $d_i=0$ if $i\in I_0$ is a spin node.
A ``spinless" sequence representing $B$ makes sense for large rank,
and the sequence $(d_1,d_2,\dotsc)$ for $\la$ also makes sense provided that we append zeros as necessary.
We associate with the dominant $\geh_0$-weight $\la$ the partition (also denoted $\la$)
that has $d_i$ columns of height $i$ for all $i$.

It was observed in \cite{SZ} that the 1-d sum has a large rank limit
which we shall call a stable 1-d sum, and moreover, that they fall into only four
distinct kinds, which are labeled by the four partitions with at most two cells:
$\varnothing$ (the empty partition), $(1)$, $(2)$, and $(1,1)$.
We write $\Xb^\di_{\la,B}(q)$ for the
stable 1-d sum of kind $\di\in \{\varnothing, (1), (2), (1,1) \}$.

We now describe the kind $\di$ associated to each nonexceptional affine family.
Let $\Pa^\di$ denote the set of partitions whose diagrams can be tiled
(without rotation) by the partition diagram of $\di$. Then $\Pa^\vn$ is the singleton consisting
of the empty partition, $\Pa^{(1)}$ is the set of all partitions, $\Pa^{(2)}$ is the set of partitions
with even row lengths, and $\Pa^{(1,1)}$ is the set of partitions with even column lengths.
Let $\Pa_n$ denote the set of partitions with at most $n$ parts.
Write $\Pa^\di_n = \Pa^\di \cap \Pa_n$. For $(r,s)\in I_0 \times \Z_{>0}$ such that $n$ is large
with respect to $r$ ($n\ge r+2$ suffices) define $\Pa^\di_n(r,s)$ to be the set of partitions $\la\in \Pa_n$
such that the 180-degree rotation of the complement of $\la$ in the $r\times s$ rectangular partition $(s^r)$,
is in the set $\Pa^\di$. We say the affine family of $\geh$ is of kind $\di$ if the KR crystal $B^{r,s}$
(for $n$ large with respect to $r$) has the $I_0$-decomposition
\begin{align}\label{E:KRI0decomp}
  B^{r,s} \cong \bigoplus_{\la\in\Pa^\di_n(r,s)} B(\la)
\end{align}
where $B(\la)$ is the irreducible $U_q(\geh_0)$-crystal of highest weight $\la$.
All nonexceptional affine families are of one of the four kinds \cite{SZ}, and note that
the kind depends precisely on the attachment of the affine Dynkin node $0$ to the rest of the
Dynkin diagram. We use the notation of \cite{Kac}.
\begin{align}\label{E:kinds}
\begin{array}{|c||c|} \hline
\di &\text{$\geh$ of kind $\di$} \\ \hline \hline
\vn & A_n^{(1)} \\ \hline
(1) & D_{n+1}^{(2)}, A_{2n}^{(2)} \\ \hline
(2) & C_n^{(1)} \\ \hline
(1,1) & B_n^{(1)}, A_{2n-1}^{(2)}, D_n^{(1)} \\ \hline
\end{array}
\end{align}
The main purpose of this paper is to establish a conjecture of \cite{SZ}. To state this conjecture, we require some notation.
The partition $\la=(\la_1,\la_2,\dotsc,\la_n)$ (with $\la_{n-1}=\la_n=0$ to avoid spin weights) encodes the
dominant $\geh_0$-weight $\sum_i (\la_i-\la_{i+1})\omega_i$.
For $\la\in\Z^n$ write $|\la|=\sum_i \la_i$ and $|B|:=\sum_i r_i s_i$ for $B$ as above.
Finally, $c_{\la\delta}^\nu$ is the Littlewood-Richardson coefficient \cite{Mac}.

\begin{conjecture} \label{CJ:X=K} \cite{SZ} For $\di\in\{(1),(2),(1,1)\}$
\begin{align}\label{E:X=K}
  \Xb_{\la,B}^{\di}(q) = q^{\frac{\abs{B}-\abs{\la}}{|\di|}} \sum_{\nu \in \Pa_n} \sum_{\delta \in \Pa_n^\di }
c_{\delta\la}^\nu \,\Xb_{\nu,B}^{\vn}(q^{\frac{2}{\left| \diamondsuit \right| }}).
\end{align}
\end{conjecture}

Conjecture \ref{CJ:X=K} gives a simple formula for all stable 1-d sums
in terms of the type $A_n^{(1)}$ 1-d sums, which are fairly well-understood
\cite{SW,Sh:crystal}.
In the case that $B$ has tensor factors of the form $B^{1,s}$, Conjecture \ref{CJ:X=K} was proved
in \cite{Sh} for $\di\in\{(1),(2)\}$ and in \cite{LS} for $\di=(1,1)$. This is much easier than the
general case: for the KR crystals $B^{1,s}$ all computations can be done explicitly.

The purpose of this paper is to prove Conjecture \ref{CJ:X=K} in full generality
(for arbitrary nonspin KR tensor factors). This is achieved in Theorem \ref{Th_dec_X}.
%
We choose specific affine root systems $\geh^\di$ for each $\di\in\{(1),(2),(1,1)\}$. This choice,
the classical subalgebra $\geh_0^\di$, and the affine Dynkin diagram $X(\geh^\di)$ are given below.
\begin{align} \label{E:gehdi}
\begin{array}{|c||c|c|c|} \hline
\di & \geh^\di & \geh_0^\di & X(\geh^\di) \\ \hline \hline
(1) & D_{n+1}^{(2)} & B_n  &
{\xymatrix@R=1ex{
*{\circ}<3pt> \ar@{<=}[r]^<{0} & *{\circ}<3pt> \ar@{-}[r]^<{1} & *{\circ}<3pt> \ar@{..}[r] &*{\circ}<3pt> \ar@{-}[r] & *{\circ}<3pt>\ar@{=>}[r]^<{n-1}^>{n} & *{\circ}<3pt> \\
}} \\ \hline
(2) & C_n^{(1)} & C_n  &
{\xymatrix@R=1ex{
*{\circ}<3pt> \ar@{=>}[r]^<{0} & *{\circ}<3pt> \ar@{-}[r]^<{1} & *{\circ}<3pt> \ar@{..}[r]^<{2}&*{\circ}<3pt> \ar@{-}[r]  & *{\circ}<3pt>\ar@{<=}[r]^<{n-1}^>{n} & *{\circ}<3pt> \\
}} \\ \hline
\begin{matrix}
\\ (1,1)
\end{matrix}&
\begin{matrix}
\\ D_n^{(1)}
\end{matrix}
&
\begin{matrix}
\\ D_n
\end{matrix}
&
{\xymatrix@R=1ex{
*{\circ}<3pt> \ar@{-}[dr]^<{0} & {}& {} & {} & {} & *{\circ}<3pt> \\
{} & *{\circ}<3pt> \ar@{-}[r]^<{2}& *{\circ}<3pt> \ar@{..}[r] & *{\circ}<3pt> \ar@{-}[r]^>{n-2} & *{\circ}<3pt>\ar@{-}[ur]^>{n-1} \ar@{-}[dr]^>n & \\
*{\circ}<3pt> \ar@{-}[ur]^<{1} & {} &{}& {} & {} & *{\circ}<3pt>
}} \\[2mm] \hline
\end{array}
\end{align}
We shall call the three nonexceptional affine root systems $\geh^\di$ \textit{reversible}, since their
affine Dynkin diagrams admit the automorphism
\begin{align}\label{E:sigmaaut}
  \sigma(i) = n - i \qquad\text{for $0\le i\le n$.}
\end{align}
Reversible root systems possess the following properties.
There is an associated automorphism $\sigma$ on KR crystals $B^{r,s}$ for $r$ nonspin (Section \ref{SS:sigmaKR}).
One then extends $\sigma$ to tensor products of KR crystals by applying it to each factor.
This map has a remarkable property: it sends all of the $I_0$-highest weight vertices in any tensor product $B$
of nonspin KR crystals, into the subcrystal (called $\max(B)$) of $I_0$-components whose highest weights $\la$
correspond to partitions with the maximum number of boxes (Theorem \ref{Th_corespondence}).
Surprisingly, one can compute the precise change in the energy function (grading) under $\sigma$
acting on $I_0$-highest weight vertices (Theorem \ref{Th_SR}).
Finally, near the $I_0$-highest weight vertices in $\max(B)$, the crystal
$B$ looks like a similar tensor product $B_A$ of type $A_{n-1}^{(1)}$ crystals and moreover the gradings coincide (Theorem \ref{T:maxD=A}).
Along the way we make use of some $I_0$-crystal embeddings we call splitting maps: row splitting
$B^{r,s} \to B^{r-1,s} \otimes B^{1,s}$ (Section \ref{SS:rowsplit})
and box splitting $B^{1,s} \to (B^{1,1})^{\otimes s}$ (Section \ref{SS:boxsplit}).
These embeddings exist for any nonexceptional $\geh$ and nonspin $r\in I_0$.
When applied to the first tensor factor in a tensor product of KR crystals, row splitting
preserves energy (Theorem \ref{T:Dsplit}) and box splitting preserves coenergy.
We also employ a kind of row splitting map in Section 4 which embeds the highest weight $I_0$-crystals $B(\la)$
of classical type, into a tensor product of $I_0$-crystals
of the form $B(s\omega_1)$. This encoding, which we call the row tableau realization,
allows us to see the shadow (that is, the image under
$\sigma$) of the $I_0$-crystal decomposition of a KR crystal.
For this purpose the well-known Kashiwara-Nakashima tableau realization \cite{KN} of $B(\la)$ is less
illuminating.

\section*{Acknowledgments}
\smallskip\par\noindent
The authors thank Masaki Kashiwara and staff of RIMS, Kyoto University during their visit in January, 
2009, where this work was started. M.O. and M.S. thank Ghislain Fourier and Anne Schilling for their
earlier collaboration related to this work.
C.L. is partially supported by ANR-09-JCJC-0102-01, M.O. by JSPS grant No. 20540016, and 
M.S. by NSF DMS--0652641 and DMS--0652648.

\section{Some classical multiplicity formulae}

\subsection{Notation on classical Lie groups}

\label{subsec_Levi}

In the sequel $G$ 
is one of the complex Lie groups $GL_{n}$, $Sp_{2n}$, $SO_{2n+1}$,
or $SO_{2n}$. We follow the convention of \cite{KT} to realize $G$ as a subgroup
of $GL_{N}$ and its Lie algebra $\geh$ as a subalgebra of $\gl_N$ where
\begin{equation*}
N=\left\{
\begin{tabular}{l}
$n$ when $G=GL_{n},$ \\
$2n$ when $G=Sp_{2n},$ \\
$2n+1$ when $G=SO_{2n+1},$ \\
$2n$ when $G=SO_{2n}.$%
\end{tabular}
\right.
\end{equation*}
With this convention the maximal torus $T_G$ of $G$ and the Cartan
subalgebra $\heh_G$ of $\geh$ coincide respectively with
the subgroup and the subalgebra of diagonal matrices of $G$ and $\geh$.
Similarly the Borel subgroup $B_G$ of $G$ and the Borel subalgebra
$\mathfrak{b}_G$ of $\geh$ coincide respectively with the
subgroup and subalgebra of upper triangular matrices of $G$ and $\geh$.
There is an embedding of Lie algebras $\gl_n \rightarrow \geh$ that restricts
to an embedding $\heh_{GL_n} \rightarrow \heh_G$ of Cartan subalgebras
and of their real forms $\heh_{GL_n}^\R \rightarrow \heh_G^\R$.
Via restriction, there is an isomorphism of the real form of the
weight lattice of $\geh$ with that of $\gl_n$.
For any $i\in \{1,...,n\}$, let $\veps_i:\heh_{GL_n}^\R\rightarrow \R$
be the $(i,i)$ matrix entry function. The functions $\{\veps_i\mid i\in \{1,\dotsc,n\}\}$
form a $\Z$-basis of the weight lattice of $\gl_n$, which we identify with
$\Z^n$ via $\sum_{i=1}^n a_i \veps_i \mapsto (a_1,a_2,\dotsc,a_n)$.
In this way we may regard weights of $\geh$ as elements in $\R^n$.
Let $\Sigma_G^+$ and $R_G^+$ be the
sets of simple and positive roots of $G$, respectively. As usual $\rho_G$
is the half sum of the positive roots of $G$. The set
$\Pa_n$ is contained in the cone of dominant weights of $G$.
Let $V^{G}(\la)$ be the finite dimensional
irreducible $G$-module of highest weight $\la$.
Let $W_G$ be the Weyl group of $G$. Then $W_{GL_n}=S_n$ can be regarded as a
subgroup of any $W_G$ for $G=GL_n, Sp_{2n}, SO_{2n+1}$ or $SO_{2n}$.
Given $\la\in\Z^n$ (the weight lattice of $GL_n$),
let $\bla=(-\la_n,\dotsc,-\la_1)=-w_0^{A_{n-1}}(\la)$ where $w_0^{A_{n-1}}\in W_{GL_n}$
is the longest element and let $\overline{\Pa_n}$ denote the
image of $\Pa_n$ under $\la\mapsto\bla$. Note that for $\la\in \Pa_n$,
the contragredient dual of the polynomial $GL_n$-module $V^{GL_n}(\la)$
is isomorphic to $V^{GL_n}(\bla)$.

\subsection{Decomposition of classical tensor product multiplicities}

For $G=Sp_{2n}$, $SO_{2n+1}$, or $SO_{2n}$, $\diamondsuit\in\{(1),(2),(1,1)\}$
, and $\nu\in\mathcal{P}_{n}$, define the $G$-module
\[
W^{G}_{\di}(\nu)=\bigoplus_{\lambda\in\mathcal{P}_{n}}\bigoplus_{\delta
\in\mathcal{P}_{n}^{\di}} V^{G}(\lambda)^{\oplus c_{\delta\lambda}^{\nu}}.
\]
The module $W^{G}_{\di}(\nu)$ is defined specifically to have irreducible
decomposition which mimics the decomposition of KR modules of kind
$\diamondsuit$ into their classical components.

Let $\eta=(\eta_{1},\dotsc,\eta_{p})$ be a $p$-tuple of positive integers
summing to $n$. Consider $\lambda\in\mathcal{P}_{n}$ and $(\mu^{(1)}%
,\ldots,\mu^{(p)})$ a $p$-tuple of partitions such that $\mu^{(k)}%
\in\mathcal{P}_{\eta_{k}}$ for any $k=1,\dotsc,p$. Define the coefficients
$c_{\mu^{(1)},\ldots,\mu^{(p)}}^{\lambda}$ and $\mathfrak{K} _{\mu
^{(1)},\ldots,\mu^{(p)}}^{\lambda,\diamondsuit}$ by
\begin{align}
V^{GL_{n}}(\mu^{(1)})\otimes\cdots\otimes V^{GL_{n}}(\mu^{(p)}) &
\simeq\bigoplus_{\lambda\in\mathcal{P}_{n}}V^{GL_{n}}(\lambda)^{\oplus
c_{\mu^{(1)},\ldots,\mu^{(p)}}^{\la}}\label{E:KRmult}\\
W_{\di}^{G}(\mu^{(1)})\otimes\cdots\otimes W_{\di}^{G}(\mu^{(p)}) &
\simeq\bigoplus_{\lambda\in\mathcal{P}_{n}}V^{G}(\lambda)^{\oplus\mathfrak
{K}_{\mu^{(1)},\ldots,\mu^{(p)}}^{\lambda,\diamondsuit}}.\label{E:kappadef}
\end{align}
We have the following proposition obtained by specializing at $q=1$ Theorem
4.4.2 in \cite{Lec}. It shows that the coefficients $\mathfrak{K}_{\mu
^{(1)},\ldots,\mu^{(p)}}^{\lambda,\diamondsuit}$ do not depend on the Lie
group $G=Sp_{2n},SO_{2n+1}$ or $SO_{2n}$.

\begin{proposition}
\label{prop_qdual} For $n$ sufficiently large, we have
\[
\mathfrak{K}_{\mu^{(1)},\ldots,\mu^{(p)}}^{\lambda,\diamondsuit}=\sum_{\nu
\in\mathcal{P}_{n}}\ \sum_{\delta\in\mathcal{P}_{n}^{\di}}c_{\lambda,\delta
}^{\nu}\ c_{\mu^{(1)},\ldots,\mu^{(p)}}^{\lambda}.
\]
\end{proposition}

We also recall Littlewood's formula \cite{Li} (see also \cite{HTW}): Write
$\widetilde{\mathcal{P}}_{n}$ for the set of pairs $(\gamma^{+},\gamma^{-})$
such that $\gamma^{-}$and $\gamma^{+}$ are partitions with respectively
$r^{+}$ and $r^{-}$ nonzero parts, and $r^{+}+r^{-}\leq n$.\ We identify each
$(\gamma^{+},\gamma^{-})\in\widetilde{\mathcal{P}}_{n}$ with the $GL_{n}
$-dominant weight $(\gamma_{1}^{+},\ldots,\gamma_{r^{+}}^{+},0^{n-r^{+}-r^{-}
},-\gamma_{r^{-}}^{-},\ldots,-\gamma_{1}^{-})\in\mathbb{Z}^{n}$ and denote by
$V^{GL_{n}}(\gamma^{+},\gamma^{-})$ the corresponding $GL_{n}$-module with
highest weight $(\gamma^{+},\gamma^{-}).$ For all $\nu\in\mathcal{P}_{n}$ and
$(\gamma^{+},\gamma^{-})\in\widetilde{\mathcal{P}}_{n}$
\begin{equation}
\lbrack\downarrow_{GL_{n}}^{G}V^{G}(\nu):V^{GL_{n}}(\gamma^{+},\gamma
^{-})]=\sum_{\delta\in\mathcal{P}_{n}^{\di},\kappa\in\mathcal{P}_{n}}
c_{\gamma^{+},\gamma^{-}}^{\kappa}c_{\delta,\kappa}^{\nu}\label{E:Lit}
\end{equation}
where $G=SO_{2n+1},Sp_{2n},SO_{2n}$ corresponds to $\diamondsuit
=(1),(2),(1,1)$ respectively, $\downarrow_{GL_{n}}^{G}V$ is a $G$-module $V$ restricted
to $GL_n$, and $\lbrack W:V]$ is the multiplicity of the irreducible module $V$ in $W$.

\begin{remark}
\label{R:classicalmult} For $\lambda,\mu,\nu\in\mathcal{P}_{n}$ with
$n\geq\max(\ell(\lambda)+\ell(\mu),\ell(\nu))+2$, if $[V^{G}(\lambda)\otimes
V^{G}(\mu):V^{G}(\nu)]>0$ then $|\nu|\leq|\lambda|+|\mu|$, and if equality
occurs then the multiplicity is the LR coefficient $c_{\lambda\mu}^{\nu}$.
This can be easily deduced from the following formula due to King \cite{Ki}
\[
\lbrack V^{G}(\lambda)\otimes V^{G}(\mu):V^{G}(\nu)]=\sum_{\delta,\xi,\eta
}c_{\delta,\xi}^{\nu}c_{\delta,\eta}^{\lambda}c_{\xi,\eta}^{\mu}
\]
which holds in particular under the assumption $n\geq\max(\ell(\lambda
)+\ell(\mu),\ell(\nu))+2$. The multiplicities are then independent of the
group $G$ considered.
\end{remark}

\section{Crystal generalities}
\label{S:crystalgen}

\subsection{Affine root systems}
\label{SS:affroot}
Let $I=\{0,1,\dotsc,n\}$ be the set of nodes of the affine Dynkin diagram $X$ with
generalized Cartan matrix $(a_{ij})_{i,j\in I}$, all associated with the affine Lie algebra $\geh$.
We use the labeling of affine Dynkin diagrams in \cite{Kac}.
Let $(a_0,\dotsc,a_n)$ and $(a_0^\vee,\dotsc,a_n^\vee)$ be the unique sequences of relatively prime positive integers such that
\begin{align}
  \sum_{j\in I} a_{ij} a_j &= 0 &\qquad&\text{for all $i\in I$} \\
  \sum_{i\in I} a_i^\vee a_{ij} &= 0 &\qquad&\text{for all $j\in I$.}
\end{align}
Then
\begin{align} \label{E:a0vee}
  a_0^\vee=\begin{cases}
  2 &\text{for $\geh=A_{2n}^{(2)}$} \\
  1 &\text{otherwise.}
  \end{cases}
\end{align}

Let $P$ be the affine weight lattice, $P^*=\Hom_{\Z}(P,\Z)$, and $\ip{\cdot}{\cdot}:P^*\times P\to\Z$ the
evaluation pairing.
By definition $P$ has $\Z$-basis denoted $\{\delta/a_0, \La_0,\La_1,\dotsc,\La_n\}$ and $P^*$ has
dual $\Z$-basis $\{d,\al_0^\vee,\al_1^\vee,\dotsc,\al_n^\vee\}$. In particular
\begin{align} \label{E:dualbases}
  \ip{\al_i^\vee}{\La_j} = \chi(i=j)\qquad\text{for $i,j\in I$.}
\end{align}
Here $\chi(P)=1$ if $P$ is true and $\chi(P)=0$ otherwise.
The $\La_i$ are called affine fundamental weights,
$\delta$ is called the null root, $d$ is called the degree derivation, and $\al_i^\vee$ are the simple coroots.
Let $P^+=\{\La\in P\mid \ip{\al_i}{\La}\ge0\text{ for all $i\in I$}\}$ be the set of dominant weights.
Define the elements $\al_j\in P$ (the simple roots) by
\begin{align}\label{E:alphadef}
\al_j = \chi(j=0)\,\delta/a_0 + \sum_{i\in I} a_{ij} \La_i\qquad\text{for $j\in I$.}
\end{align}
One may check that
\begin{align}
\ip{\al_i^\vee}{\al_j} &= a_{ij}\qquad\text{for all $i,j\in I$} \\
\delta &= \sum_{j\in I} a_j \al_j
\end{align}
and that $\{\al_i\mid i\in I\}$ is a linearly independent set.
The canonical central element $c\in P^*$ is defined by
\begin{align}\label{E:cdef}
    c=\sum_{i \in I} a_i^\vee \al_i^\vee.
\end{align}
The \textit{level} of a weight $\la\in P$
is defined by
\begin{align}\label{E:leveldef}
  \lev(\la) = \ip{c}{\la}.
\end{align}
By \eqref{E:a0vee} and \eqref{E:dualbases} we have
\begin{align}
\label{E:levLa}
  \lev(\La_i) &= a_i^\vee \\
\label{E:levLa0}
  \lev(\La_0) &= 1.
\end{align}
Define the lattice $P' = P/(\Z\delta/a_0)$. For $i\in I$, write $\al_i'$ for the image of $\al_i$ under
the natural projection $P\to P'$. Then $\al_0'=-\theta/a_0$. Since $\ip{\al_i^\vee}{\delta}=0$ for all $i\in I$,
$\lev:P'\to\Z$ is well-defined. Denote $P^0 = \{\la\in P'\mid \lev{\la}=0\}$.


Let $\geh_0$ be the simple Lie algebra obtained from $\geh$ by ``omitting the $0$ node".
Let $P_0$ be the weight lattice of $\geh_0$. There is a natural projection $P\to P_0$
with kernel $\Z(\delta/a_0) \oplus \Z\La_0$. Let $\omega_i = \pi(\La_i)$ for $i\in I$
(so that $\omega_0=0$ by convention). Then $P_0 = \bigoplus_{i\in I_0} \Z \omega_i$.
The dual lattice $P_0^*=\Hom_\Z(P_0,\Z)$ has dual $\Z$-basis denoted $\al_i^\vee$ for $i\in I_0$.
There is a natural inclusion $P_0^* \to P^*$ defined by $\al_i^\vee\mapsto\al_i^\vee$ for $i\in I_0$.
There is a natural projection $P'\to P_0$ with section
\begin{equation}\label{E:lift}
\begin{split}
  P_0 &\to P' \\
  \omega_i &\mapsto \La_i - \lev(\La_i)\La_0 = \La_i - a_i^\vee \La_0 \qquad\text{for $i\in I_0$.}
\end{split}
\end{equation}
The image of this section is $P^0$.

Let $P_0^+ = \{\la\in P_0 \mid \ip{\al_i^\vee}{\la}\ge0\text{ for all $i\in I_0$}\}$ be the dominant
weights in $P_0$.
Let $Q_0 = \bigoplus_{i\in I_0} \Z \al_i$ be the sublattice of $P_0$ given by the
root lattice.

\subsection{The extended affine Weyl group and Dynkin automorphisms}
\label{SS:aut}

The affine Weyl group $W$ is the subgroup of the group $\Aut(P)$ of linear automorphisms of $P$
generated by the maps
\begin{align*}
  s_i \la = \la - \ip{\al_i^\vee}{\al} \,\al_i\qquad\text{for $\la\in P$ and $i\in I$.}
\end{align*}
The action of $W$ on $P^*$ is defined by either of the equivalent formulae:
\begin{align*}
  \ip{w\cdot\mu}{w\cdot\la} &= \ip{\mu}{\la}&\qquad&\text{for $w\in W$, $\la\in P$, $\mu\in P^*$} \\
  s_i \mu &= \mu - \ip{\mu}{\al_i} \,\al_i^\vee&\qquad&\text{for $\mu\in P^*$, $i\in I$.}
\end{align*}
We write $W_0$ for the Weyl group of $\geh_0$, which is the subgroup of $W$ generated
by $s_i$ for $i\in I_0$. $W_0$ acts on $P_0$ and $P_0^*$.

Let $\Aut(X)$ be the group of automorphisms of the affine Dynkin diagram $X$. Let $\tau\in\Aut(X)$.
By definition $\tau$ is a permutation of the Dynkin node set $I$ of $X$ such that there is a
bond of multiplicity $m$ from $i\in I$ to $j\in I$
if and only if there is a bond of multiplicity $m$ from $\tau(i)$ to $\tau(j)$, for all $i,j\in I$.
In particular,
\begin{align}
\label{E:autoa}
a_{\tau(i)}&=a_i &\qquad&\text{and} \\
\label{E:autoavee}
a_{\tau(i)}^\vee&=a_i^\vee&\qquad&\text{for $i\in I$} \\
\label{E:autoCartan}
a_{\tau(i),\tau(j)}&=a_{ij}&\qquad&\text{for $i,j\in I$.}
\end{align}
$\tau\in\Aut(X)$ induces $\tau\in \Aut(P)$ by $\tau(\delta/a_0)=\delta/a_0$ and $\tau(\La_i)=\La_{\tau(i)}$ for all $i\in I$.
This satisfies $\tau(\al_i)=\al_{\tau(i)}$ for all $i\in I$. $\tau\in\Aut(X)$ also induces $\tau\in \Aut(P^*)$ by
\begin{align}
  \ip{\tau(\mu)}{\tau(\la)} = \ip{\mu}{\la}\qquad\text{for all $\la\in P$ and $\mu\in P^*$.}
\end{align}
It satisfies $\tau(d)=d$ and $\tau(\al_i^\vee)=\al_{\tau(i)}^\vee$ for all $i\in I$. $\tau\in\Aut(X)$
induces an automorphism $\tau$ on $W$ denoted $w\mapsto w^\tau$ where
$s_i^\tau = s_{\tau(i)}$ for all $i\in I$.

Define the subset of \textit{special nodes} $I^s\subset I$ to be the orbit of $0\in I$ under $\Aut(X)$.
Every element of $\Aut(X)$ is determined by its action on $I^s$. Let
\begin{align}\label{E:thetadef}
\theta = \sum_{i\in I_0} a_i \al_i = \delta - a_0 \al_0.
\end{align}
If $\geh$ is untwisted then $\theta$ is the highest root of $\geh_0$.
Let $M\subset P_0$ be the sublattice generated by the $W_0$-orbit of $\theta/a_0$:
\begin{align}
\label{E:Mdef}
  M &= \sum_{w\in W_0} \Z \,w \cdot (\theta/a_0).
\end{align}
The semidirect product $W_0 \ltimes P_0$ acts on $P'$ by
\begin{align}
  (w t_\la) \cdot \La = w( \La + \lev(\La) \la)\qquad\text{for $w\in W_0$, $\la\in P_0$, and $\La\in P'$}
\end{align}
where $\la$ is regarded as an element of $P^0\subset P'$ via \eqref{E:lift} and $t_\la$ is the translation
corresponding to $\la$.
We have
\begin{equation}
\begin{split}
  W &\cong W_0 \ltimes M \\
  s_0&\mapsto s_\theta \, t_{-\theta/a_0}.
\end{split}
\end{equation}
For each $\ell\in\Z$ the action of $W_0 \ltimes P_0$ on $P'$ stabilizes the
affine subspace $\ell\La_0 + P^0\subset P'$ of level $\ell$ weights.
Therefore for each $\ell\in \Z$, the level $\ell$ action is defined by the
representation $\pi_\ell:W_0\ltimes P_0\to \AutAff(P_0)$
by affine linear automorphisms of $P_0$, given by
\begin{equation}
\begin{split}
  \pi_\ell(w t_\la)\cdot \beta &= -\ell\La_0 + w t_\la(\ell\La_0+\beta) \\
  &= w (\beta+\ell\la)\qquad\text{for $w\in W_0$, $\la,\beta\in P_0$.}
\end{split}
\end{equation}

For $r\in I_0$ define \cite{HKOTT}
\begin{align}\label{E:crdef}
c_r=\max(1,a_r/a_r^\vee).
\end{align}
\begin{remark}\label{R:cr}
We have $c_r=1$ except that $c_r=2$ for $\geh=B_n^{(1)}$ with $r=n$, $\geh=C_n^{(1)}$ with $1\le r\le n-1$,
$\geh=F_4^{(1)}$ with $r\in \{3,4\}$, and $c_r=3$ if $\geh=G_2^{(1)}$ with $r=2$. In particular $c_i=1$ for $i\in I^s$.
\end{remark}
Define the sublattices of $P_0$ given by
\begin{align}
\label{E:M'def}
  M' &= \bigoplus_{i\in I_0} \Z c_i \al_i \\
\label{E:tMdef}
\tM &= \bigoplus_{i\in I_0} \Z c_i \omega_i.
\end{align}
We have $\tM \supset M \supset M'$ with $M=M'$
except for $\geh=A_{2n}^{(2)}$ where $M'\subset M$ is a sublattice of index $2$.
We define an injective group homomorphism
\begin{align}
  \tM/M \hookrightarrow \Aut(X)
\end{align}
with image denoted $\Sigma$. First, there is a bijection $I^s\to \tM/M$
given by $i\mapsto c_i\omega_i+M$. Subtraction by $c_i\omega_i+M$ induces a permutation of $\tM/M$.
The induced permutation of $I^s$ under the above bijection, extends to $\tau^i\in\Aut(X)$.
We define $\Sigma=\{\tau^i\mid i\in I^s\}$; it is the group of \textit{special automorphisms}.

Define the extended affine Weyl group (in particular for twisted affine types) by
\begin{align}\label{E:exaffWeyldef}
\tW = W \rtimes \Sigma
\end{align}
via $\tau w \tau^{-1} = w^\tau$ for $\tau\in \Sigma$ and $w\in W$.
We have $\tW \cong W_0 \ltimes \tM$ with
\begin{align}
\label{E:specialauttrans}
\tau^i &= w_0^{\omega_i} t_{-c_i \omega_i}\qquad\text{for $i\in I^s$, where} \\
\label{E:w0ladef}
&\text{$w_0^\la\in W_0$ is the shortest element such that $w_0^\la \,\la = w_0 \la$.}
\end{align}

\begin{remark}
In untwisted type one may identify $M$ with the coroot lattice $Q_0^\vee$
and $\tM$ with the coweight lattice $P_0^\vee$, although these identifications may involve
some uniform dilation.
\end{remark}

\begin{example}\label{X:special}
\begin{align} \label{E:specnodes}
\begin{array}{|c||c|c|c|c|c|c|c|} \hline
\geh &A_n^{(1)}&B_n^{(1)}&C_n^{(1)}&D_n^{(1)}&A_{2n-1}^{(2)}&A_{2n}^{(2)}&D_{n+1}^{(2)} \\ \hline
I^s & \{0,1,\dotsc,n\} & \{0,1\} & \{0,n\} & \{0,1,n-1,n\}& \{0,1\} & \{0\}& \{0,n\} \\ \hline
\end{array}
\end{align}

For $A_n^{(1)}$ and $i\in I^s$, $\tau^i$ subtracts $i$ mod $n+1$.

For $D_n^{(1)}$, in terms of permutations of $I^s$, are defined as follows.
$\tau^0$ is the identity and $\tau^1=(0,1)(n-1,n)$. If $n$ is odd,
$\tau^{n-1}=(0,n,1,n-1)$ and $\tau^n=(0,n-1,1,n)$ and if $n$ is even,
$\tau^{n-1}=(0,n-1)(1,n)$ and $\tau^n=(0,n)(1,n-1)$.
\end{example}

Note that $\tM/M$ admits an involution given by negation.
The corresponding affine Dynkin involution is given as follows.
Let $w_0\in W_0$ be the longest element. The map $\alpha\mapsto -w_0\alpha$
is an involution on the set of positive roots of $\geh_0$ that sends sums to sums, and therefore
restricts to an involution on the set of simple roots. So there is an involutive
automorphism of the Dynkin diagram of $\geh_0$ denoted $i\mapsto i^*$,
defined by
\begin{align}\label{E:*def}
  - w_0 \al_i &= \al_{i^*}\qquad\text{for $i\in I_0$.}
\end{align}
This extends to an element denoted $*\in \Aut(X)$ by defining $0^*=0$.
The induced automorphism of $P$ is given by
\begin{align} \label{E:*weight}
  \la \mapsto - w_0 \la \qquad\text{for $\la\in P$.}
\end{align}
In particular
\begin{align}
  -w_0 \omega_i = \omega_{i^*}\qquad\text{for $i\in I_0$.}
\end{align}
By \eqref{E:autoa}, \eqref{E:autoavee}, and \eqref{E:crdef},
we see that
\begin{align}\label{E:*c}
  c_{i^*} = c_i \qquad\text{for $i\in I$.}
\end{align}
Therefore $-w_0 c_i \omega_i = c_{i^*} \omega_{i^*}$.
Since $w_0 c_i \omega_i + M = c_i \omega_i + M$,
we have $c_{i^*} \omega_{i^*} + M = - c_i \omega_i + M$ in
the group $\tM/M\cong \Sigma$. It follows that for all $i\in I^s$,
negation in $\tM/M$ corresponds to the involution $i\mapsto i^*$ on $I^s$, and that
\begin{equation} \label{E:auto*}
\begin{split}
  \tau^i(0) &= i^* \\
  (w_0^{\omega_i})^{-1} &= w_0^{\omega_{i^*}}.
\end{split}
\end{equation}

\begin{example}\label{X:*}
We have $i^*=i$ except in the following cases.
For $A_{n-1}$ we have $i^*=n-i$.
For $D_n$ and $n$ odd, $(n-1)^*=n$ and $n^*=n-1$.
For $E_6$ $i\mapsto i^*$ is the unique nontrivial automorphism.
\end{example}

\subsection{Crystals}
\label{SS:crystals}

Let $\geh$ be an affine Lie algebra.
We consider the following categories of crystal graphs of
modules over a quantum affine algebra $U'_q(\geh)$: $\CCH(\geh)$, direct sums
of affine highest weight crystals, and $\CC(\geh)$, tensor products of
Kirillov-Reshetikhin (KR) crystals. For KR crystals we refer to \cite{FOS1}.
Let $\CCH(\geh_0)$ be the category of direct sums of crystal graphs
of highest weight $U_q(\geh_0)$-modules.

Let $B$ be a crystal in one of the above categories.
$B$ is a graph with vertex set also denoted $B$ and directed edges labeled by the
elements of the set $K$ of Dynkin nodes of $\geh$. We call $B$ a $K$-crystal. For $K'\subset K$
write $B_{K'}(b)$ for the $K'$-connected
component of $b\in B$, that is, the connected component of the graph in which
all directed edges are removed except those labeled by $K'$.
For $i\in K$, each $\{i\}$-connected component
is a finite directed path called an $i$-string. Then for $b\in B$,
$f_i(b)$ (resp. $e_i(b)$) is the next (resp. previous) vertex on the $i$-string of $b$
if it exists, and is declared to be the special symbol $0$ otherwise.
Let $\vphi_i(b)$ and $\veps_i(b)$ denote the number of steps to the end (resp. start)
of the $i$-string of $b$. For a sequence $\as=(i_1,\dotsc,i_p)$ of indices in $K$ define
\[
e_{\as}(b)=e_{i_1}(e_{i_2}(\dotsm e_{i_p}(b)\dotsm))
\]
and $f_{\as}(b)$ similarly.

For $B\in\CC(\geh)$ or $B\in\CCH(\geh)$, let $K=I$ and
define the functions $\vphi,\veps:B\to P'$ by
\begin{align}\label{L:phidef}
  \vphi(b) &= \sum_{i\in I} \vphi_i(b) \La_i \\
  \veps(b) &= \sum_{i\in I} \veps_i(b) \La_i.
\end{align}
For $B\in\CCH(\geh_0)$ we have $\vphi,\veps:B\to P_0$
with $I$ replaced by $I_0$ and $\La_i$ replaced by $\omega_i$.

For $B\in\CC(\geh)$ or $B\in\CCH(\geh)$
we define the weight function $\wt:B\to P'$ by
\begin{align} \label{E:wtdef}
  \wt(b) = \vphi(b) - \veps(b).
\end{align}
For $B\in\CC(\geh)$ the values of $\wt$ lie in the level zero sublattice $P^0\subset P'$.
For $B\in\CCH(\geh_0)$ we have $\wt:B\to P_0$ defined by \eqref{E:wtdef}.

For $B\in\CC(\geh)$ or $B\in\CCH(\geh)$ we have
\begin{align} \label{E:wte}
\wt(e_i(b))&=\wt(b)+\alpha_i'&\qquad&\text{for $i\in I$ if $e_i(b)\ne0$} \\
\label{E:wtf}
\wt(f_i(b))&=\wt(b)-\alpha_i'&\qquad&\text{for $i\in I$ if $f_i(b)\ne0$.}
\end{align}
For $B\in\CCH(\geh_0)$ the same conditions hold with $\alpha_i'=\alpha_i$
and $i\in I_0$.

For $K'\subset K$, the set of $K'$-highest weight vertices in the $K$-crystal $B$ is defined by
\begin{align*}
  \hw_{K'}(B) = \{b\in B \mid e_i(b) = 0 \text{ for all $i\in K'$ }\}.
\end{align*}
Let $\hw_{K'}(b)$ denote the unique $K'$-highest weight vertex in the $K'$-component of $b$.

If $\la$ is a $K'$-dominant weight then define
\begin{align}
  \hw_{K'}^\la(B) = \{b\in \hw_{K'}(B)\mid \wt_{K'}(b)=\la \}
\end{align}
for the subset of $\hw_{K'}(B)$ of vertices of weight $\la$
and $B(\la)=B_{K'}(\la)$ for the irreducible $K'$-crystal of highest weight $\la$.

Let $B_1,B_2$ be $K$-crystals. Then $B_1\otimes B_2$ is a $K$-crystal
via Kashiwara's tensor convention
\begin{align}\label{E:tensor}
  e_i(b_1\otimes b_2) &=
  \begin{cases}
  e_i(b_1)\otimes b_2 & \text{if $\vphi_i(b_1) \ge \veps_i(b_2)$} \\
  b_1 \otimes e_i(b_2) & \text{if $\vphi_i(b_1)<\veps_i(b_2)$.}
  \end{cases}
\end{align}

\begin{lemma} \label{L:tensorhwv} Let $B_1,B_2$ be $K$-crystals
and $b_1,c_1\in B_1$ and $b_2,c_2\in B_2$ such that $c_1\otimes c_2\in \hw_K(B_1\otimes B_2)$
and $b_1\otimes b_2 \in B_K(c_1 \otimes c_2)$. Then $c_1\in \hw_K(B_1)$ and
$b_1\in B_K(c_1)$.
\end{lemma}
\begin{proof} $c_1\in\hw_K(B_1)$ holds by \eqref{E:tensor}.
Let $\as=(i_1,\dotsc,i_m)$ be a sequence of elements in $K$ such that
$e_{\as}(b_1 \otimes b_2) = c_1 \otimes c_2$. By \eqref{E:tensor}
there is a subsequence $\bs$ of $\as$ such that $e_{\bs}(b_1)=c_1$.
\end{proof}

\subsection{KR crystal generalities}
\label{SS:KRgen}

Let $\CC=\CC(\geh)$ be the tensor category of tensor products of KR crystals
$B^{r,s}$. An $I$-crystal $B$ is \textit{regular} if for all subsets $K\subset I$
with $|K|=2$, the $K$-components of $B$ are isomorphic to
crystal graphs of $U_q(\geh_K)$-crystals where $\geh_K$ is the subalgebra of $\geh$
corresponding to $K$.

\begin{theorem} \label{T:KR}
Let $\geh$ be of nonexceptional affine type.
\begin{enumerate}
\item \cite{OS} For every $(r,s)\in I_0 \times \Z_{>0}$, there is an irreducible
$U'_q(\geh)$-module $W_s^{(r)}$ with affine crystal basis $B^{r,s}$. In particular
every $B\in\CC$ is regular.
\item \cite{FOS1} The affine crystal structure on $B^{r,s}$ is determined.
\end{enumerate}
\end{theorem}

\begin{proposition} \label{P:RH} Let $B_1,B_2\in\CC$.
\begin{enumerate}
\item There is an $I$-crystal isomorphism
$R=R_{B_1,B_2}: B_1 \otimes B_2 \to B_2 \otimes B_1$ called the combinatorial $R$-matrix.
By uniqueness, for $B\in\CC$, $R_{B,B}$ is the identity on $B\otimes B$.
\item There is a unique map
$\Hb=\Hb_{B_1,B_2}: B_1\otimes B_2 \to \Z$, called coenergy function up to additive constant, 
such that $\Hb$ is constant on $I_0$-components, and for
$b=b_1\otimes b_2\in B_1 \otimes B_2$,
\begin{align} \label{E:Hb0}
  \Hb(e_0(b)) &= \Hb(b) +
  \begin{cases}
   -1 & \text{in case LL} \\
   0 &\text{in case LR or RL} \\
   1 &\text{in case RR}
  \end{cases}
\end{align}
where in case LL, when $e_0$ is applied to $b_1\otimes b_2$ and to
$R_{B_1,B_2}(b_1\otimes b_2)=b_2'\otimes b_1'$ as in
\eqref{E:tensor}, it acts on the left factor both times, in case RR
$e_0$ acts on the right factor both times, etc.
\end{enumerate}
\end{proposition}
\begin{proof} Arguing as in \cite{KMN} one may deduce these properties
from the existence of the universal R-matrix, the Yang-Baxter relation
for $R$, and Theorem \ref{T:KR}(1).
\end{proof}

Let $B$ be regular. An element $b\in B$ is called
an \textit{extremal vector} of weight $\la$ if $\wt(b)=\la$ and there exist elements $\{b_w\}_{w\in W}$ such that
\begin{itemize}
\item $b_w=b$ for $w=e$,
\item if $\langle \al_i^\vee,w\la\rangle\ge0$ then $e_i(b_w)=0$ and
    $f_i^{\langle \al_i^\vee,w\la\rangle}(b_w)=b_{s_iw}$,
\item if $\langle \al_i^\vee,w\la\rangle\le0$ then $f_i(b_w)=0$ and
    $e_i^{-\langle \al_i^\vee,w\la\rangle}(b_w)=b_{s_iw}$.
\end{itemize}
A finite regular crystal $B$ with weights in $P^0$ is called \textit{simple} \cite{AK} \cite{NS} if there exists $\la\in P^0$
such that the weight of any extremal vector is contained in $W\la$ and $B$ contains a unique element of weight $\la$.
Here $W$ is the affine Weyl group, which acts on $P^0\cong P_0$ by the level zero action.

\begin{proposition} \label{P:simpleconnected}\
\begin{enumerate}
\item[(1)] Every $B\in\CC$ is simple. In particular $B$ contains a unique extremal vector $u(B)$
with $\wt(u(B))\in P_0^+$. Moreover $u(B^{r,s})\in B^{r,s}$ is the unique
vector of weight $s\omega_r$ and
$u(B_1\otimes B_2)=u(B_1)\otimes u(B_2)$ for $B_1,B_2\in \CC$.
\item[(2)] For every $B\in \CC$, $B$ is $I$-connected.
\end{enumerate}
\end{proposition}
\begin{proof}
By \cite{AK} a simple crystal is connected and the tensor product of simple crystals
is also simple. In \cite[Section 4.2]{NS} Naito and Sagaki proved that
a finite regular crystal $B$ with coenergy function $\Hb_{B\,B}$
is simple.\footnote{Although they assume that $B$ is realized as a fixed point crystal,
their proof is valid under the given condition.}
The equality $u(B_1\ot B_2)=u(B_1)\ot u(B_2)$ follows from the fact that the r.h.s is extremal.
\end{proof}

\begin{remark} \label{R:uniqueisomorphisms}
\begin{itemize}
\item[(1)] Proposition \ref{P:simpleconnected} implies that if there is an $I$-crystal isomorphism 
$g:B\to B'$ for $B,B'\in\CC$, then it is unique: it must satisfy $g(u(B))=u(B')$ and the rest of 
its values are determined since $B$ is $I$-connected.
\item[(2)] For $B_1,B_2\in\CC$ we normalize the coenergy function $\Hb$ by $\Hb(u(B_1\ot B_2))=0$.
\end{itemize}
\end{remark}

The \textit{level} of $B\in\CC$ is defined by
\begin{align}
  \lev(B) = \min_{b\in B} \lev(\vphi(b)) = \min_{b\in B} \lev(\veps(b)).
\end{align}
The subset $B_{\min}\subset B$ is defined by
\begin{align*}
  B_{\min} &= \{ b\in B \mid \lev(\vphi(b))=\lev(B)\} \\
  &= \{ b\in B \mid \lev(\veps(b)) = \lev(B) \}.
\end{align*}

The crystal $B$ is said to be \textit{perfect} (in the sense of \cite{NS}; compare with \cite{KMN})
if $B$ is the crystal graph of a $U'_q(\geh)$-module, $B$ is simple, and
the maps $\vphi$ and $\veps$ are bijections from $B_{\min}$ to the set
of weights $\la\in P'$ that are dominant and have $\lev(\la)=\lev(B)$.

\begin{theorem} \cite{FOS2}
With $c_r$ as in \eqref{E:crdef},
\begin{enumerate}
\item $\lev(B^{r,s}) = \lceil \frac{s}{c_r} \rceil$.
\item $B^{r,s}$ is perfect if and only if $s/c_r\in\Z$.
\end{enumerate}
\end{theorem}

\begin{lemma} \label{L:phimin}
Let $\geh$ be of nonexceptional affine type,
$(r,s)\in I_0\times\Z_{>0}$, $\ell=\lev(B^{r,s})$ and $j\in I^s$.
Then there is a unique element $u_j(r,s)\in B^{r,s}$ such that
$\veps(u_j(r,s))=\ell\La_j$.
Moreover, writing $t_{-c_{r^*} \omega_{r^*}} = w\tau$ for
$w\in W$ and $\tau\in\Sigma$ with $*$ as in \eqref{E:*def} we have
\begin{align} \label{E:minphi}
\vphi(u_j(r,s))) = \begin{cases}
\ell \La_{\tau(j)} & \text{if $B^{r,s}$ is perfect} \\
(\ell-1)\La_n+\La_{n-r} &\text{if $\geh=C_n^{(1)},1\le r\le n-1,j=n$} \\
(\ell-1)\La_{\tau(j)}+\La_r & \text{otherwise.}
\end{cases}
\end{align}
\end{lemma}
\begin{proof}
Suppose first that $B^{r,s}$ is perfect. Then $c_r\ell=s$ and $u_j(r,s)$ is unique.
Moreover the value of $\vphi(u)$ is verified by \cite{FOS2}. Explicitly:
\begin{enumerate}
\item $\geh=A_n^{(1)}$. $u_j(r,s)$ consists of $s$ copies of the same column
that consists of the elements $j+1,j+2,\dotsc,j+r$ (mod $n+1$), sorted into increasing order.
\item $\geh=A_{2n}^{(2)}$. $u_0(r,s)=\hw_{I_0}(B(0))$.
\item $\geh=D_{n+1}^{(2)}$. Suppose $r\not\in I^s$.
$u_0(r,s)=\hw_{I_0}(B(0))$. For $s=2s'$, $u_n(r,s)\in B(s\omega_r)$ is the KN tableau
with $s'$ columns $(n-r+1)\dotsm(n-1)n$ and
$s'$ columns $\overline{n}\overline{n-1}\dotsm \overline{n-r+1}$.
For $s=2s'+1$, $u_n(r,s)\in B(s\omega_r)$ has, in addition to the
columns for $u_n(r,2s')$, a middle column of height $r$ is
given by $0\cdots0$.
For $r=n\in I^s$,
$u_0(n,s)$ (resp. $u_n(n,s)$) is the unique element of $B(s\omega_n)$ of weight
$s\omega_n$ (resp. $-s\omega_n$).
\item $\geh=C_n^{(1)}$. For $r\not\in I^s$, since $c_r=2$ and we are in the perfect
case, $s$ must be even (say $s=2\ell$), and $u_j(r,2\ell)$ is given
as for $D_{n+1}^{(2)}$. For $r=n\in I^s$, again $u_j(n,s)$ is given
as for $D_{n+1}^{(2)}$.
\item $\geh\in\{B_n^{(1)}, D_n^{(1)}, A_{2n-1}^{(2)}\}$.
Recall that for $\geh=B_n^{(1)},r=n$ $B^{r,s}$ is perfect of level $\ell$ when $s=2\ell$.
First let $r\in I_0$ not be a type $D_n^{(1)}$ spin node.
$u_0(2i,s)=\hw_{I_0}(B(0))$ and $u_1(2i,s)\in B(\ell\omega_2)$
has $\ell'$ columns $1\overline{2}$ and $\ell'$ columns $2\overline{1}$ for $\ell=2\ell'$, and
in addition a middle column $2\overline{2}$ for $\ell=2\ell'+1$.
$u_0(2i+1,s)=\hw_{I_0}(B(\ell\omega_1))$ and $u_1(2i+1,s)\in B(\ell\omega_1)$
is the tableau $\bar{1}^\ell$. $D_n^{(1)}$ has additional special nodes $j\in\{n-1,n\}$.
Suppose $r$ is even. For $s=2s'$, $u_n(r,s)\in B(s\omega_r)$ has $s'$ columns
$(n-r+1)\cdots(n-1)n$ and $s'$ columns $\ol{n}\ol{n-1}\cdots\ol{n-r+1}$. For $s=2s'+1$,
$u_n(r,s)$ has, in addition to the columns for $u_n(r,2s')$, a middle column given by
$\ol{n}n\ol{n}n\cdots$. If $r$ is odd, replace $s'$ columns $(n-r+1)\cdots(n-1)n$ with
$(n-r+1)\cdots(n-1)\ol{n}$. $u_{n-1}(r,s)\in B(s\omega_r)$ is given from $u_n(r,s)$ above
by interchanging $n$ and $\ol{n}$. Now let us set $r=n$ for type $D_n^{(1)}$. $u_j(n,s)$
for $j=0,1,n-1,n$ is given by the unique element of $B(s\omega_n)$ of weight $s\omega_n,
s(\omega_{n-1}-\omega_1),s((1-\gamma)\omega_1-\omega_{n-1}),s(\gamma\omega_1-\omega_n)$
where $\gamma=0,1,\gamma\equiv n\,(\text{mod }2)$. If $r=n-1$, we interchange $\omega_n$
and $\omega_{n-1}$ in the above description.
\end{enumerate}

We enumerate the nonperfect cases \cite{FOS2}.
\begin{enumerate}
\item $\geh=B_n^{(1)}$, $r=n$ and $s=2\ell-1$.
For $n$ even, $u_0(n,2\ell-1)=\hw_{I_0}(\omega_n)$ and
$u_1(n,2\ell-1)\in B(\omega_n)$ is defined by $\wt(u)=\omega_n-\omega_1$.
For $n$ odd, $u_0(n,2\ell-1)=\hw_{I_0}(B((\ell-1)\omega_1+\omega_n))$.
$u_1(n,2\ell-1)$ has a half-column consisting of $23\dotsm(n-1)n\bar{1}$
and $\ell-1$ columns consisting of a single $\bar{1}$.
\item $\geh=C_n^{(1)}$ for $1\le r\le n-1$ and $s=2\ell-1$.
$u_0(r,2\ell-1)=\hw_{I_0}(B(\omega_r))$.
$u_n(r,2\ell-1)$ has $\ell-1$ columns $(n-r+1)\dotsm(n-1)n$ and $\ell$ columns
$\overline{n}\overline{n-1}\dotsm\overline{n-r+1}$.
\end{enumerate}
\end{proof}

By Lemma \ref{L:phimin} we may define $m(B^{r,s})=u_0(r,s)\in B^{r,s}$ or equivalently
\begin{align}
\label{E:mdef}
  \veps(m(B^{r,s})) = \lev(B^{r,s}) \La_0.
\end{align}
Similarly, there exists a unique element $m'(B^{r,s})\in B^{r,s}$ such that
\begin{align}
\label{E:mpdef}
  \vphi(m'(B^{r,s})) = \lev(B^{r,s}) \La_0.
\end{align}
Define
\begin{align} \label{E:brsla}
  b(r,s,\la) = \hw_{I_0}^\la(B^{r,s})\qquad\text{for $\la\in\Pa^\di_n(r,s)$.}
\end{align}

\begin{remark} \label{R:minBrs}
Suppose $\diamondsuit\ne\vn$ and $r\in I_0$ is not a spin node.
By \eqref{E:KRI0decomp} the right hand side of \eqref{E:brsla} is a singleton.
We have
\begin{align}
\label{E:ubrs}
  u(B^{r,s}) &= b(r,s,(s^r)) \\
\label{E:mbrs}
  m(B^{r,s}) &= b(r,s,\lamin^\di(r,s))
\end{align}
where $\lamin=\lamin^\di(r,s)\in \Pa^\di(r,s)$ is
the partition with $\abs{\lamin}$ minimum. Explicitly
\begin{align}\label{E:lamin}
  \lamin^\di(r,s) &= \begin{cases}
  (s) & \text{if $r$ is odd and $\di=(1,1)$} \\
  (1^r) & \text{if $s$ is odd and $\di=(2)$} \\
  \varnothing &\text{otherwise.}
  \end{cases}
\end{align}
\end{remark}

\subsection{Grading by intrinsic coenergy}
\label{SS:Ddef}

Each $B\in \CC$ has a canonical $I_0$-equivariant grading by
the intrinsic coenergy function $\Db:B\to \Z$ which is defined as
follows.
\begin{enumerate}
\item If $B=B^{r,s}$ is a KR crystal then define
\begin{align}\label{E:DoneKR}
  \Db_B(b) = \Hb_{B,B}(m'(B) \otimes b) - \Hb_{B,B}(m'(B) \otimes u(B)).
\end{align}
\item
If $B_1,B_2\in \CC$ then
\begin{align} \label{E:intrinsic}
  \Db_{B_1\otimes B_2}(b_1 \otimes b_2) =
  \Db_{B_1}(b_1) + \Db_{B_2}(b_2') +
  \Hb_{B_1,B_2}(b_1\otimes b_2)
\end{align}
where $R_{B_1,B_2}(b_1\otimes b_2)=b_2'\otimes b_1'$.
\end{enumerate}
The resulting grading satisfies $$\Db_{(B_1\otimes B_2)\otimes B_3} =
\Db_{B_1\otimes (B_2\otimes B_3)}$$ for all $B_1,B_2,B_3\in\CC$ \cite{OSS3}.
For $B_1,\dotsc,B_p\in\CC$ one may prove by induction that
\begin{align} \label{E:DbCC}
  \Db_{B_1\otimes \dotsm\otimes B_p}(b) = \sum_{i=1}^p \Db_{B_i}(b_i^{(1)}) +
  \sum_{1\le i<j\le p} \Hb_{B_i,B_j}(b_i \otimes b_j^{(i+1)})
\end{align}
where $b=b_1\otimes \dotsm\otimes b_p$ with $b_i\in B_i$ for $1\le i\le p$
and $b_j^{(k)}$ is the $k$-th tensor factor of the element obtained from $b$
by the composition of combinatorial $R$-matrices that swaps the $j$-th tensor
factor to the $k$-th position. We have
\begin{align}
\label{E:Dbiso}
\Db_B = \Db_{B'} \circ g\qquad\text{for any $g:B\cong B'$ with $B,B'\in\CC$.}
\end{align}

\begin{lemma}\label{L:DBrs} Let $B$ be a KR crystal of level $\ell$. Then
\begin{enumerate}
\item $\Db_B$ is constant on $I_0$-components.
\item $\Db_B(e_0(b)) = \Db_B(b) + 1$ if $\veps_0(b)> \ell$.
\item $\Db_B(u(B))=0$.
\end{enumerate}
\end{lemma}
\begin{proof}
Follows immediately from \eqref{E:DoneKR}, the properties of $\Hb_{B,B}$,
and \eqref{E:tensor}.
\end{proof}

\begin{lemma}
\label{lem_decrease} Let $B_1,B_2\in\CC$, and let $b_1\in B_1$ and $b_2\in B_2$
be such that $e_0(b_1\otimes b_2)\neq 0$ and let $R_{B_1,B_2}(b_1\otimes b_2)=b_2'\otimes b_1'$.
Assume that
\begin{align*}
\Db(e_0(b_1)) &= \Db(b_1)+1&&\text{if $e_0(b_1)\ne0$} \\
\Db(e_0(b_2')) &= \Db(b_2') + 1 &&\text{if $e_0(b_2')\ne0$.}
\end{align*}
Then $\Db(e_0(b_1\otimes b_2))=\Db(b_1\otimes b_2)+1$.
\end{lemma}
\begin{proof} This follows from \eqref{E:intrinsic}, computing
the four cases of \eqref{E:Hb0}.
\end{proof}

We shall prove the following explicit formula for $\Db_{B^{r,s}}$
at the end of Section \ref{SS:sigmaKR}.

\begin{proposition} \label{P:DBrs} For $\geh$ nonexceptional of kind $\di\in\{(1),(2),(1,1)\}$
and $(r,s)\in I_0\times \Z_{>0}$ with $r$ nonspin, we have
\begin{align} \label{E:DBrs}
  \Db_{B^{r,s}}(b(r,s,\la)) = \dfrac{rs - \abs{\la}}{|\di|} \qquad\text{for all $\la\in \Pa^\di_n(r,s)$.}
\end{align}
\end{proposition}

\subsection{Affine highest weight crystals}
\label{SS:affhw}

Let $B(\La)$ be the crystal graph of the irreducible integrable highest
weight module of highest weight $\La\in P^+$. $\hw_I(B(\La))$ is a singleton denoted $u_\La$.
The enhanced weight function $\wtaf:B(\La)\to P$ is defined by $\wtaf(u_\La)=\La$
and \eqref{E:wte} and \eqref{E:wtf} except that
$\alpha_i'\in P'$ is replaced by the affine simple root $\alpha_i\in P$.
Alternatively, let $b\in B(\La)$. Then there is a sequence $\as=(i_1,i_2,\dotsc,i_p)$ of
elements of $I$ such that $u_\La = e_\as(b)$. Define $\widehat{D}(b)$ to be
the number of times that $0$ occurs in the sequence $\as$. This
yields a well-defined $\Z$-grading $\widehat{D}:B(\La)\to\Z$. Then
\begin{align}
  \wtaf(b) = (\ip{d}{\La}-\widehat{D}(b))(\delta/a_0) + \sum_{i\in I}
  (\vphi_i(b)-\veps_i(b))\La_i.
\end{align}


The following Theorem is fundamental to the Kyoto path model
for affine highest weight crystals.

\begin{theorem} \label{T:KMN} \cite[Proposition 2.4.4]{KMN}
Let $\geh$ be an affine algebra,
$B\in\CC(\geh)$ the crystal graph of a $U'_q(\geh)$-module,
and $\La\in P^+$ a dominant weight with $\lev(\La)=\lev(B)$.
Then there is an affine crystal isomorphism
\begin{align} \label{E:KMNiso}
  B(\La) \otimes B \cong
  \bigoplus_u B(\vphi(u))
\end{align}
where $u$ runs over the elements of $B$ such that $\veps(u)=\La$.
\end{theorem}

\subsection{One-dimensional sums and stability}
For $B\in\CC$ and $\la\in P_0^+$, define the one-dimensional sum
\begin{align} \label{E:Xdef}
  \Xb_{\la,B}(q) = \sum_{b\in \hw_{I_0}^\la(B)} q^{\Db(b)}.
\end{align}

\begin{notation} \label{N:BR} Let
\begin{equation}
  B = B^{r_1,s_1} \otimes B^{r_2,s_2} \otimes\dotsm\otimes B^{r_p,s_p}.
\end{equation}
We write $R_i=(s_i^{r_i})$, which is a rectangular partition with $r_i$ rows
and $s_i$ columns. Let $R=(R_1,R_2,\dotsc,R_p)$. We write
$B=B^R$ if we wish to emphasize the indexing set of rectangles.
\end{notation}

For nonexceptional $\geh$, let $n=\rk(\geh_0)$ and define
\begin{align}\label{E:stableBandla}
  \CCinf(\geh) &= \{ B=B^R\in \CC(\geh) \mid \sum_i r_i \le n-2 \} \\
  \Pa_n^\infty &= \{\la\in \Pa_n\mid \ell(\la)\le n-2\} \\
  \CCHinf(\geh_0) &= \{ B\in \CCH(\geh_0) \mid \text{if $B_{I_0}(\nu)$ appears in $B$ then $\nu\in\Pa_n^\infty$.}\}.
\end{align}
These restrictions have the effect of guaranteeing that spin weights
do not appear.

For $\di\in\{(1),(2),(1,1)\}$ and fixed $R$ and $\la$
define the stable 1-d sum $\Xb^\di_{\la,B^R}(q)$ to be
$\Xb_{\la,B^R}(q)$ of type $\geh$ where $\geh$ is chosen such that $n=\rk(\geh_0)$ is
large enough so that $B^R\in \CCinf(\geh)$ and $\la\in \Pa_n^\infty$.
Without loss of generality we may choose $\geh$ to be reversible
(that is, of the form $\geh^\di$; see \eqref{E:gehdi}).

\section{$\geh^\di$, $I_0$, and $A_{n-1}$-crystals}
\label{S:hat}

In this section we assume $\geh$ is one of the reversible affine algebras $\geh^\di$.
Its classical subalgebra $\geh_0^\di$ (see \eqref{E:gehdi}) contains
the subalgebra $\mathfrak{sl}_n$ of type $A_{n-1}$ given by
restricting to the Dynkin node subset $I_{A_{n-1}}=\{1,2,\dotsc,n-1\}$.
Using the notation of Section \ref{S:crystalgen} we write
$B(b):=B_{I_0}(b)$, $B_{A_{n-1}}(b):=B_{I_{A_{n-1}}}(b)$,
and $\hw_{A_{n-1}}(b):=\hw_{I_{A_{n-1}}}(b)$. In fact $\mathfrak{gl}_n\subset \geh_0^\di$
and we use the $\mathfrak{gl}_n$ weights below.

\subsection{Some subcrystals}
\label{SS:subcrystals}

For $\geh_0$ of type $B_n$, $C_n$, or $D_n$
and $B\in\CCHinf(\geh_0)$, define the $I_0$-subcrystal
\begin{align} \label{E:maxdef}
  \max(B) =
\bigcup_{\substack{b\in\hw_{I_0}(B) \\ \abs{\wt(b)}=M(B)}}
B_{I_0}(b)
\end{align}
where $M(B)$ is the maximum value of $\abs{\nu}$ over
$\nu\in\Pa_n$ such that $B_{I_0}(\nu)$ is a component of $B$. Define
\begin{align}
\label{E:topsdef}
  \tops(B) &= \bigcup_{b\in\hw_{I_0}(B)} B_{A_{n-1}}(b).
\end{align}
It is an $A_{n-1}$-subcrystal of $B$ given by taking all the
$A_{n-1}$-components of $I_0$-highest weight vertices in $B$.
These $A_{n-1}$-components sit at the top of their respective $I_0$-components.

\begin{remark} \label{R:topsBnu}
For $\nu\in\Pa_n^\infty$ we have $\tops(B(\nu))\cong B_{A_{n-1}}(\nu)$.
Moreover this is the only $A_{n-1}$-component of $B(\nu)$ of highest
weight $\nu$. Therefore there is a canonical inclusion $i_A^\nu:
B_{A_{n-1}}(\nu) \to B_{I_0}(\nu)$. This isomorphism just says that
a type $A_{n-1}$ tableau can be regarded as an KN tableau for $\geh_0$.
\end{remark}

For $B\in\CCHinf(\geh_0)$, define
\begin{align} \label{E:hatdef}
  \Bh &= \bigcup_{\la\in\Pa_n} \bigcup_{c\in \hw_{A_{n-1}}^{\bla}(B)}
  B_{A_{n-1}}(c).
\end{align}
$\Bh$ is the $A_{n-1}$-subcrystal of $B$ given by the dual polynomial part of $B$ regarded as an $A_{n-1}$-crystal.
The terminology ``dual polynomial part"
makes sense: $\geh_0\supset \mathfrak{gl}_n$ so that $B$ admits a $\mathfrak{gl}_n$ weight function.

For $\nu\in\Pa_n^\infty$, write
\begin{align}\label{E:Bnudef}
\Bh(\nu) := \widehat{B(\nu)}.
\end{align}
It is an $A_{n-1}$-subcrystal of the irreducible highest weight $I_0$-crystal $B(\nu)$.

\subsection{Row tableaux realization of $\Bh(\nu)$}
\label{SS:rowtabs}

This subsection only concerns crystals of types $B_n$, $C_n$, and $D_n$,
and herein we let $\di=(1),(2),(1,1)$ correspond to
$B_n$, $C_n$, and $D_n$ respectively; they coincide with $\geh^\di_0$ but we do not employ
any affine algebra here.

In \cite{KN}, the classical type crystal graph $B(\nu)$ was realized by tableaux which we will call
Kashiwara-Nakashima (KN) tableaux. These tableaux are based on the unique $I_0$-crystal
embedding
\begin{align*}
B(\nu) \hookrightarrow B(\omega_{\nu_1'}) \otimes B(\omega_{\nu_2'}) \otimes\dotsm
\end{align*}
where $\nu_j'$ is the size of the $j$-th column of the partition $\nu$.

However we shall use a different realization of $B(\nu)$ (which we call
``row tableaux") which is better suited for the study of $\Bh(\nu)$.
For $\nu\in\Pa_n^\infty$, there is a unique embedding of $I_0$-crystals
\begin{align}
  \rowtab_\nu: B(\nu) \hookrightarrow B(\nu_1 \omega_1) \otimes \dotsm \otimes B(\nu_p \omega_1)
\end{align}
where $p=\ell(\nu)$. The image of $\rowtab_\nu$ is the connected component
\begin{align*}
  \Image(\rowtab_\nu) = B_{I_0}(1^{\nu_1} \otimes 2^{\nu_2} \otimes \cdots \otimes p^{\nu_p}).
\end{align*}
Here $a^m$ denotes the word consisting of $m$ copies of the symbol $a$.
The image of $b\in B(\nu)$ is a tensor product
$\rowtab(b)=R_{1}\otimes R_{2}\otimes \cdots \otimes R_{p}$ with $R_i\in B(\nu_i\omega_1)$; it is called
the \textit{row tableau} associated with the element $b\in B(\nu)$ and may be depicted as a tableau of shape
$\nu$ whose $i$-th row is $R_i$. Each $R_i$ is a KN tableau of the single-row shape $(\nu_i)$.
In general $\rowtab(b)$ does not coincide with the corresponding KN tableau of shape $\nu$.
We are not aware of a simple characterization of the image of $\rowtab_\nu$.
Nevertheless we characterize the image of $\Bh(\nu)$
under $\rowtab_\nu$.

For a tableau $c$ of shape $\nu$ and $D\subset\nu$ a skew shape, let $c|^D$ denote the
restriction of $c$ to the subshape $D$.

For $\delta\in\Pa_n^\di$ with $\delta\subset\nu$, let
$\Bhp^\di(\nu,\delta) \subset B(\nu_1\omega_1) \otimes\dotsm\otimes B(\nu_p\omega_1)$
be the set of vertices $b=R_1\otimes \dotsm \otimes R_p$ such that:
\begin{enumerate}
\item  $b|^{\nu\setminus\delta}$ is a skew semistandard tableau on $\{\overline{n},\dotsc,\overline{1}\}$.
\item $b|^\delta = C_\delta^\di$, where the latter tableau is the unique tableau such that:
\begin{itemize}
\item For $\di=(1)$, the $i$-th row equals $n^a0^{\delta_i-2a}\bar{n}^a$ where $a=\lfloor \delta_i/2 \rfloor$.
\item For $\di=(2)$, the $i$-th row equals $n^a \bar{n}^a$ where $a=\delta_i/2$.
\item For $\di=(1,1)$, the $j$-th column consists of $\delta'_j/2$ copies of
$\vcenter{\tableau[sby]{n \\ \overline{n}}}$ .
\end{itemize}
\end{enumerate}

Let $\Bhp^\di(\nu) = \bigcup_{\delta\in \Pa_n^\di} \Bhp^\di(\nu,\delta)$.

\begin{example}
For $\di=(1,1)$, $\nu =(4,4,4,2,1,1)$ and $\delta=(3,3,1,1)$,
\begin{equation*}
\begin{tabular}{|c|ccc}
\cline{1-1}
\lower.5mm\hbox{$\overline{n-3}$} &  &  &  \\ \cline{1-1}
\lower.5mm\hbox{$\overline{n-1}$} &  &  &  \\ \cline{1-2}
$\mathbf{n}$ & \lower.5mm\hbox{$\overline{n-1}$} & \multicolumn{1}{|c}{} &  \\ \hline
$\overline{\mathbf{n}}$ & $\overline{n}$ & \multicolumn{1}{|c}{$\overline{n}$
} & \multicolumn{1}{|c|}{\lower.3mm\hbox{$\overline{n-2}$}} \\ \hline
$\mathbf{n}$ & $\mathbf{n}$ & \multicolumn{1}{|c}{$\mathbf{n}$} &
\multicolumn{1}{|c|}{\lower.5mm\hbox{$\overline{n-1}$}} \\ \hline
$\overline{\mathbf{n}}$ & $\overline{\mathbf{n}}$ & \multicolumn{1}{|c}{$
\overline{\mathbf{n}}$} & \multicolumn{1}{|c|}{$\overline{n}$} \\ \hline
\end{tabular}
\end{equation*}
is a row tableau in $\hw_{A_{n-1}}(\Bhp^\di(\nu,\delta))$.
\end{example}

\begin{remark} \label{R:Bnu'} \
\begin{enumerate}
\item Given any $b\in \Bhp^\di(\nu)$, the unbarred letters in $b$ determine the unique
$\delta\in\Pa_n^\di$ such that $b\in \Bhp^\di(\nu,\delta)$, and $b$ is determined
by $\delta$ and $b|^{\nu\setminus\delta}$.
By definition $b\in \Bhp^\di(\nu)$ contains no letters in $\{1,\dotsc,n-1\}.$
\item Let $b^\nu$ be the lowest weight vector of $\Bhp^\di(\nu)$.
Then $\rowtab(b^\nu)\in\Bhp^\di(\nu,\emptyset)$ where $\delta=\emptyset$ is the empty partition.
\end{enumerate}
\end{remark}

\begin{proposition}
\label{P:rowtabhatiso} The map $\rowtab_\nu$ restricts to an isomorphism
\begin{align} \label{E:rowtabhatiso}
\Bh(\nu) \cong \Bhp^\di(\nu).
\end{align}
\end{proposition}

Proposition \ref{P:rowtabhatiso} will be deduced from Proposition \ref{P:Bhatprop} below.

The \textit{reading word} of a single-rowed tableau is obtained by reading its letters
from \textit{right to left}. The reading word of a tableau obtained by reading its rows
from top to bottom. A word $w=x_1x_2\dotsm x_\ell$ with $x_i\in \{\bar{n},\overline{n-1},\dotsc,\bar{1}\}$
is \textit{Yamanouchi} if for all $j$, in the subword $x_1x_2\dotsm x_j$ there at least as many letters
$\overline{i+1}$ as there are letters $\bar{i}$ for $1\le i\le n-1$.

\begin{proposition}
\label{P:Bhatprop} Let $b\in \Bhp^\di(\nu,\delta)$ for some $\delta\in\Pa_n^\di$.
\begin{enumerate}
\item $\Bhp^\di(\nu,\delta)$ is an $A_{n-1}$-crystal.
\item $b\in\hw_{A_{n-1}}(\Bhp^\di(\nu,\delta))$ if and only if the row-reading word of the
skew semistandard subtableau of $b$ of shape $\nu/\delta$, is Yamanouchi
of weight $\bla$ for some $\la\in\Pa_n$.
\item If $\vphi_n(b)>0$ then $f_n(b)\in\Bhp^\di(\nu)$.
\item There exists a finite sequence $\mba=(j_1,j_2,\dotsc)$ in $I_0$
such that $b=e_\mba(\rowtab(b^\nu))$. In particular $\Bhp^\di(\nu)\subset \Image(\rowtab_\nu)$.
\item  Assume $b\in\hw_{A_{n-1}}^{\bla}(\Bhp^\di(\nu,\delta))$ for some $\la\in\Pa_n$
and let $\mba$ be as above. Then
\begin{align}\label{E:ncard}
\mathrm{card}\, \{k\mid
j_k=n\}=\frac{\abs{\nu}-\abs{\la}}{|\di|}.
\end{align}
\end{enumerate}
\end{proposition}

The proof of Proposition \ref{P:Bhatprop} is deferred to Appendix \ref{A:hatcrystal}.

\begin{proof}[Proof of Proposition \ref{P:rowtabhatiso}]
$\rowtab_\nu(\Bh(\nu))$ and $\Bhp^\di(\nu)$ are both $A_{n-1}$-subcrystals
of $\Image(\rowtab_\nu)$, by definition and Proposition \ref{P:Bhatprop}(4) respectively.
Therefore it suffices to show they have the same $A_{n-1}$-highest weight vertices.
All such vertices have weight of the form $\bla$ for some $\la\in\Pa_n$.
For $\la\in\Pa_n$ and $\delta\in\Pa_n^\di$,
$|\hw_{A_{n-1}}^{\bla}(\Bhp^\di(\nu,\delta))| = c^\nu_{\delta\la}$
by Proposition \ref{P:Bhatprop}(2) and the Littlewood-Richardson Rule \cite{F}. All of these highest
weight vertices are in $\rowtab_\nu(\hw_{A_{n-1}}^{\bla}(\Bh^\di(\nu)))$. The result follows by
summing over $\delta\in\Pa_n^\di$ and using \eqref{E:Lit}.
\end{proof}

\subsection{$\Bh(\nu)$ when $\nu$ is a rectangle}
We assume $\geh=\geh^\di$ is reversible, and
apply the previous results to $\max(B^{r,s})\cong B(s\omega_r)$ for $B^{r,s}\in\CCinf(\geh^\di)$.
For the rectangular partition $\nu =(s^r)\in\Pa_n^\infty$ let
\begin{align}
\label{E:bamin}
  \ba(r,s,\la) &= \hw_{A_{n-1}}^{\bla}(B_{I_0}(s^r)) \qquad\text{for $\la\in\Pa_n^\di(r,s)$} \\
\label{E:bminrs}
  \bamin^\di(r,s) &= \ba(r,s,\lamin^\di(r,s))
\end{align}
where $\lamin^\di(r,s)$ is defined in \eqref{E:lamin}. Note that the set on the right hand
side of \eqref{E:bamin} is a singleton, by \eqref{E:Lit} and the Littlewood-Richardson
Rule. We regard the elements $\ba(r,s,\la)$ as being in $B^{r,s}$ since $B^{r,s}$ contains
a unique $I_0$-component $B_{I_0}(s^r)$. We note that
\begin{align}
  \hw_{A_{n-1}}(\Bh(s^r)) = \{ \ba(r,s,\la)\mid \la\in\Pa^\di_n(r,s) \}.
\end{align}

\begin{remark} \label{R:rowtabmins} For $\la\in\Pa_n^\di(r,s)$, let $\delta\in\Pa_n^\di$
be the partition complementary to $\la$ in the rectangle $(s^r)$. Then
by Propositions \ref{P:rowtabhatiso} and \ref{P:Bhatprop}(2),
$\rowtab_{(s^r)}(\ba(r,s,\la))$ is explicitly
given by the row tableau of shape $(s^r)$ whose restriction to the shape $\delta$,
is the canonical tableau $C_\delta^\di$
and whose restriction to $(s^r)/\delta$ is the unique Yamanouchi tableau of that
shape in the letters $\{\overline{n},\dotsc,\bar{2},\bar{1}\}$; each column of the latter
subtableau consists of letters $\bar{n}$, $\overline{n-1}$, etc., reading from bottom to top.
\end{remark}

For $\nu=(s^r)$ we are going to see that every $A_{n-1}$-highest weight vertex in
$\Bh(\nu)$ is reachable by $I_0$-lowering operators, starting with a certain fixed element.
This is not true for a general partition $\nu\in \Pa_n^\infty$.


\begin{proposition}
\label{P:Brsdual} Let $(r,s)\in I_0\times \Z_{>0}$ with $B^{r,s}\in\CCinf(\geh^\di)$
and $\ell=\lev(B^{r,s})$. Then for any $\la\in \Pa^\di_n(r,s)$
there exists a finite sequence $\bs=(j_1,j_2,\dotsc)$ in $I_0$ such that
\begin{align}
\label{E:mintola}
u_{\ell\La_n} \otimes \ba(r,s,\la) &= f_{\bs}(u_{\ell\La_n} \otimes \bamin^\di(r,s)) \\
\mathrm{card}\{k \mid j_k = n\} &=
\frac{\abs{\la}-\abs{\lamin^\di(r,s)}}{|\di|}.
\end{align}
\end{proposition}

This result follows by induction using Lemma \ref{L:npath} below.
For $h\ge2$ if $\di=(1,1)$ and $h\ge1$ if $\di\in\{(1),(2)\}$, define the following
sequences (the semicolons are just for readability):
\begin{align}
\notag
\ta'(h) &=
\begin{cases}
(n-2,n-3,\dotsc,n-h+1;n-1,n-2,\dotsc,n-h+2) & \text{for $\di=(1,1)$} \\
(n-1,n-2,\dotsc,n-h+1) & \text{for $\di=(1)$} \\
((n-1)^2,(n-2)^2,\dotsc,(n-h+1)^2) & \text{for $\di=(2)$}
\end{cases} \\
\label{E:ta}
\ta(h) &= (n;\ta'(h)).
\end{align}

\begin{notation}\label{N:lambdaminus}
Given $\la\in\Pa_n^\di(r,s)$ with $\la\ne\lamin=\lamin^\di(r,s)$, we define a canonical smaller element
$\lam\in\Pa_n^\di(r,s)$ obtained from $\la$ by removing a particular copy of the shape $\di$.
Suppose the rightmost column in which $\la$ and $\lamin$ differ, is the $p$-th.
Let $h=\la'_p$ be the height of that column. Let $\lam\in \Pa_n^\di(r,s)$ be obtained
from $\la$ by removing a vertical domino from the $p$-th column if $\di=(1,1)$,
removing a cell from the $p$-th column if $\di=(1)$, and removing a cell from the
$p$-th and $(p-1)$-th columns if $\di=(2)$.

We note that if $\delta\in\Pa_n^\di$ is a nonempty partition then
$\delta^-\in\Pa_n^\di$ can be defined similarly.
\end{notation}

\begin{lemma} \label{L:npath}
Let $\la\in\Pa_n^\di(r,s)$ with $\la\ne\lamin^\di(r,s)$. Then
\begin{align} \label{E:nnextla}
u_{\ell\La_n} \otimes \ba(r,s,\la) &= f_{\ta(h)}(u_{\ell\La_n} \otimes \ba(r,s,\lam)).
\end{align}
\end{lemma}

The proof of Lemma \ref{L:npath} is deferred to Appendix \ref{A:hatcrystal}.

\section{Affine crystals and the involution $\sigma$}
\label{S:sigmasec}

In this section we summarize necessary facts on a single KR crystal $B^{r,s}$ belonging to
$\CCinf(\geh^\di)$ and show that a tensor product $B$ of such KR crystals has an automorphism $\sigma$,
which we call the reversing crystal automorphism. This $\sigma$ will be effectively used to show our
main theorem (Theorem \ref{Th_dec_X}).

\subsection{KR crystal $B^{r,s}$}
\label{SS:KR}

We consider a single KR crystal $B^{r,s}\in\CCinf(\geh^\di)$. Note that $r\in I_0$ is nonspin. We recall
the crystal structure of $B^{r,s}$. Firstly, the $U_q(\geh^\di_0)$-crystal structure is described as
follows. As we explained in Introduction, $B^{r,s}$ decomposes into a multiplicity-free direct sum of highest weight
crystals $B(\la)$, where $\la$ runs over $\Pa^\di_n(r,s)$, the set of partitions
obtained by removing $\di$'s from $(s^r)$. The action of Kashiwara operators $e_i,f_i$ ($i\in I_0$) on
$B^{r,s}$ is given by realizing its elements by KN tableaux. Hence, we are left to describe the action of $e_0$
and $f_0$. To do this we explain the notion of $\pm$-diagrams and a certain automorphism $\varsigma$ on $B^{r,s}$
for $\di=(1,1)$ introduced in \cite{Sc}. From here to Lemma \ref{lem:e1 action} we assume $\di=(1,1)$.

A $\pm$-diagram $P$ of shape $\La/\la$ is a sequence of partitions $\la\subset \mu \subset \La$
such that $\La/\mu$ and $\mu/\la$ are horizontal strips (i.e. every column contains at most one box). We
depict this $\pm$-diagram by the skew tableau of shape $\La/\la$ in
which the cells of $\mu/\la$ are filled with the symbol $+$ and
those of $\La/\mu$ are filled with the symbol $-$. Write
$\La=\mathrm{outer}(P)$ and $\la=\mathrm{inner}(P)$ for the outer and inner shapes of the
$\pm$-diagram $P$. We call $\mu$ the middle shape. Set $J=\{2,3,\ldots,n\}$. There is a bijection
$\Phi:P\mapsto b$ from $\pm$-diagrams $P$ of shape $\La/\la$ to the set of $J$-highest weight elements
$b$ of $J$-weight $\la$. For details refer to section 4.2 of \cite{Sc}.

Now suppose $b\in B^{r,s}$ is a $J$-highest weight element corresponding to a $\pm$-diagram $P$
of shape $\La/\la$. Let $c_i=c_i(\la)$ be
the number of columns of height $i$ in $\la$ for all $1\le i<r$ with
$c_0=s-\la_1$. If $i\equiv r-1 \pmod{2}$, then in $P$, above each
column of $\la$ of height $i$, there must be a $+$ or a $-$.
Interchange the number of such $+$ and $-$ symbols. If $i\equiv r
\pmod{2}$, then in $P$, above each column of $\la$ of height $i$,
either there are no signs or a $\mp$ pair. Suppose there are $p_i$
$\mp$ pairs above the columns of height $i$. Change this to
$(c_i-p_i)$ $\mp$ pairs. The result is $\mathfrak{S}(P)$, which has the
same inner shape $\la$ as $P$ but a possibly different outer shape.
The columns of height $r$ in $P$ are not changed by $\mathfrak{S}$.
The following map $\varsigma$ (called $\sigma$ in \cite{Sc}) is an automorphism on $B^{r,s}$ corresponding to
interchanging the nodes $0$ and $1$ of the Dynkin diagram of $D^{(1)}_n$.

\begin{definition} \label{def:tau}
Let $b\in B^{r,s}$ and $\as$ be a sequence of elements of $J$ such that
$e_{\as}(b)$ is a $J$-highest weight element. Let $\mathbf{a'}$ be the reverse sequence of $\as$. Then
\begin{equation} \label{eq:def tau}
\varsigma(b) := f_{\mathbf{a'}} \circ \Phi \circ \mathfrak{S} \circ \Phi^{-1} \circ e_{\as}(b).
\end{equation}
\end{definition}
With this $\varsigma$ the Kashiwara operators $e_{0}$ and $f_{0}$ are given by
\begin{equation} \label{eq:e0}
\begin{split}
f_{0} &= \varsigma \circ f_{1} \circ \varsigma,\\
e_{0} &= \varsigma \circ e_{1} \circ \varsigma.
\end{split}
\end{equation}
By \eqref{eq:def tau} and \eqref{eq:e0} $e_0$ and $f_0$ commutes with $e_i$ or $f_i$ for
$J'=\{3,4,\ldots,n\}$. Hence, the calculation of the actions of $e_0$ and $f_0$ are reduced to
$J'$-highest weight elements. Note that $J'$-highest weight elements are in one-to-one correspondence
with pairs of $\pm$-diagrams $(P,p)$, where the inner shape of $P$ is the outer shape of $p$. To calculate
the action of $e_0$ it suffices to know the action of $e_1$ on $(P,p)$, that is described in \cite{Sc}.
\begin{enumerate}
\item Successively run through all $+$ in $p$ from left to right and, if possible, pair it with
the leftmost yet unpaired $+$ in $P$ weakly to the left of it.
\item Successively run through all $-$ in $p$ from left to right and, if possible, pair it with
the rightmost  yet unpaired $-$ in $P$ weakly to the left.
\item Successively run through all yet unpaired $+$ in $p$ from left to right and, if possible,
pair it with the leftmost yet unpaired $-$ in $p$.
\end{enumerate}
\begin{lemma} \cite[Lemma 5.1]{Sc} \label{lem:e1 action}
If there is an unpaired $+$ in $p$,  $e_1$ moves the rightmost unpaired $+$ in $p$ to $P$.
Else, if there is an unpaired $-$ in $P$, $e_1$ moves the leftmost unpaired $-$ in $P$ to $p$.
Else $e_1$ annihilates $(P,p)$.
\end{lemma}

For types $\di=(2),(1)$, we use a construction of $B^{r,s}$ in section 4.3 and 4.4 of \cite{FOS1} (where
it is called $V^{r,s}$). As above we can assume $b\in B^{r,s}$ is $J$-highest. Let $p=\Phi^{-1}(b)$ and
let $\hat{p}$ be $p$ itself if $\di=(2)$, and the $\pm$-diagram whose inner, middle and outer shapes
are all doubled rowwise if $\di=(1)$. Let $c_i$ ($1\le i\le r$) be the number of columns of height $i$
in $\mathrm{outer}(\hat{p})$. We also set $c_0=\gamma s-\mathrm{outer}(\hat{p})_1$ where
$\gamma=2/|\di|$. Note that $c_i$ is even except when $\di=(2),i=r$ and $r$ is odd. There exists a unique
$\pm$-diagram $P$ such that $\mathrm{inner}(P)=\mathrm{outer}(\hat{p})$, the length of
$\mathrm{inner}(P)\le r$ and there are equal number $c_i/2$ of columns with $\mp$ and $\cdot$ in $P$
if $i<r,i\equiv r\;(2)$, with $+$ and $-$ if $i\not\equiv\;(2)$. Then the pair of $\pm$-diagrams
$(P,\hat{p})$ can be considered to correspond to a $\{3,4,\ldots,n\}$-highest element of $B^{r,\gamma s}$
of type $\di=(1,1)$. We now apply $e_1\circ\varsigma\circ e_1$ to $(P,\hat{p})$ following the procedure explained
previously to get $(P',\hat{p}')$. Let $p'$ be $\hat{p}'$ if $\di=(2)$, and the $\pm$-diagram whose inner,
middle and outer shapes are all halved rowwise. (This is possible by Lemma 4.7 (1) in \cite{FOS1}.)
Finally, setting $b'=\Phi(p')$ we obtain $e_0b=b'$. To calculate the action of $f_0$ we replace
$e_1\circ\varsigma\circ e_1$ with $f_1\circ\varsigma\circ f_1$.

\subsection{The reversing crystal automorphism $\sigma$}
\label{SS:sigma}

Recall $\sigma\in\Aut(X)$ from \eqref{E:sigmaaut}.

\begin{theorem} \label{T:sigma} For every $B$ that is a tensor product of
KR crystals in $\CCinf(\geh^\di)$, there is a unique map
$\sigma=\sigma_B:B\to B$ such that
\begin{align} \label{E:sigmaprop}
  \sigma \circ e_i = e_{\sigma(i)} \circ \sigma
\end{align}
for all $i\in I$ and $b\in B$. Moreover
\begin{align}
\label{E:sigmawt}
  \wt(\sigma(b)) &= -w_0^{A_{n-1}}(\wt(b)) \\
\label{E:sigmainv}
  \sigma^2 &= \text{id} \\
\label{E:sigmaiso}
  \sigma_{B'} \circ g &= g \circ\sigma_B\qquad\text{for any $g:B\cong B'$ for $B,B'\in\CC$.}
\end{align}
Here $w_0^{A_{n-1}}\in W$ is the longest element of the type $A_{n-1}$ Weyl group
generated by $s_1$ through $s_{n-1}$.
\end{theorem}

First we assume the existence of $\sigma$ satisfying
\eqref{E:sigmaprop} and deduce \eqref{E:sigmawt}, \eqref{E:sigmainv},
and \eqref{E:sigmaiso}.

For \eqref{E:sigmawt} we recall the discussion of the weight function on
KR crystals (and therefore on $B$) in Section \ref{SS:KRgen}
and associated notation. By \eqref{E:sigmaprop} and \eqref{E:wtdef}
we have $\sigma(\wt(b)) = \wt(\sigma(b))$, computing in the lattice
$P'$. Now $\wt$ takes values in $P^0\cong P$ and one may
check that the action of $\sigma$ on $P^0$ agrees with that of
$-w_0^{A_{n-1}}$ on $P$.

For \eqref{E:sigmainv}, $\sigma^2$ is an $I$-crystal isomorphism $B\to B$.
By connectedness and the fact that $B$ contains a unique element
$u(B)$ of its weight, there is only one such isomorphism, namely,
the identity.

For \eqref{E:sigmaiso}, by the connectedness of $B$ the proof reduces to
verifying the relation for a single value. However the value of both
sides on $u(B)$ must agree, for the answer must be the unique
element of $B'$ whose weight is $-w_0^{A_{n-1}}(\wt(u(B)))$.

Next, we prove the uniqueness of $\sigma$ assuming its existence.
Since $B\in\CC$ is connected we need only show that
\eqref{E:sigmaprop} uniquely specifies some single value of $\sigma$.
The vertex $u(B)$ is the only element of its
weight in $B$. The weight $w_0^{A_{n-1}}(u(B))$ occurs in $B$ since
$B$ is an $A_{n-1}$-crystal. Since $B$ is an $I_0$-crystal (of classical type $B_n$, $C_n$,
or $D_n$) with no
spin weight, it is self-dual, so its weights are closed under
negation. In particular the weight $-w_0^{A_{n-1}}(u(B))$ must also
occur in $B$. Since $\wt(u(B))$ occurs exactly once, the weight
$-w_0^{A_{n-1}}(\wt(u(B)))$ also occurs exactly once. By
\eqref{E:sigmawt} $\sigma(u(B))$ must be the unique element of $B$
of weight $-w_0^{A_{n-1}}(\wt(u(B)))$. It follows that $\sigma$ is unique.

It only remains to prove the existence of $\sigma$. By \eqref{E:tensor}
we may reduce to the case $B=B^{r,s}$. The existence of $\sigma$ on $B^{r,s}$
is proved in the next several subsections.

\subsection{Definition of $\sigma$ on KR crystals}
\label{SS:sigmaKR}

Define the sequences
\begin{equation}\label{E:as}
\begin{split}
\as'(h) &=
\begin{cases}
(2,3,\dotsc,h-1;1,2,\dotsc,h-2) &\text{if $\di=(1,1)$} \\
(1,2,\dotsc,h-1) & \text{if $\di=(1)$} \\
(1^2,2^2,\dotsc,(h-1)^2) &\text{if $\di=(2)$}
\end{cases} \\
\as(h) &=(0;\as'(h)).
\end{split}
\end{equation}
Recalling $\ta(h)$ from \eqref{E:ta} we have
\begin{equation} \label{E:asigma}
\sigma(\as(h)) = \ta(h).
\end{equation}

\begin{lemma}\label{L:nextla}
Let $\la\in\Pa(r,s)$ and $\la\ne\lamin^\di(r,s)$. Let $\ell=\lev(B^{r,s})$ and $\lam$ be as in
Notation \ref{N:lambdaminus}. Then
\begin{equation} \label{E:nextla}
u_{\ell\La_0} \otimes b(r,s,\la) = f_{\as(h)}(u_{\ell\La_0} \otimes
b(r,s,\lam)).
\end{equation}
\end{lemma}

\begin{proof}
We first treat the case $\di=(1,1)$. Suppose $r$ is even.
We apply $f_{\as'(h)}$. Then $b(r,s,\lam)$ changes to the
KN tableau $t_1$ of shape $\lam$ whose columns are filled with
$123\cdots$, except the rightmost, which is filled with $34\cdots$
instead. Now we want to apply $f_0$ to $u_{s\Lambda _{0}}\otimes
t_1$. To do this we first go to the $J$-highest element
$e_{(h-1,\dotsc,3,2)}(t_1)$ of $t_1$, where we have set $J=\{2,3,\ldots,n\}$. Then we have $P=\Phi
^{-1}(e_{(h-1,\dotsc,3,2)}(t_1))$ is the $\pm $-diagram such that
there is no sign in the rightmost column and only $+$ in the other
ones. Hence $\mathfrak{S}(P)$ is the $\pm $-diagram described as
follows. Denote the position of the rightmost column of $\la$ by $a$.
The height of the outer shape from the 1st to the $(a-1)$-th column is
the same as $P$, but from the $a$-th to the
$s$-th column the height is larger than $P$ by 2. There is only $-$
from the 1st to the $(a-1)$-th column, and $\mp$ from the $a$-th
to the $s$-th column. Now we have $\varsigma(t_1)=f_{(2,3,\dotsc,h-1)}
\Phi (\mathfrak{S}(P))$ described as
follows. The shape of $\varsigma(t_1)$ is the same as the outer shape
of $\mathfrak{S}(P)$. To get contents we first place the string $
23\cdots k\bar{1}$ in each column and then reading from left to
right, top to bottom we change $\bar{1}$ to $\bar{2}$ and $2$ to $1$
$(s-a+1)$ times. Note that $\veps_1(\varsigma
(t_1))=s+a-1$. One finds $f_1\varsigma(t_1)$ is a $J$-highest
element corresponding to the $\pm $-diagram that differs from
$\mathfrak{S}(P)$ only in the $a$-th column where there is only
$-$. Hence we have $f_0 t_1=b(r,s,\la)$ by definition. Since
$\veps_0(t_1)=s+a-1\ge s$, we also have
$f_0(u_{s\La_0}\otimes t_1)=u_{s\La_0}\otimes
f_0 t_1 = u_{s\La_0}\otimes b(r,s,\la)$.

Next suppose $r$ is odd. In this case the first row of $\la$ has $s$
nodes. Denote the position of the rightmost column with height greater than 1 by $a$. The
calculation goes similarly to the $r$ even case. The $\pm$-diagram $P$
is given as follows. The outer shape is the same as $\lam$.
There is no sign in the $a$-th column and only $+$ in the
other columns. Applying $f_{(2,3,\dotsc,h-1)} \circ \Phi \circ \mathfrak{S}$,
one obtains $\varsigma(t_1)$ described as follows. The shape of $\varsigma (t_1)$
is the same as $t_1$ except in the $a$-th column where the height of $\varsigma (t_1)$
is larger than that of $t_1$ by 2. To get contents we place
the string $23\cdots k\bar{1}$ ($\bar{1}$ in the column of height 1) in each
column. Only in the leftmost column we put $\bar{2}$ instead of $\bar{1}$.
Note that $\veps_1(\varsigma(t_1))=s+a-1$. We obtain $f_0 t_1 = b(r,s,\la)$,
and since $\veps _0(t_1)\ge s$, we again
have $f_0(u_{s\La_0}\otimes t_1)=u_{s\La_0}\otimes b(r,s,\la)$.

Next we treat the case $\di=(2)$. (Since the case $\di=(1)$ is similar, we omit its proof.)
Applying $f_{\mathbf{a'}(h)}$ makes $b(r,s,\la^-)$ change to the KN tableau $t_2$ of shape
$\la^-$ whose columns are filled with $123\cdots$, except the rightmost two, which is filled
with $23\cdots$ instead. Note that $t_2$ is $J$-highest. $p=\Phi^{-1}(t_2)$ is the $\pm$-diagram
such that there is no sign in the rightmost two columns and only $+$ in the other ones. From
this $p$ construct $P$ as prescribed in the previous subsection. We want to apply
$f_1\circ\varsigma\circ f_1$ to this pair $(P,p)$ of $\pm$-diagrams. Denote the position of
the rightmost column of $\la$ by $a$. By Lemma \ref{lem:e1 action} the application of $f_1$
changes $(P,p)$ as follows. In the $(a-1)$-th column there is $+$ (resp. $\mp$) when
$h\equiv r\;(2)$ (resp. $h\not\equiv r\;(2)$) in $P$ and no sign in $p$. $f_1$ moves $+$ in
$P$ to $p$. Denote this new pair by $(P',p')$. Next $\varsigma$ changes $P'$ as follows. In the
columns of $P'$ of height $h$, the number of columns with $\mp$ (resp. $+$) increases by 1
while the number of those with $\cdot$ (resp. $-$) decreases by 1 when $h\equiv r\;(2)$
(resp. $h\not\equiv r\;(2)$). By applying $f_1$ again, we obtain $(P'',p'')$ described as follows.
$p''$ differs from $p$ only at the $(a-1)$-th and $a$-th positions. $\mathrm{outer}(p'')$ is of
height $h$ there with $+$'s. $P''$ is a unique $\pm$-diagram determined from $p''$ as in
the previous subsection. To show \eqref{E:nextla} we still need to check
$\veps_0(b(r,s,\la^-))\ge\ell$. Since the application of $e_0(=e_1\circ\varsigma\circ e_1)$ is similar
to above, we only give its value. Let $c_i$ ($1\le i\le r$) be the number of columns of $\la$ of
height $i$ and set $c_0=s-\la_1$. Then we have
\[
\veps_0(b(r,s,\la^-))=c_r+c_{r-1}+\cdots+c_h-1+c_0/2.
\]
Noting that $(c_r+c_{r-1}+\cdots+c_h+c_0+\bar{r})/2=\ell$
($\bar{r}=0\mbox{ or }1,\bar{r}\equiv r\;(2)$) and $c_h\ge2$, we obtain $\veps_0(b(r,s,\la^-))
\ge\ell$.
\end{proof}

For a KR crystal $B$ of level $\ell$,
say that the $i$-arrow $b\to b'=f_i(b)$ is \textit{good} if either $i\in I_0$ or
$i=0$ and $\veps_0(b)\ge\ell$.
Traversing the above edge backwards (using a raising operator), going from $b'$ to $e_i(b')=b$ is good
if $i\in I_0$ or $i=0$ and $\veps_0(b')>\ell$.

\begin{lemma} \label{L:goodarrows}
Let $B^{r,s}$ be a KR crystal of level $\ell$.
Then for every $b\in B^{r,s}$ there is a sequence of good arrows from $b$ to $m(B^{r,s})$.
\end{lemma}

\begin{proof}
Noting that from \eqref{E:mdef} $u_{\ell\La_0}\ot m(B^{r,s})$ is an affine highest weight vector in
$B(\ell\La_0)\ot B^{r,s}\simeq B(\vphi(m(B^{r,s})))$, the lemma is clear from the previous one.
\end{proof}

We obtain the following for KR crystals $B^{r,s}$ for $\geh$ of kind $(1,1),(2),(1)$
where $r\in I_0$ is nonspin.

\begin{corollary} \label{C:DBrs}
For a KR crystal $B$ of level $\ell$, there is a unique function
$\Db_B$ satisfying the conditions of Lemma \ref{L:DBrs}. Moreover,
identifying elements of $u_{\ell\La_0} \otimes B$ with their images in
$B(\vphi(m(B)))$ under the isomorphism \eqref{E:KMNiso}, we have
\begin{align} \label{E:DhwDBrs}
  \widehat{D}_{B(\vphi(m(B)))}(u_{\ell\La_0} \otimes b) + \Db_B(b) = \Db_B(m(B))
\end{align}
where $m(B)$ is defined in Lemma \ref{L:phimin}.
\end{corollary}
\begin{proof} By Lemma \ref{L:goodarrows} $B$ is connected by good arrows.
But properties (1) and (2) of Lemma \ref{L:DBrs} specify how
$\Db_B$ must change across good arrows. Therefore a single value completely specifies
$\Db_B$. This is furnished by property (3) of Lemma \ref{L:DBrs}.
The left hand side of \eqref{E:DhwDBrs}, viewed as a function
of $b\in B$, is invariant under good arrows in $B$. But $B$ is connected by good arrows
so this function is constant, and its value is obtained
by setting $b=m(B)$ and using that $\widehat{D}=0$ on the affine highest weight vector.
\end{proof}

Let $\ell=\lev(B^{r,s})$ and let $u\in B^{r,s}$ be as in Lemma \ref{L:phimin} using $j=0$.
From Theorem \ref{T:KMN} there are bijections
\begin{equation}  \label{sigma}
B(\ell\Lambda_0)\otimes B^{r,s}\cong B(\vphi(u)) \longrightarrow B(\sigma(\vphi(u))) \cong
B(\ell\Lambda_n) \otimes B^{r,s}.
\end{equation}
The first and third maps are isomorphisms given by Theorem \ref{T:KMN}
and the middle maps are the unique automorphism in highest weight crystals
induced by relabeling everything according to $\sigma\in\Aut(X)$.
\begin{lemma} \label{L:sigma}
Let $\ba(r,s,\la)$ be as in \eqref{E:bamin}.
For $\la\in\Pa^\di(r,s)$, $u_{\ell\La_0}\ot b(r,s,\la)$ is sent to $u_{\ell\La_n}\ot \ba(r,s,\la)$
under the previous bijection.
\end{lemma}

\begin{proof}
The proof proceeds by induction on $\Pa^\di(r,s)$.
The claim holds for $\lamin(=\lamin^\di(r,s))$ since these elements are the unique affine highest 
weight elements of both sides of \eqref{sigma}.
For $\la\in \Pa^\di(r,s)$ with $\la\ne\lamin$ the claim follows from Lemmas \ref{L:npath}, \ref{L:nextla},
\eqref{E:asigma} and induction.
\end{proof}

\begin{proposition} \label{P:sigma} For $B^{r,s}\in\CCinf(\geh^\di)$
there is a unique map $\sigma:B^{r,s}\to B^{r,s}$ such that
\begin{enumerate}
\item Equation \eqref{E:sigmaprop} holds for good arrows.
\item $\sigma(m(B^{r,s}))=\bamin^\di(r,s)$.
\end{enumerate}
\end{proposition}
\begin{proof}
Such a map $\sigma$ is necessarily unique. Assertion 2 specifies one value of $\sigma$.
By Lemma \ref{L:goodarrows} $B^{r,s}$ is connected by good arrows, so
Assertion 1 determines all other values of $\sigma$.
So it suffices to prove existence. Consider the bijection \eqref{sigma}. For an element $b\in B^{r,s}$
the image of $u_{\ell\La_0}\ot b$ by the bijection should belong to $u_{\ell\La_n}\ot B^{r,s}$ by
Lemma \ref{L:sigma}. Denote this image by $u_{\ell\La_n}\ot\sigma(b)$. This map $\sigma$ satisfies
the two conditions.
\end{proof}

\begin{proposition} \label{P:KRsigmaall}
The map $\sigma$ of Proposition \ref{P:sigma} satisfies
\eqref{E:sigmaprop} for all $i\in I$ and $b\in B^{r,s}$.
\end{proposition}

The proof of Proposition \ref{P:KRsigmaall} for $\di=(1,1)$ is deferred to Appendix \ref{A:sigma}.
For $\di=(1),(2)$ the map $\sigma$ constructed in Theorem 7.1 of \cite{FOS1} is the one we need.

\begin{proof}[Proof of Theorem \ref{T:sigma}]
As noted at the end of subsection \ref{SS:sigma},
it suffices to establish the case of a single KR crystal.
The map $\sigma$ in Proposition \ref{P:sigma} works by Proposition \ref{P:KRsigmaall}.
\end{proof}

The following Lemma is used later.

\begin{lemma}
\label{Lem_nice-path} For any $\la\in\Pa^\di(r,s)$,
there is a sequence $\as=(i_1,\dotsc,i_m)$ of indices in $I_n$
such that
\begin{equation} \label{E:raisetobrs}
e_{\as}(u_{\ell\La_0}\otimes b(r,s,\la))=u_{\ell\La_0}\otimes m(B^{r,s}),
\end{equation}
where $\ell=\lev(B^{r,s})$. Moreover
\begin{align}\label{E:0steps}
\card\,\{j\mid i_j=0\} = \dfrac{\abs{\la}-\abs{\lamin^\di(r,s)}}{|\di|}.
\end{align}
\end{lemma}
\begin{proof} This follows from Proposition \ref{P:Brsdual},
\eqref{E:sigmaprop}, and Lemma \ref{L:sigma}.
\end{proof}

\begin{proof}[Proof of Proposition \ref{P:DBrs}]
Equation \eqref{E:0steps} yields
\begin{align*}
\widehat{D}_{B(\vphi(u_0(r,s)))}(u_{\ell\La_0} \otimes b(r,s,\la)) =
\dfrac{\abs{\la}-\abs{\lamin^\di(r,s)}}{|\di|}.
\end{align*}
By Corollary \ref{C:DBrs} we have
\begin{align}\label{E:Drel1}
  \Db_{B^{r,s}}(b(r,s,\la)) = \Db_{B^{r,s}}(m(B^{r,s})) - \dfrac{\abs{\la}-\abs{\lamin^\di(r,s)}}{|\di|}.
\end{align}
Applying this for $\la=(s^r)$ we have
\begin{align}\label{E:Drel2}
  \Db_{B^{r,s}}(b(r,s,(s^r))) = \Db_{B^{r,s}}(m(B^{r,s})) - \dfrac{rs - \abs{\lamin^\di(r,s)}}{|\di|}.
\end{align}
Subtracting \eqref{E:Drel2} from \eqref{E:Drel1} and using Lemma \ref{L:DBrs}(3)
and the fact that $u(B^{r,s})=b(r,s,(s^r))$, we obtain \eqref{E:DBrs} as required.
\end{proof}

\section{Splittings}

In this section we define maps that embed a KR crystal into the tensor product of
KR crystals. These maps are $I_0$-crystal embeddings which are compatible with
the grading. These results hold for any nonexceptional affine algebra $\geh$
and any $r\in I_0$ with $r\ne1$ and $r$ nonspin.

\subsection{Row splitting}
\label{SS:rowsplit}

In this section we construct a map which we call row splitting,
because in type $A$, the map simply splits off the top row of a
rectangular tableau.

\begin{proposition}
\label{prop:row split} For $\geh$ nonexceptional,
$r\in I_0$ not a spin node and $r\ne1$, there exists a unique map
\begin{equation*}
\SR:B^{r,s}\longrightarrow B^{r-1,s}\otimes B^{1,s}
\end{equation*}
satisfying
\begin{align}
\label{E:Sgood}
\SR(e_i(b)) &= e_i(\SR(b)) \qquad\text{for any good arrow $b\to e_i(b)$.}
\end{align}
\end{proposition}
\begin{proof} By Lemma \ref{L:goodarrows}, $B^{r,s}$ is connected by good arrows.
By \eqref{E:Sgood} it follows that $\SR$ is completely determined by any
single value. Again by \eqref{E:Sgood}, $\SR$ is an $I_0$-crystal embedding.
But $\SR(u(B^{r,s}))=u(B^{r-1,s}) \otimes u'$ where $u'$ is the unique element in $B^{1,s}$
of weight $s(\omega_r-\omega_{r-1})$,
since these elements are the only ones in their respective crystals that are $I_0$-highest weight vertices
of weight $s\omega_r$. So it remains to show existence.

Let $\ell=\lev(B^{r,s})$ be the common level of $B^{i,s}$ for
$i\in I_0$ nonspin. By Lemma \ref{L:phimin} and Theorem \ref{T:KMN}
there are isomorphisms
\begin{align}\label{E:affiso1}
  B(\ell\La_0) \otimes B^{r,s} &\cong B(\vphi(m(B^{r,s}))) \\
\label{E:affiso2}
  B(\ell\La_0) \otimes B^{r-1,s} \otimes B^{1,s} &\cong \bigoplus_{u'} B(\vphi(u'))
\end{align}
where $u'\in B^{1,s}$ satisfies
\begin{align}\label{E:affhwpiece}
  \veps(u')=\vphi(m(B^{r-1,s})).
\end{align}

In the nonperfect case there may be more than one such $u'$.
However there is a unique $u'\in B^{1,s}$ such that \eqref{E:affhwpiece} holds and also
\begin{align} \label{E:affhwpiecephi}
  \vphi(u') = \vphi(m(B^{r,s})).
\end{align}

First suppose $B^{i,s}$ is perfect for $i\in I_0$ nonspin.
Since $m(B^{r-1,s})\in B^{r-1,s}_{\min}$, $u'$ satisfying \eqref{E:affhwpiece} is unique,
in which case we must show this $u'$ satisfies \eqref{E:affhwpiecephi}.

For every $i\in I_0$ define $t_{-c_{i^*}\omega_{i^*}}=w_i \tau_i$ where $w_i\in W$ and $\tau_i\in \Sigma$.
One may verify that $\tau_r = \tau_{r-1} \tau_1$. Perfectness yields the isomorphism
\begin{align} \label{E:newsplitiso}
B(\ell\La_0) \otimes B^{r,s} \cong B(\ell\La_{\tau_r(0)}) \cong B(\ell\La_0) \otimes B^{r-1,s} \otimes B^{1,s}
\end{align}
with $u_{\ell\La_0} \otimes m(B^{r,s}) \mapsto u_{\ell\La_0} \otimes m(B^{r-1,s}) \otimes u'$.
Equation \eqref{E:affhwpiecephi} follows by
applying $\vphi$ to these highest weight vectors.

Suppose $B^{1,s}$ is not perfect. Then $\geh=C_n^{(1)}$, $s=2\ell-1$
and $\lev(B^{1,s})=\ell$. In this case $\vphi(m(B^{r-1,s}))=(\ell-1)\La_0+\La_{r-1}$
and $\vphi(m(B^{r,s}))=(\ell-1)\La_0+\La_r$.
There are exactly three elements $u'\in B^{1,s}$ with $\veps(u')=\vphi(m(B^{r-1,s}))=(\ell-1)\La_0+\La_{r-1}$.
Namely, $r,\ol{r-1}\in B(\omega_1)$ and $1(r-1)\ol{r-1}\in B(3\omega_1)$. (If $s=1$, neglect the last one.)
The values of $\varphi$ are
$(\ell-1)\La_0+\La_r,(\ell-1)\La_0+\La_{r-2},(\ell-2)\La_0+\La_1+\La_{r-1}$, respectively.
%
%
%
Let $u'=r$. 
$B(\La')\cong B(\ell\La_0) \otimes B^{r,s}$
in $B(\ell\La_0) \otimes B^{r-1,s} \otimes B^{1,s}$ such that
$u_{\ell\La_0} \otimes m(B^{r,s}) \mapsto u_{\ell\La_0} \otimes m(B^{r-1,s}) \otimes u'$.

Since $B^{r,s}$ is connected by good arrows, we may define $\SR$ by
\begin{align} \label{E:splitdef}
  \SR(b) = b_1 \otimes b_2 \qquad\text{ where $u_{\ell\La_0} \otimes b\mapsto u_{\ell\La_0} \otimes b_1 \otimes b_2$ under \eqref{E:newsplitiso}.}
\end{align}
Equation \eqref{E:Sgood} follows immediately.

%
\end{proof}

\subsection{Splitting $B\in\CC$ into rows}
Let $\geh$ be nonexceptional.
We use Notation \ref{N:BR} for $B=B^R$. Let $B^{\rows(R)}\in\CCinf(\geh)$ be
defined by replacing each $B^{r,s}$ in $B$ by $(B^{1,s})^{\otimes
r}$. We define a map
\begin{align}
  \spl=\spl_R: B^R \to B^{\rows(R)}
\end{align}
as follows. Starting $B^R$ we define a sequence of maps that go
through various crystals in $\CC$, ending with $B^{\rows(R)}$. We
repeat the following step. We locate the leftmost tensor factor of
the form $B^{r,s}$ with $r>1$, apply a sequence of combinatorial
$R$-matrices to swap it to the left, and apply $\SR \otimes
\text{id}$ (which we will sometimes by abuse of notation also denote
$\SR$), which trades in $B^{r,s}$ for $B^{r-1,s} \otimes B^{1,s}$.
Eventually the current crystal consists tensor factors of the form
$B^{1,s}$, and we apply a sequence of combinatorial $R$-matrices to
reorder the tensor factors, obtaining $B^{\rows(R)}$. Call the
composite map $\spl_R$. It is an $I_0$-crystal morphism, being the
composition of such (see \eqref{E:Sgood}).

\begin{remark} \label{R:splitunique}
One can apply splitting of the \textit{first} tensor factor and
combinatorial $R$-matrices in any order until $B^{\rows(R)}$ is
reached. We conjecture that the resulting map is independent
of the order that these steps were taken.
\end{remark}

\begin{proposition}
\label{P:Ssigmacommute} For $\geh=\geh^\di$ reversible, $B^R\in \CCinf(\geh^\di)$,
and $b\in \tops(B^R)$ we have
\begin{align}\label{E:splitsigma}
\spl_R \circ \sigma_{B^R} (b)=\sigma_{B^{\rows(R)}} \circ \spl_R(b).
\end{align}
\end{proposition}
\begin{proof} By \eqref{E:sigmaiso} we may reduce
to the case $B=B^{r,s}$ and $\spl_R=S$. Let $\ell=\lev(B^{r,s})$.
By \eqref{E:sigmaprop}
and the fact that $S$ is an $I_0$-crystal morphism, we may assume
$b\in \hw_{I_0}(\tops(B^{r,s}))$.
By Lemma \ref{Lem_nice-path}, there is a sequence
$\mba=(i_1,\dotsc,i_p)$ of indices in $I_n$ such that
$e_\mba(u_{\ell\La_0} \otimes b)=u_{\ell\La_0}\otimes m(B^{r,s})$. Therefore
we have
\begin{align} \label{E:start}
  e_\mba(b) = m(B^{r,s})
\end{align}
and moreover this sequence consists of good arrows.
Applying $\sigma \SR$ we obtain
\begin{align*}
  \sigma(\SR(m(B^{r,s}))) &= \sigma(\SR(e_\mba(b))) \\
  &= e_{\sigma(\mba)} \sigma(\SR(b))
\end{align*}
using \eqref{E:Sgood} and \eqref{E:sigmaprop}. Applying
$\SR \sigma$ to \eqref{E:start} we have
\begin{align*}
\SR(\sigma(m(B^{r,s})) &= \SR(\sigma(e_\mba(b))) \\
&= e_{\sigma(\mba)} \SR(\sigma(b))
\end{align*}
using \eqref{E:sigmaprop}, the fact that $\sigma(\mba)$ has indices in $I_0$, and
\eqref{E:Sgood}. Since $e_{\sigma(\mba)}$ has a left inverse, we may assume
that $b=m(B^{r,s})$. We have $\SR(m(B^{r,s}))=m(B^{r-1,s}) \otimes u'$ for some $u'\in B^{1,s}$.
By Proposition \ref{P:sigma}(2) we reduce to the equality
\begin{align}
  S(\bamin^\di(r,s)) = \bamin^\di(r-1,s) \otimes \sigma(u').
\end{align}
Since $\bamin^\di(r,s)\in B_{I_0}(s\omega_r)$, we may apply $\rowtab=\rowtab_{(s^r)}$
and similarly for $\bamin^\di(r-1,s)$.
By definition $(\rowtab_{(s^{r-1})} \otimes 1_{B(s\omega_1)})(S(\bamin^\di(r,s))) =
\rowtab_{(s^r)}(\bamin^\di(r,s)) = \rowtab_{(s^{r-1})}(\bamin^\di(r-1,s)) \otimes u''$
where $u''\in B_{I_0}(s\omega_1)$ is the last row of $\rowtab_{(s^r)}(\bamin^\di(r,s))$. So it remains to show
$u''=\sigma(u')$. Using the
explicit form of $\rowtab(\bamin^\di(r,s))$ given in Remark \ref{R:rowtabmins} one has
\begin{align*}
\veps(u'')&=
\begin{cases}
\ell\La_{n-1} &(\diamondsuit=(1,1),r:\text{even}) \\
(\ell-1)\La_n+\La_{n-r+1} &(\diamondsuit=(2),s:\text{odd}) \\
\ell\La_n &(\text{otherwise})
\end{cases} \\
\vphi(u'')&=
\begin{cases}
\ell\La_{n-1} &(\diamondsuit=(1,1),r:\text{odd}) \\
(\ell-1)\La_n+\La_{n-r} &(\diamondsuit=(2),s:\text{odd}) \\
\ell\La_n &(\text{otherwise}).
\end{cases}
\end{align*}
Therefore $\veps(\sigma(u''))$ and $\vphi(\sigma(u''))$ are given by replacing every $\La_j$ with
$\La_{n-j}$ in the above table.
But $\veps(\sigma(u''))=\vphi(m(B^{r-1,s}))$ and $\vphi(\sigma(u''))=\vphi(m(B^{r,s}))$.
Therefore by \eqref{E:splitdef} $u'=\sigma(u'')$ for there is a unique element in $B^{1,s}$
having such values of $\veps$ and $\vphi$, and we are done since $\sigma$ is an involution.
\end{proof}

\subsection{Box splitting}
\label{SS:boxsplit}

Let $\geh$ be of affine type
such that $\geh_0$ is of type $B_n$, $C_n$, or $D_n$.

Define a map $B^{1,s} \hookrightarrow B^{1,s-1} \otimes B^{1,1}$ as follows.
For $b=x_1\dotsm x_p\in B(p\omega_1) \subset B^{1,s}$,
\begin{align} \label{E:splitonebox}
  b\mapsto \begin{cases}
  1b \otimes \bar{1} & \text{if $s \ge p+2$} \\
  b \otimes \varnothing & \text{if $s=p+1$} \\
  x_2\dotsm x_p \otimes x_1 & \text{if $s=p$.}
  \end{cases}
\end{align}
Here $\varnothing$ denotes the element of $B(0)\subset B^{1,1}$ for $\geh$ of kind $(1)$.
This map is evidently an $I_0$-crystal embedding.
Iterating this map on the first tensor factor, we obtain
the following $I_0$-crystal embedding $\Sbox:B^{1,s}\hookrightarrow (B^{1,1})^{\otimes s}$:
\begin{align} \label{E:Sbox}
\Sbox(b) = x_p \otimes \dotsm \otimes x_2 \otimes x_1 \otimes \underbrace{1 \otimes\dotsm\otimes 1}_{\text{$m$ times}} \otimes
\underbrace{\varnothing}_{\text{$k$ times}} \otimes
\underbrace{\bar{1}\otimes\dotsm\otimes \bar{1}}_{\text{$m$ times}}
\end{align}
where $m=\lfloor \frac{s-p}{2} \rfloor$ and $k$ is $0$ or $1$ according as $s-p$ is even or odd.

Define a map $\SSbox:B^R\to (B^{1,1})^{\otimes \abs{R}}$ as follows.
First apply $\spl: B^R \to B^{\rows(R)}$. Then do the following
repeatedly until $(B^{1,1})^{\otimes \abs{R}}$ is reached. Find the
leftmost factor of the form $B^{1,s}$ with $s>1$ and swap it to the
left end using combinatorial $R$-matrices and then apply $\Sbox
\otimes \text{id}$ to replace this $B^{1,s}$ with
$(B^{1,1})^{\otimes s}$. Write $\SSbox$ for the composite map. We
have
\begin{align} \label{E:boxsplitI0map}
  \SSbox \circ e_i = e_i \circ \SSbox \qquad\text{for $i\in I_0$}
\end{align}
since $\SSbox$ is the composition of $I_0$-crystal morphisms
$\spl$ and $\Sbox\otimes 1$.

\begin{remark} If $R$ consists of tensor factors of the form $B^{1,s}$ then
$\Db_{B^R} = \Db_{(B^{1,1})^{\otimes \abs{R}}} \circ \SSbox$.
\end{remark}

\begin{proposition}\label{P:sigmaSboxcommute}
For $\geh=\geh^\di$ reversible and $b\in \tops(B^R)$,
\begin{equation}\label{E:sigSSboxcommute}
  \SSbox(\sigma(b))=\sigma(\SSbox(b)).
\end{equation}
\end{proposition}
\begin{proof} By Proposition \ref{P:Ssigmacommute}, \eqref{E:sigmaiso},
\eqref{E:boxsplitI0map}, and \eqref{E:sigmaprop} it suffices to
prove \eqref{E:sigSSboxcommute} for $b\in \tops(B^{1,s})$.
Consider the case $\di=(1)$ where
$\tops(B^{1,s})$ consists of elements $1^p=\hw_{I_0}(B(p\omega_1))\subset B^{1,s}$
for $0\le p\le s$. With notation as in \eqref{E:Sbox} we have
\begin{align*}
  \sigma(\Sbox(1^p))&=\sigma(1^{\otimes p+m} \otimes \varnothing^{\otimes k} \otimes \bar{1}^m) \\
  &= \bar{n}^{\otimes p+m} \otimes 0^k \otimes n^{\otimes m} \\
  &= \Sbox(n^m 0^k \bar{n}^{p+m}) \\
  &= \Sbox(\sigma(1^p)).
\end{align*}
The cases $\di\in\{(1,1),(2)\}$ are easier.
\end{proof}

\section{Correspondence on $A_{n-1}$-highest weight vertices}
Again we assume that $\geh=\geh^\di$ is reversible.

Let $\hmax(B)=\widehat{\max(B)}$ in the notation of Section
\ref{SS:subcrystals}. The goal of this section is to prove the
following theorem.

\begin{theorem}
\label{Th_corespondence} For $B\in\CCinf(\geh^\di)$ and every $\la\in\Pa_n$, $\sigma:B\to B$
restricts to a bijection
\begin{align}
\label{E:topshmax}
\hw_{A_{n-1}}^\la(\tops(B)) \overset{\sigma}\simeq \hw_{A_{n-1}}^{\bla}(\hmax(B)).
\end{align}
\end{theorem}

\begin{lemma} \label{L:samesize}
\begin{align}\label{E:samesize}
\card\,\hw_{A_{n-1}}^\la(\tops(B)) =
\card\,\hw_{A_{n-1}}^{\bla}(\hmax(B)).
\end{align}
\end{lemma}
\begin{proof} There is an $I_0$-crystal isomorphism
\begin{equation*}
\max(B^R) \simeq \bigoplus_{\nu \in \Pa_n} B(\nu )^{\oplus
c_{R_1,\ldots ,R_p}^\nu}.
\end{equation*}
By \eqref{E:Lit} we have
\begin{align*}
  \card\, \hw_{A_{n-1}}^{\bla}(\hmax(B^R)) =
  \sum_{\nu\in\Pa_n} c_{R_1,\dots,R_p}^\nu \sum_{\delta\in \Pvd_n} c^\nu_{\delta\la} =
  \mathfrak{K}_{R_1,\dotsc,R_p}^{\la,\di}
\end{align*}
where the last equality follows from Proposition \ref{prop_qdual}.
We have $X_{\la,B^R}(1)=\mathfrak{K}_{R_1,\dotsc,R_p}^{\la,\di}$
by \eqref{E:Xdef}, \eqref{E:KRI0decomp}, and \eqref{E:kappadef}.
Therefore \eqref{E:samesize} holds.
\end{proof}


\begin{proposition}
\label{Prop_corres} The map $\sigma:B\to B$ sends $\tops(B)$ into $\max(B)$.
\end{proposition}
\begin{proof}
Let $b\in\tops(B)$. By \eqref{E:boxsplitI0map}
$\SSbox(b)\in\tops((B^{1,1})^{\otimes |R|})$. Assuming the Proposition
holds for tensor powers of $B^{1,1}$ and using Proposition
\ref{P:sigmaSboxcommute} we have $\SSbox(\sigma(b))\in
\max((B^{1,1})^{\otimes |R|})$. By \eqref{E:boxsplitI0map}, we deduce
that $\sigma(b)\in \max(B^R)$.

We now assume $B=(B^{1,1})^{\otimes m}$ and $b\in \tops(B)$. We may assume that
$b\in\hw_{A_{n-1}}(\tops(B))$. By induction on $m$, the letters of $b$
lie in the set $\{1,2,\dotsc,m\}\cup\{\bar{m},\dotsc,\bar{1}\}$.
Thus the letters of $\sigma(b)$ belong to
$\{n-m+1,\ldots ,n,\overline{n},\ldots ,\overline{n-m+1}\}$. When $n$ is
sufficiently large, this implies that $\sigma(b)\in \max(B)$. This
can either be proved by induction on $m$ or more directly by using the
insertion procedure described in \cite{lec2}.
\end{proof}

\begin{proof}[Proof of Theorem \ref{Th_corespondence}]
By Proposition \ref{Prop_corres} $\sigma$ sends $\tops(B)$ into $\max(B)$.
Since $\tops(B)$ is an $A_{n-1}$-crystal whose weights lie in $\Z_{\ge0}^n$
and $\sigma$ sends such weights to $\Z_{\le0}^n$ by \eqref{E:sigmawt},
$\sigma$ must send $\tops(B)$ into $\hmax(B)$.
By \eqref{E:sigmaprop} and \eqref{E:sigmawt} $\sigma$
sends $\hw_{A_{n-1}}^\la(\tops(B))$ into $\hw_{A_{n-1}}^{\bla}(\hmax(B))$.
Theorem \ref{Th_corespondence} follows due to Lemma \ref{L:samesize}
and the injectivity of $\sigma$ (which holds by \eqref{E:sigmainv}).
\end{proof}

\section{A relation between $\Db$ and $\Db\circ\sigma$}
In this section we assume $\geh=\geh^\di$ is reversible.
Define the map $B\to\Pa_n$ by $b\mapsto\la(b)$ where
\begin{align} \label{E:deflab}
  B_{I_0}(b) \cong B(\la(b)).
\end{align}
The goal of this section is to prove the following Theorem.

\begin{theorem} \label{Th_SR} For $B^R\in\CC(\geh^\di)$ and $b\in\tops(B^R)$
\begin{equation}\label{rel}
  \Db(b) = \Db(\sigma(b)) + \frac{\abs{R} - \abs{\la(b)}}{|\di|}.
\end{equation}
\end{theorem}
We use Notation \ref{N:BR}. For $b\in\CC(\geh^\di)$ set $\nu(b)=\wt(b)=(\nu_1(b),\nu_2(b),\ldots)$ 
and $|\nu(b)|=\sum_i\nu_i(b)$.
Note that some $\nu_i(b)$'s may be negative. Hence $|\nu(b)|=\sum_i\nu_i(b)$ may also become negative. 
We prepare a lemma.
\begin{lemma} Let $B_1,B_2\in\CC(\geh^\di)$. Let $b_1\ot b_2$ be an element of $B_1\ot B_2$
and suppose it is mapped to $b'_2\ot b'_1$ by the combinatorial $R$-matrix. Then we have
\begin{equation} \label{strange H}
\ol{H}(b_1\ot b_2)-\ol{H}(\sigma(b_1)\ot\sigma(b_2))=\frac{|\nu(b'_2)|-|\nu(b_2)|}{|\di|}.
\end{equation}
\end{lemma}

\begin{proof}
Since $B_1\ot B_2$ is connected, it is sufficient to show 
\begin{itemize}
\item[(i)] if $b_1=u(B_1),b_2=u(B_2)$, \eqref{strange H} holds, and
\item[(ii)] \eqref{strange H} with $b_1\ot b_2$ replaced by $e_i(b_1\ot b_2)$ holds,
	provided that \eqref{strange H} holds and $e_i(b_1\ot b_2)\ne0$.
\end{itemize}
For (i) recall $b'_1=b_1,b'_2=b_2$ if $b_1=u(B_1),b_2=u(B_2)$. Since $u(B_1)\ot u(B_2)$
can be reached from $\sigma(u(B_1))\ot\sigma(u(B_2))$ by applying $e_i$\,($i\ne0$), we have
$\ol{H}(u(B_1)\ot u(B_2))=\ol{H}(\sigma(u(B_1))\ot\sigma(u(B_2)))=0$. Hence (i) is verified.

For (ii) recall $|\nu(e_ib)|-|\nu(b)|=-|\di|\,(i=0),=|\di|\,(i=n),=0\,(\text{otherwise})$.
If $i\ne0,n$, both sides do not change when we replace $b_1\ot b_2$ with $e_i(b_1\ot b_2)$.
If $i=0$, the first term of the l.h.s decreases by one in case LL, increases by one in case RR,
and does not change in case LR or RL. (For the meaning of LL, etc, see Proposition \ref{P:RH}(2).)
The second term does not change, while the r.h.s varies in the same way as the first term of the l.h.s.
The $i=n$ case is similar. 
\end{proof}

\begin{proof}[Proof of Theorem \ref{Th_SR}]
We may reduce to the case that $b\in\hw_{A_{n-1}}(\tops(B))$
since $\tops(B)$ is an $A_{n-1}$-crystal and the entire equation
\eqref{rel} is invariant under $A_{n-1}$-arrows.

We proceed by induction on the number $p$ of tensor factors in $B^R$.
When $p=1$ we have $B=B^{r,s}$.
By \eqref{E:KRI0decomp} $b=b(r,s,\la)$ for some $\la\in \Pa(r,s)$.
By Theorem \ref{Th_corespondence}, $\sigma(b)\in \hmax(B)\subset \max(B)=B((s^r))\subset B^{r,s}$.
But $\Db$ is $0$ on $B((s^r))$ by the definition of $\Db_{B^{r,s}}$.
Therefore $\Db(\sigma(b))=0$. Then \eqref{rel} holds by Proposition \ref{P:DBrs}.

Let $B=B'\ot B^{r_p,s_p}$ and $b_1\ot b_2\in B'\ot B^{r_p,s_p}$ is mapped to $b'_2\ot b'_1
\in B^{r_p,s_p}\ot B'$ by the affine crystal isomorphism. Then $\sigma(b_1)\ot\sigma(b_2)$
should be mapped to $\sigma(b'_2)\ot\sigma(b'_1)$. Using \eqref{E:intrinsic} we have
\begin{align*}
\ol{D}(b)&=\ol{D}(b_1)+\ol{D}(b'_2)+\ol{H}(b_1\ot b_2),\\
\ol{D}(\sigma(b))&=\ol{D}(\sigma(b_1))+\ol{D}(\sigma(b'_2))+\ol{H}(\sigma(b_1)\ot\sigma(b_2)).
\end{align*}
On the other hand, by the previous lemma we have
\[
\ol{H}(b_1\ot b_2)-\ol{H}(\sigma(b_1)\ot\sigma(b_2))=\frac{|\la(b'_2)|-|\la(b_2)|}{|\di|}.
\]
Using the induction hypothesis we obtain
\begin{align*}
\ol{D}(b)-\ol{D}(\sigma(b))&=\frac{|B'|-|\la(b_1)|}{|\di|}+\frac{|B^{r_p,s_p}|-|\la(b'_2)|}{|\di|}
+\frac{|\la(b'_2)|-|\la(b_2)|}{|\di|}\\
&=\frac{|B|-|\la(b)|}{|\di|}
\end{align*}
as desired.
\end{proof}

\section{Energy function on max elements}
\subsection{Highest elements in $\max(B^{r_1,s_1}\otimes
B^{r_2,s_2})$}

\begin{proposition}
\label{prop:ht max}
Let $b_1\otimes b_2\in\hw_{I_0}(\max(B^{r_1,s_1}\otimes B^{r_2,s_2}))$ and $r=\min(r_1,r_2)$.
\begin{itemize}
\item[(1)]
Then $b_1=b(r_1,s_1,(s_1^{r_1}))$ and
there exists a partition $\la \subset (s_2^{r_2})$ such that $\ell(\la)\le
r$ and $\la_r\ge s_2-s_1$, and $b_2\in B(s_2^{r_2})$ is the
tableau whose entries are $i$ in the $i$-th row in $\la$ and
$r_1+1,r_1+2,\ldots$ from bottom to top outside of $\la$.

\item[(2)] Let $\la$ be as in (1).
Suppose $b_1\otimes b_2$ is sent to $b_2'\otimes b_1'$ by the combinatorial $R$.
Then the corresponding partition $\mu$ of $b_1'$ is obtained from $\la$ by adding
$s_1-s_2$ (resp. removing $s_2-s_1$) columns of height $r$ if $s_1 \ge s_2$
(resp. $s_1 \le s_2$).
\end{itemize}
\end{proposition}

\begin{proof}
(1) is immediate from \eqref{E:tensor}. For (2) note that
the combinatorial $R$ preserves the weight. Given a highest element $b_1\otimes b_2$ as in (1),
there is a unique highest element in $\max(B^{r_2,s_2}\otimes B^{r_1,s_1})$
of the same weight.
\end{proof}

\begin{example}
\label{ex:ht max}
\begin{equation*}
\vcenter{
\tableau[sby]{4&4&4&4\\3&3&3&3\\2&2&2&2\\1&1&1&1}
}
\otimes
\vcenter{
\tableau[sby]{5&6&6&7&9\\ \BD &5&5&6&8\\ \BC&\BC&\BC&5&7\\ \BB&\BB&\BB&\BB&6\\ \BA&\BA&\BA&\BA&5}
}
\end{equation*}
is the unique element of $\hw_{I_0}(\max(B^{4,4}\otimes B^{5,5}))$ with associated $\la
=(4,4,3,1)$. By the combinatorial $R$ it is sent to
\begin{equation*}
\vcenter{
\tableau[sby]{5&5&5&5&5\\4&4&4&4&4\\3&3&3&3&3\\2&2&2&2&2\\1&1&1&1&1}
}
\otimes
\vcenter{
\tableau[sby]{6&6&7&9\\ \BC&\BC&6&8\\ \BB&\BB&\BB&7\\ \BA&\BA&\BA&6}
}
\end{equation*}
\end{example}

Our goal in this section is to prove the following proposition.

\begin{proposition}
\label{prop:H for max} Assume $s_1\geq s_2$ and let $r=\min(r_1,r_2)$.
Let $b_1\otimes b_2\in \hw_{I_0}(\max(B^{r_1,s_1}\otimes B^{r_2,s_2}))$
whose partition is $\la$. Then
\begin{equation*}
\Hb(b_1\otimes b_2) = \frac2{\abs{\diamondsuit}}(r s_2-\abs{\la}).
\end{equation*}
\end{proposition}

Let $e_i^{\max}(b) = e_i^{\veps_i(b)}(b)$.

\begin{lemma}
\label{lem:e0} Let $B^{r,s}$ be a KR crystal of type $D_n^{(1)}$.
Let $\alpha,\beta,\gamma\in\Z_{\ge0}$ sum to $s$
and let $b$ be the element of $\mathrm{max}(B^{r,s})$ with
$\alpha$ columns whose entries are $1,2,\ldots ,r$ from bottom to
top, $\beta$ columns with $2,3,\ldots,r+1$ and $\gamma$ columns with
$3,4,\ldots,r+2$. Then
\begin{itemize}
\item[(1)] $\veps_0(b)= 2 \alpha +\beta$, $\vphi_0(b)=0$, and

\item[(2)] $e_0^{\max}(b)$ is the tableau with $\gamma$ columns
with $3,4,\ldots,r+2$, $\beta$ columns with $3,4,\ldots,r+1,\overline{1}$
and $\alpha$ columns with $3,4,\ldots,r,\overline{2},\overline{1}$.
\end{itemize}
\end{lemma}
\begin{proof}
$b$ is a $\{3,4,\ldots ,n\}$-highest weight vertex. As is explained in section 4.2
of \cite{FOS1}, such elements are in one-to-one correspondence with pairs of
$\pm $-diagrams $(P,p)$, where the inner shape of $P$ is the outer shape of $p$.
$b$ corresponds to $(P,p)$, where $P$ has the outer shape $(s^{r})$ and
the inner shape $(s^{r-1},s-\alpha )$, and $p$ has the inner shape $
(s^{r-2},s-\alpha ,s-\alpha -\beta )$. The signs in $P$ and $p$ are all $+$.
Once we have the corresponding pair of $\pm $-diagrams, it is easy to see $\veps_0$,
$\vphi_0$, and the action of $e_0$. As a result we see $e_0^{\max}(b)$
corresponds to the pair of $\pm$-diagrams with all $+$ in $(P,p)$ being replaced with $-$.
In turn this yields the above tableau.
\end{proof}

\begin{lemma}
\label{lem:e0 ht} Let $b$ and $\la$ be as in Proposition \ref{prop:ht max}(1).
Let $r=\min(r_1,r_2)$ and let $w_0^r$ be the longest element of the symmetric
group $S_r\subset W$ generated by $s_1$ through $s_{r-1}$.
Then $\hw_{I_0}(e_{0}^{\max}(b^{w_0^r}))$ has
associated partition $\lam$ obtained by adding
$\min(2,r-\la_j')$ (resp. $\min(1,r-\la_j')$) boxes to the $j$-th column for $1\le j\le s_2$
for $\di=(1,1)$ (resp. $\di=(1),(2)$).
\end{lemma}

\begin{example}
Let $b$ be the element of $\max(B^{4,4}\otimes B^{5,5})$ of Example
\ref{ex:ht max}.
\begin{equation*}
b=
\vcenter{
\tableau[sby]{4&4&4&4\\3&3&3&3\\2&2&2&2\\1&1&1&1}
}
\otimes
\vcenter{
\tableau[sby]{5&6&6&7&9\\ \BD &5&5&6&8\\ \BC&\BC&\BC&5&7\\ \BB&\BB&\BB&\BB&6\\ \BA&\BA&\BA&\BA&5}
}
\overset{w_0^r}{\longrightarrow}
\vcenter{
\tableau[sby]{4&4&4&4\\3&3&3&3\\2&2&2&2\\1&1&1&1}
}
\otimes
\vcenter{
\tableau[sby]{5&6&6&7&9\\4&5&5&6&8\\3&4&4&5&7\\2&3&3&4&6\\1&2&2&3&5}
}
\end{equation*}
\begin{equation*}
\overset{e_0^{\max}}{\longrightarrow}
\vcenter{
\tableau[sby]{\ichi&\ichi&\ichi&\ichi\\ \nii&\nii&\nii&\nii \\ 4&4&4&4\\3&3&3&3}
}
\otimes
\vcenter{
\tableau[sby]{7&7&\ichi&\ichi&\ichi\\6&6&6&8&\nii\\5&5&5&5&7\\4&4&4&4&6\\3&3&3&3&5}
}
\overset{\hw_{I_0}}{\longrightarrow}
\vcenter{
\tableau[sby]{4&4&4&4\\3&3&3&3\\2&2&2&2\\1&1&1&1}
}
\otimes
\vcenter{
\tableau[sby]{5&5&5&5&7\\ \BD&\BD&\BD&\BD&6\\ \BC&\BC&\BC&\BC&5\\ \BB&\BB&\BB&\BB&\BB\\ \BA&\BA&\BA&\BA&\BA}
}
\end{equation*}
We indicate $\la$ and $\lam$ by boldface entries.
\end{example}

\begin{proof}[Proof of Lemma \ref{lem:e0 ht}]
We first treat the case of $\di=(1,1)$.
$b^{w_0^r}$ is obtained from $b$ by modifying the $\la$-part of the second component of $b$
as follows. The column of entries $1,2,\dotsc,h$ ($h\le r$) reading from bottom to top is
replaced by $r-h+1,r-h+2,\dotsc,r$.

Next we want to apply $e_0^{\max}$. Suppose $r_1\le r_2$. (The other case is similar.)
Write $b^{w_0^r}=\tilde{b}_1\ot\tilde{b}_2$. From
Lemma \ref{lem:e0} we have $\vphi_0(\tilde{b}_1)=0$, and $e_0^{\max}(\tilde{b}_1)$ is
the tableau with $s_1$ columns of entries $3,4,\dotsc,r_1,\ol{2},\ol{1}$. To calculate
$e_0^{\max}(\tilde{b}_2)$ define a sequence $\boldsymbol{a}=\boldsymbol{a}_{r_2}
\sqcup\cdots\sqcup\boldsymbol{a}_2\sqcup\boldsymbol{a}_1$ where
\[
\boldsymbol{a}_j=((j+2)^{s_2-\la_{r_1-2}},\dotsc,(r_1+j-2)^{s_2-\la_2},
(r_1+j-1)^{s_2-\la_1})
\]
for $j=1,2,\dotsc,r_2$, and set $\tilde{b}'_2=e_{\boldsymbol{a}}(\tilde{b}_2)$.
Then $\tilde{b}'_2$ is the tableau with $\la_{r_1}$ columns of entries
$1,2,\dotsc,r_2$, $\la_{r_1-1}-\la_{r_1}$ columns of entries $2,3,\dotsc,r_2+1$ and
$s_2-\la_{r_1-1}$ columns of entries $3,4,\dotsc,r_2+2$. Again from Lemma \ref{lem:e0}
$e_0^{\max}(\tilde{b}'_2)$ is the tableau with $s_2-\la_{r_1-1}$ columns of
$3,4,\dotsc,r_2+2$, $\la_{r_1-1}-\la_{r_1}$ columns of $3,4,\dotsc,r_2+1,\ol{1}$ and
$\la_{r_1}$ columns of $3,4,\dotsc,r_2,\ol{2},\ol{1}$. Since $e_i,f_i$ for $3\le i\le n$
commutes with $e_0$, we have $\tilde{b}''_2=e_0^{\max}(\tilde{b}_2)
=f_{\text{Rev}(\boldsymbol{a})}e_0^{\max}(\tilde{b}'_2)$, where $\text{Rev}
(\boldsymbol{a})$ is the reverse sequence of $\boldsymbol{a}$. The $j$-th row of
$\tilde{b}''_2$ from bottom ($1\le j\le r_2$) is given by
\begin{align*}
&(j+2)^{\la_{r_1-2}}(j+3)^{\la_{r_1-3}-\la_{r_1-2}}\cdots
(r_1+j-1)^{\la_1-\la_2}(r_1+j)^{s_2-\la_1}&\text{ for }1\le j\le r_2-2\\
&(r_2+1)^{\la_{r_1-2}-\la_{r_1}}(r_2+2)^{\la_{r_1-3}-\la_{r_1-2}}\cdots
(r_1+r_2-1)^{s_2-\la_1}\ol{2}^{\la_{r_1}}&\text{ for }j=r_2-1\\
&(r_2+2)^{\la_{r_1-2}-\la_{r_1-1}}(r_2+3)^{\la_{r_1-3}-\la_{r_1-2}}\cdots
(r_1+r_2)^{s_2-\la_1}\ol{1}^{\la_{r_1-1}}&\text{ for }j=r_2.
\end{align*}
Thus we have $e_0^{\max}b^{w_0^r}=e_0^{\max}(\tilde{b}_1)\ot
\tilde{b}''_2$.

Finally, we want to calculate the $I_0$-highest vertex of $e_0^{\max}b^{w_0^r}$.
This calculation is long but not difficult, and it is checked that the statement
is true.

For the proof for $\di=(1),(2)$ we use the construction of a KR crystal in
\cite[\S4.3\&\S4.4]{FOS1}. Namely, $B^{r,s}$ is realized as a suitable subset of an
$A^{(2)}_{2n+1}$-KR crystal where $0$ actions are defined in the same way as $D_n^{(1)}$.
Hence we do not repeat the proof.
\end{proof}

\begin{proof}[Proof of Proposition~\ref{prop:H for max}]
Let $b'_2\otimes b'_1$ be the image of $b_1\otimes b_2$
by the combinatorial $R$. We apply $w_0^r$ to both $b_1\otimes
b_2$ and $b_2'\otimes b_1'$ as prescribed in Lemma \ref{lem:e0 ht}.
Noting that $e_0$ commutes with $e_j$ and $f_j$ for $j\ge 3$ we
find the 0-signature of these elements are
$-^{2s_1}\otimes -^{\la_r+\la_{r-1}}$ and
$-^{2s_2}\otimes -^{\la_{r}+\la _{r-1}+2s_1-2s_2}$ for $\di=(1,1)$,
$-^{s_1}\otimes -^{\la_r}$ and
$-^{s_2}\otimes -^{\la_r+s_1-s_2}$ for $\di=(2)$,
$-^{2s_1}\otimes -^{2\la_r}$ and
$-^{2s_2}\otimes -^{2\la_r+2s_1-2s_2}$ for $\di=(1)$
by Lemma \ref{lem:e0}. Setting
$b_1^\circ\otimes b_2^\circ=\hw_{I_0}(e_0^{\max}((b_1\otimes b_2)^{w_0^r}))$ and recalling \eqref{E:Hb0} we have
\begin{equation*}
\Hb(b_1^\circ\otimes b_2^\circ)=\Hb(b_1\otimes b_2)+
\begin{cases}
(\la_{r}+\la_{r-1})-2s_2&\text{for }\di=(1,1)\\
\la_{r}-s_2&\text{for }\di=(2)\\
2\la_{r}-2s_2&\text{for }\di=(1).
\end{cases}
\end{equation*}
This formula implies the desired result.
\end{proof}

\subsection{The general case}
In this section let $\geh$ be an affine algebra such that $\geh_0$ is of type
$B_n$, $C_n$, or $D_n$.
Using Remark \ref{R:topsBnu} with $\nu=(s^r)\in\Pa_n^\infty$ there is a unique
embedding of $A_{n-1}$-crystals
\begin{align}\label{E:embedKR}
B_A^{r,s} \cong B_{A_{n-1}}(s^r) \overset{i_A}{\longrightarrow}
B_{I_0}(s^r) \subset B^{r,s},
\end{align}
which yields an $A_{n-1}$-crystal isomorphism
\begin{align} \label{E:topsiso}
 B_A^{r,s} \cong
\tops(\max(B^{r,s})).
\end{align}

We use Notation \ref{N:BR}. Define
\begin{equation} \label{E:BA}
  B_A = B_A^R = B_A^{r_1,s_1} \otimes\dotsm \otimes B_A^{r_p,s_p}
\end{equation}
where $B_A^{r,s}$ is the type $A_{n-1}^{(1)}$ KR crystal. There is
an embedding
\begin{equation} \label{E:embed}
  i_A^R: B_A^R \to B^R
\end{equation}
given by the tensor product of embeddings \eqref{E:embedKR},
inducing the isomorphism of $A_{n-1}$-crystals
\begin{align} \label{E:embedBR}
  B_A^R \cong \tops(\max(B^R)).
\end{align}

\begin{theorem} \label{T:maxD=A}
For $\di\in\{(1),(2),(1,1)\}$, $B^R\in\CCinf(\geh)$ and $\nu\in\Pa_n^\infty$
such that $|\nu|=|R|$ we have
\begin{align}
  \Xb_{\nu,B^R}^\di(q) = \Xb_{\nu,B_A^R}^\emptyset(q^{\frac2{\abs{\di}}}).
\end{align}
\end{theorem}
\begin{proof} Immediate from \eqref{E:embedBR} and Proposition \ref{P:DmaxB} below.
\end{proof}

\begin{lemma} \label{L:embedR}
Let $R$ and $R'$ be sequences of rectangles that are reorderings of
each other with $B^R,B^{R'}\in\CCinf(\geh)$
and let $g:B^R\to B^{R'}$ be the unique isomorphism of
$I$-crystals. Denote by $g_A:B_A^R\to B_A^{R'}$ the corresponding
isomorphism of crystals of type $A_{n-1}^{(1)}$. Then on $B_A^R$ we
have
\begin{align}
  g \circ i_A^R = i_A^{R'} \circ g_A
\end{align}
\end{lemma}
\begin{proof} One may reduce to the case that $R=(R_1,R_2)$ and
$R'=(R_2,R_1)$ and further to considering only $A_{n-1}$-highest
weight vertices. But then the two sides must agree since $B_A^{R_1}
\otimes B_A^{R_2}$ is $A_{n-1}$-multiplicity-free.
\end{proof}

\begin{lemma} \label{L:embedH} For $B^{R_1}\otimes B^{R_2}\in\CCinf(\geh)$,
we have $\Hb_{B_A^{R_1} \otimes B_A^{R_2}} = \Hb_{B^{R_1} \otimes
B^{R_2}} \circ i_A^{R_1,R_2}$.
\end{lemma}
\begin{proof} This follows from Proposition \ref{prop:H for max}
and the analogous type $A_{n-1}^{(1)}$ result
\cite{Sh:crystal} \cite{SW}.
\end{proof}

\begin{proposition}
\label{P:DmaxB} $\Db_A = \Db\circ i_A$ on $B_A$.
\end{proposition}
\begin{proof} By \eqref{E:intrinsic}, induction, and Lemmata
\ref{L:embedR} and \ref{L:embedH}, we may reduce to the case of a
single tensor factor $B^{r,s}$. Since $B_A^{r,s} \cong
B_{A_{n-1}}(s^r)$ as $A_{n-1}$-crystals \cite{KMN2},
$\Db_{B_A^{r,s}}=0$. But $i_A$ sends the $A_{n-1}$-highest weight
vertex of $B_A^{r,s}$ to $b(r,s,(s^r))=u(B^{r,s})$, on which
$\Db_{B^{r,s}}$ has value $0$ by definition.
\end{proof}

\section{Main results}

\label{Sec_theorems}

\subsection{The decomposition theorem}

We prove Conjecture \ref{CJ:X=K} and any tensor
product of KR crystals.

\begin{theorem}
\label{Th_dec_X} Let $B^R\in\CCinf(\geh)$ where $\geh$ is of kind $\di\in\{(1),(2),(1,1)\}$.
Then for any $\la\in\Pa_n$ we have
\begin{equation*}
\Xb_{\la ,B^R}^\di(q)=q^{\frac{\abs{R}-\abs{\la}}{|\di|}}\sum_{\nu
\in \Pa_n}\sum_{\delta \in \Pa_n^\di} c_{\la,\delta}^\nu
\,\Xb_{\nu,B_A^R}^\emptyset(q^{\frac{2}{\left| \diamondsuit \right| }}).
\end{equation*}
\end{theorem}
\begin{proof}
We have
\begin{align*}
\Xb_{\la,B^R}^\di(q) &= q^{\frac{\abs{R}-\abs{\la}}{|\di|}}
\sum_{b\in \hw_{I_0}^\la(B^R)} q^{\Db(\sigma(b))} \\
&=
q^{\frac{\abs{R}-\abs{\la}}{|\di|}}\sum_{b\in\hw_{A_{n-1}}^{\bla}(\hmax(B^R))}
q^{\Db(b)}
\end{align*}
by \eqref{E:Xdef} and Theorems \ref{Th_SR} and
\ref{Th_corespondence}. $\max(B^R)$ has $I_0$-decomposition
\begin{equation*}
\max(B^R)=\bigoplus_{\substack{\nu \in \Pa_n
\\ \abs{\nu}=\abs{R}}}\bigoplus_{c\in \hw_{I_0}^\nu(B^R)} B(c).
\end{equation*}
For $c\in\hw_{I_0}(\max(B))$, let $\Bh(c):=\widehat{B(c)}$ for the
dual polynomial part of $B(c)$; see Section \ref{SS:subcrystals}.
Taking the dual polynomial part, we have
\begin{equation*}
\hmax(B^R)=\bigoplus_{\substack{\nu \in \Pa_n
\\ \abs{\nu}=\abs{R}}}\bigoplus_{c\in \hw_{I_0}^\nu(B^R)} \Bh(c).
\end{equation*}
Taking $\hw_{A_{n-1}}^{\bla}$, we have
\begin{equation*}
\hw_{A_{n-1}}^{\bla}(\hmax(B^R)) = \bigsqcup_{\substack{\nu\in \Pa_n
\\ \abs{\nu}=\abs{R}}} \bigsqcup_{c\in \hw_{I_0}^\nu(B^R)}
\hw_{A_{n-1}}^{\bla}(\Bh(c)).
\end{equation*}
For $c\in \hw_{I_0}^\nu(B^R)$, observe that $\Db(b)=\Db(c)$ for
$b\in \hw_{A_{n-1}}^{\bla}(\Bh(c))$ since these vertices $b$ all
belong to the same classical component $B(c)$. This gives
\begin{equation*}
\Xb_{\la,B^R}^\di(q)=q^{\frac{\abs{R}-\abs{\la}}{|\di|}}\sum_{\substack{\nu
\in \Pa_n \\ \abs{\nu}=\abs{R}}} \sum_{c\in \hw_{I_0}^\nu(B^R)}
q^{\Db(c)} \card\, \hw_{A_{n-1}}^{\bla}(\Bh(c)).
\end{equation*}
But by \eqref{E:Lit} we have
\begin{align*}
 \card\, \hw_{A_{n-1}}^{\bla}(\Bh(c)) = \sum_{\delta \in
 \Pa_n^\di}
c_{\la,\delta}^\nu.
\end{align*}
By Theorem \ref{T:maxD=A} we have
\begin{align*}
\Xb_{\la,B^R}^\di(q)&=q^{\frac{\abs{R}-\abs{\la}}{|\di|}}
\sum_{\substack{\nu \in \Pa_n \\ \abs{\nu}=\abs{R}}} \sum_{\delta
\in \Pa_n^\di} c_{\la,\delta}^\nu \sum_{c\in \hw_{I_0}^\nu(B^R)}
q^{\Db(c)} \\
&=q^{\frac{\abs{R}-\abs{\la}}{|\di|}} \sum_{\nu\in\Pa_n}
\sum_{\delta\in\Pa_n^\di} c_{\la\delta}^\nu \,\Xb_{\nu,B^R}^\emptyset(q^{\frac{2}{\left| \diamondsuit \right| }}).
\end{align*}
\end{proof}

\subsection{Link with parabolic Lusztig $q$-analogues}

We now give a brief overview on parabolic Lusztig $q$-analogues in type $
A_{n-1},B_{n},C_{n}$ and $D_{n}$. Assume $G=GL_{n},SO_{2n+1},SP_{2n}$ or $
SO_{2n}$. Consider $U$ a subset of $\Sigma _{G}^{+}$ and denote by $\pi _{U}$
the standard parabolic subgroup of $G$ (that is, containing the Borel
subgroup $B_{G})$ defined by $U$. Write $L_{U}$ for the Levi subgroup of the
parabolic $\pi _{U}$ and $\mathfrak{l}_{U}$ its corresponding Lie algebra.
Let $R_{U}$ be the subsystem of roots spanned by $U$ and $R_{U}^{+}$ the
subset of positive roots in $R_{U}$. Then $R_{U}$ and $R_{U}^{+}$ are
respectively the set of roots and the set of positive roots of $\mathfrak{l}
_{U}$.

\noindent The Levi subgroup $L_{U}$ corresponds to the removal, in the
Dynkin diagram of $G$, of the nodes which are not associated to a simple
root belonging to $U$. When $U\neq \Sigma _{G}^{+}$, write
\begin{equation*}
V=\Sigma _{G}^{+}\setminus U=\{\alpha _{j_{1}},\ldots ,\alpha _{j_{p}}\}
\end{equation*}
where for any $k=1,\ldots ,p$, $\alpha _{j_{k}}$ is a simple root of $\Sigma
_{G}^{+}$ and $j_{1}<\cdots <j_{p}$. Then set $l_{1}=j_{1}$, $
l_{k}=j_{k}-j_{k-1}$, $k=2,\ldots ,p$ and $l_{p+1}=n-j_{p}$. The Levi group $
L_{U}$ is isomorphic to a direct product of classical Lie groups determined
by the $(p+1)$-tuple $l_{U}=(l_{1},\ldots ,l_{p+1})$ of nonnegative integers
summing to $n$. Namely, we have
\begin{equation*}
L_{U}\simeq
\begin{cases}
GL_{l_{1}}\times \cdots \times GL_{l_{p}} & \text{if $G=GL_{n}$} \\
GL_{l_{1}}\times \cdots \times GL_{l_{p}}\times SO_{2l_{p+1}+1} & \text{if $
G=SO_{2n+1}$} \\
GL_{l_{1}}\times \cdots \times GL_{l_{p}}\times Sp_{2l_{p+1}} & \text{if $
G=Sp_{2n}$ } \\
GL_{l_{1}}\times \cdots \times GL_{l_{p}}\times SO_{2l_{p+1}} & \text{if $
G=SO_{2n}$.}
\end{cases}
\end{equation*}
Let $\mathcal{P}_{U}=\mathcal{P}_{l_{1}}\times \cdots \times \mathcal{P}
_{l_{p+1}}.\;$Then each $(p+1)$-partition of $\mathcal{P}_{U}$ can be
regarded as a dominant weight for $L_{U}.$ For any $\mathbf{\mu }\in
\mathcal{P}_{U}$, let $V^{L_{U}}(\mathbf{\mu })$ be the finite dimensional
irreducible representation of $L_{U}$ with highest weight $\mathbf{\mu }.$
We denote by $\mu \in \mathbb{N}^{n}$ the concatenation of the parts of the
partitions $\mu ^{(k)},$ $k=1,\ldots ,p.$

\noindent Define the partition function $\mathcal{P}^{U}$ by the formal identity
\begin{equation*}
\prod_{\alpha \in R_{G}^{+}\setminus R_{U}^{+}}\frac{1}{1-e^{\alpha }}
=\sum_{\beta \in \mathbb{Z}^{n}}\mathcal{P}^{U}(\beta )e^{\beta }.
\end{equation*}
Consider $\lambda \in \mathcal{P}_{n}$ and $\mathbf{\mu }\in \mathcal{P}_{U}$
then we have \cite[Theorem 8.2.1]{GW}
\begin{equation*}
\lbrack V^{G}(\lambda ):V^{L_{U}}(\mathbf{\mu })]=\sum_{w\in
W_{G}}(-1)^{\ell (w)}\mathcal{P}^{U}(w\circ \lambda -\mu ).
\end{equation*}
Here $[V^{G}(\lambda ):V^{L_{U}}(\mathbf{\mu })]$ is the branching
multiplicity of the irreducible $L_{U}$-module $V^{L_{U}}(\mathbf{\mu })$ in
the restriction of $V^{G}(\lambda )$ to $L_{U}$.\ For $
GL_{n},SO_{2n+1},SP_{2n},SO_{2n},$ we define the $q$-partition function $
\mathcal{P}_{q}^{U}$ from the formal identity
\begin{equation}
\prod_{\alpha \in R_{G}^{+}\setminus R_{U}^{+}}\frac{1}{1-qe^{\alpha }}
=\sum_{\beta \in \mathbb{Z}^{n}}\mathcal{P}_{q}^{U}(\beta )e^{\beta }.
\label{qpart}
\end{equation}
In type $B_{n},$ we shall also need another partition function.\ Consider
the weight function $L$ on the set $R_{SO_{2n+1}}^{+}$ of positive roots of $
SO_{2n+1}$ such that $L(\alpha )=2$ (resp. $L(\alpha )=1$) on the long
(resp.\ short) roots. The partition function $\mathcal{P}_{q}^{U,L}$ is
defined by
\begin{equation*}
\prod_{\alpha \in R_{SO_{2n+1}}^{+}\setminus R_{U}^{+}}\frac{1}{
1-q^{L(\alpha )}e^{\alpha }}=\sum_{\beta \in \mathbb{Z}^{n}}\mathcal{P}
_{q}^{U,L}(\beta )e^{\beta }.
\end{equation*}

\begin{definition}
\label{def_K}Let $\lambda $ be a partition of $\mathcal{P}_{n}$ and $\mathbf{
\mu }\in \mathcal{P}_{U}$.

\begin{enumerate}
\item  The parabolic Lusztig $q$-analogue $K_{\lambda ,\mu }^{G,U}(q)$ is
the polynomial
\begin{equation}
K_{\lambda ,\mu }^{G,U}(q)=\sum_{w\in W_{G}}(-1)^{\ell (w)}\mathcal{P}
_{q}^{U}(w\circ \lambda -\mu )  \label{def_Kdis}
\end{equation}
where $w\circ \lambda =w(\lambda +\rho _{G})-\rho _{G}.$

\item  The stable parabolic Lusztig $q$-analogue $^{\infty }K_{\lambda
,\mu }^{G,U}(q)$ is the polynomial
\begin{eqnarray}
^{\infty }K_{\lambda ,\mu }^{G,U}(q) &=&\sum_{w\in S_{n}}(-1)^{\ell (w)}
\mathcal{P}_{q}^{U}(w\circ \lambda -\mu )\text{ for }G=GL_{n},SP_{2n},SO_{2n}
\label{def_Kdis_infinite} \\
^{\infty }K_{\lambda ,\mu }^{G,U}(q) &=&\sum_{w\in S_{n}}(-1)^{\ell (w)}
\mathcal{P}_{q}^{U,L}(w\circ \lambda -\mu )\text{ for }G=SO_{2n+1}.
\end{eqnarray}
\end{enumerate}
\end{definition}

\bigskip

\noindent \textbf{Remarks :}

\noindent $\mathrm{(i):}$ When $U=\Sigma _{G}^{+}$, $\mathfrak{l}_{U}$ is
the Cartan subalgebra of $\mathfrak{g}$ and $K_{\lambda ,\mu }^{G,U}(q)$ is
the usual Lusztig $q$-analogue.

\noindent $\mathrm{(ii):}$ The terminology for the polynomials $^{\infty
}K_{\lambda ,\mu }^{G,U}(q)$ is motivated by the following identities proved
in \cite{Lec}

\begin{eqnarray*}
^{\infty }K_{\lambda ,\mu }^{G,U}(q) &=&^{\infty }K_{\lambda +\kappa ,\mu
+\kappa }^{G,U}(q)\text{ for }G=GL_{n},SO_{2n+1},SP_{2n},SO_{2n} \\
^{\infty }K_{\lambda ,\mu }^{G,U}(q) &=&K_{\lambda +k\kappa ,\mu +k\kappa
}^{G,U}(q)\text{ for }G=GL_{n},SP_{2n},SO_{2n}\text{ and }k\text{
sufficiently large}
\end{eqnarray*}
where $\kappa =(1,\ldots ,1)\in \mathcal{P}_{n}.$ In particular $^{\infty
}K_{\lambda ,\mu }^{GL_{n},U}(q)=K_{\lambda ,\mu }^{GL_{n},U}(q).$

\bigskip

The problem of the positivity of the coefficients appearing in the
polynomials $K_{\lambda ,\mu }^{G,U}(q)$ has been barely addressed in the
literature.

\begin{conjecture}
\label{conj}Let $\lambda $ be partition of $\mathcal{P}_{n}$ and $\mathbf{\
\mu }\in \mathcal{P}_{U}$ such that $\mu $ is a partition. Then $K_{\lambda
,\mu }^{G,U}(q)$ has nonnegative coefficients.
\end{conjecture}

We have the following result due to Broer \cite{Bro}.

\begin{theorem}
\label{TH_broer}Let $\lambda $ be a partition of $\mathcal{P}_{n}$ and $
\mathbf{\mu }=(\mu ^{(1)},\ldots ,\mu ^{(p)})$ a dominant weight of $L_{U}$
such that the $\mu ^{(k)}$'s are rectangular partitions of decreasing widths
with $\mu ^{(p)}=0$ when $L_{U}$ is not a direct product of linear groups.\
Then $K_{\lambda ,\mu }^{G,U}(q)$ has nonnegative coefficients.
\end{theorem}

This theorem has been recently extended in \cite{Ha}. Nevertheless, as far
as we are aware, Conjecture \ref{conj} has not been completely proved yet.

\bigskip

Let $\eta =(\eta _{1},\ldots ,\eta _{p})$ be a $p$-tuple of positive
integers summing $n.$ Consider $\lambda \in \mathcal{P}_{n}$ and $\mathbf{\
\mu }=(\mu ^{(1)},\ldots ,\mu ^{(p)})$ a $p$-tuple of partitions such that $
\mu ^{(k)}$ belongs to $\mathcal{P}_{\eta _{k}}$ for any $k=1,\ldots ,p.$
Recall that $\mathfrak{K}_{\mu ^{(1)},\ldots ,\mu ^{(p)}}^{\lambda
,\diamondsuit }$ is the multiplicity of $V^{G}(\lambda )$ in $W^{G}(\mu
^{(1)})\otimes \cdots \otimes W^{G}(\mu ^{(p)}).\ $Write $\mu \in \mathbb{N}
^{n}$ for the $n$-tuple obtained by reading successively the parts of the
partitions $\mu ^{(1)},\ldots ,\mu ^{(p)}$ defining $\mathbf{\mu }$ from
left to right. Let $a$ be the minimal integer such that
\begin{eqnarray}
a &\geq &\frac{\left| \mu \right| -\left| \lambda \right| }{2}\text{ and}
\label{def-lambdahat} \\
\widehat{\lambda } &=&(a-\lambda _{n},\ldots ,a-\lambda _{1})\in \mathbb{N}
^{n},\widehat{\mu }=(a-\mu _{n},\ldots ,a-\mu _{1})\in \mathbb{N}^{n}.
\notag
\end{eqnarray}
Then $\widehat{\lambda}$ is a partition of length $n$.

\noindent For any $k=1,\ldots ,p$, set $\widehat{\eta }_{k}=\eta _{p-k+1}$
and $\widehat{\eta }=(\widehat{\eta }_{1},\ldots ,\widehat{\eta }_{p})$.
Denote by $\widehat{\mathbf{\mu }}=(\widehat{\mu }^{(1)},\ldots ,\widehat{
\mu }^{(p)})$ the $p$-tuple of partitions such that $\widehat{\mu }
^{(1)}=(\mu _{1},\ldots ,\mu _{\widehat{\eta }_{1}})\in \mathcal{P}_{
\widehat{\eta }_{1}}$ and $\widehat{\mu }^{(k)}=(\mu _{\widehat{\eta }
_{1}+\cdots +\widehat{\eta }_{k-1}+1},\ldots ,\mu _{\widehat{\eta }_{1}+\cdots
+\widehat{\eta }_{k}})\in \mathcal{P}_{\widehat{\eta }_{k}}$ for any $
k=2,\ldots ,p.$ The Lie groups $GL_{n},SO_{2n+1},SP_{2n},SO_{2n}$ contain
Levi subgroups $L_{U}$ isomorphic to $GL_{\widehat{\eta }_{1}}\times \cdots
\times GL_{\widehat{\eta }_{p}}.\;$With the above terminology, the
corresponding subset of simple roots is
\begin{equation}
U=\{0<\alpha _{i}<\widehat{\eta }_{1}\}\cup _{1\leq k\leq p-1}\{\alpha
_{i}\mid \widehat{\eta }_{1}+\cdots +\widehat{\eta }_{k}<i<\widehat{\eta }
_{1}+\cdots +\widehat{\eta }_{k+1}\}.  \label{I_from_mu}
\end{equation}
In particular when $G=SO_{2n+1},SP_{2n}$ or $SO_{2n},$ $U$ never contains
the simple root $\alpha _{n}.$

\begin{example}
Consider $\mu ^{(1)}=(5,4,4),$ $\mu ^{(2)}=(6,3,2)$ and $\mu ^{(3)}=(4,3).\;$
Take $\lambda =(4,4,3,2,2,1,0,0).$ Then $a=8,$ $\widehat{\mu }^{(1)}=(5,4),$
$\widehat{\mu }^{(2)}=(6,5,2),$ $\widehat{\mu }^{(3)}=(4,4,3)$ and $\widehat{
\lambda }=(8,8,7,6,6,5,4,4).$
\end{example}

\noindent The coefficients $\mathfrak{K}_{\mu ^{(1)},\ldots ,\mu
^{(p)}}^{\lambda ,\diamondsuit }$ defined in (\ref{E:KRmult}) can in fact
be regarded as branching coefficients corresponding to the restriction to
Levi subgroup isomorphic to a direct product of linear groups. The following
duality was established in \cite{Lec} .

\begin{proposition}
\label{prop_dual}With the previous notation for $\widehat{\lambda },\widehat{
\mu }$ we have for $G=SO_{2n+1},SP_{2n}$ and $SO_{2n}$ $\mathfrak{K}_{\mu
^{(1)},\ldots ,\mu ^{(p)}}^{\lambda ,\diamondsuit }=^{\infty }K_{\widehat{
\lambda },\widehat{\mu }}^{G,U}(1).$
\end{proposition}

We then define for $G=SO_{2n+1},SP_{2n}$ and $SO_{2n}$ the $q$-analogue $
\mathfrak{K}_{\mu ^{(1)},\ldots ,\mu ^{(p)}}^{\lambda ,\diamondsuit }(q)$ of
$\mathfrak{K}_{\mu ^{(1)},\ldots ,\mu ^{(p)}}^{\lambda ,\diamondsuit }$ by
setting
\begin{equation*}
\mathfrak{K}_{\mu ^{(1)},\ldots ,\mu ^{(p)}}^{\lambda ,\diamondsuit
}(q)=^{\infty }K_{\widehat{\lambda },\widehat{\mu }}^{G,U}(q).
\end{equation*}

\begin{theorem}
\label{th_qdual2}\cite{Lec}

\begin{enumerate}
\item  We have the decomposition
\begin{equation}
\mathfrak{K}_{\mu ^{(1)},\ldots ,\mu ^{(p)}}^{\lambda ,\diamondsuit }(q)=q^{
\tfrac{\left| \mu \right| -\left| \lambda \right| }{\left| \diamondsuit
\right| }}\sum_{\nu \in \mathcal{P}_{n}}\ \sum_{\delta \in \mathcal{P}
_{n}^{\diamondsuit }}c_{\lambda ,\delta }^{\nu }\ K_{\nu ,\mu
}^{GL_{n},U}(q^{\frac{2}{\left| \diamondsuit \right| }}).  \label{Dfrak}
\end{equation}

\item  The polynomial $\mathfrak{K}_{\mu ^{(1)},\ldots ,\mu ^{(p)}}^{\lambda
,\diamondsuit }(q)$ has nonnegative integer coefficients when the $\mu ^{(k)}
$'s are rectangular partitions of decreasing widths
\end{enumerate}
\end{theorem}

\noindent \textbf{Remarks :}

\noindent $\mathrm{(i):}$ \textbf{\ }Assertion 2 of the previous theorem
follows directly from Theorem \ref{TH_broer} for $G=SP_{2n}$ and $SO_{2n}.$
For $G=SO_{2n+1},$ we have to use Assertion 1 and Theorem \ref{TH_broer} for
$G=GL_{n}$.

\noindent $\mathrm{(ii):}$ Proposition \ref{prop_dual} generalizes a similar
duality result in type $A_{n-1}.$ For $(\mu ^{(1)},\ldots ,\mu ^{(p)})$ a $p$
-tuple of partitions, we have $K_{\nu ,\mu }^{GL_{n},U}(1)=c_{\mu
^{(1)},\ldots ,\mu ^{(p)}}^{\nu }$ where $c_{\mu ^{(1)},\ldots ,\mu
^{(p)}}^{\nu }$ is the multiplicity of $V^{GL_{n}}(\nu )$ in $V^{GL_{n}}(\mu
^{(1)})\otimes \cdots \otimes V^{GL_{n}}(\mu ^{(p)}).$ We set for completion
\begin{equation}
\mathfrak{K}_{\mu ^{(1)},\ldots ,\mu ^{(p)}}^{\lambda ,\emptyset }(q)=K_{\nu
,\mu }^{GL_{n},U}(q).  \label{dual_A}
\end{equation}

Recall the following theorem connecting one-dimensional sums in affine type 
$A_{n}^{(1)}$ with parabolic Lusztig $q$-analogues for $GL_{n}$.

\begin{theorem}
\cite{Sh}\label{th-shimo} Let $B$ be the tensor product of type $A_{n}^{(1)}$ KR
crystals associated to the $p$-tuple of rectangular partitions $
(R^{(1)},...,R^{(p)})$ of decreasing widths.\ Then for any partition $
\lambda $ in $\mathcal{P}_{n}$, we have
\begin{equation*}
\overline{X}_{\lambda ,B}^{\emptyset }(q)=q^{||R||}\mathfrak{K}
_{R^{(1)},\ldots ,R^{(p)}}^{\lambda ,\emptyset }(q^{-1})=q^{||R||}\ K_{\nu
,\mu }^{GL_{n},U}(q^{-1})
\end{equation*}
where $U$ is defined in (\ref{I_from_mu}) and
\begin{align} \label{E:||R||}
  ||R|| = \sum_{1\le i<j\le p} |R_i \cap R_j|.
\end{align}
\end{theorem}

\begin{theorem}
\label{Th_X=K}Let $B$ be a tensor product of $p$ KR crystals. Assume the
widths of the rectangles $R^{(1)},\ldots ,R^{(p)}$ associated to $B$ are
decreasing and the large rank hypothesis is satisfied. Then, for any $
\lambda \in \mathcal{P}_{n}$
\begin{equation*}
\overline{X}_{\lambda ,B}^{\diamondsuit }(q)=q^{\frac{2(||R||+|R|-|\lambda |)
}{\left| \diamondsuit \right| }}\ \mathfrak{K}_{R^{(1)},\ldots
,R^{(p)}}^{\lambda ,\diamondsuit }(q^{-1})=q^{\frac{2(||R||+|R|-|\lambda |)}{
\left| \diamondsuit \right| }}\,{}^{\infty }K_{\widehat{\lambda },\widehat{
\mu }}^{G,U}(q^{-1})
\end{equation*}
where $U$ is defined in (\ref{I_from_mu}) and $||R||$ in \eqref{E:||R||}.
\end{theorem}

\begin{proof}
This follows from Theorems \ref{th-shimo}, \ref{th_qdual2} and \ref{Th_dec_X}.
\end{proof}

\begin{theorem}
\label{TH_transpose}Let $B$ be a tensor product of KR crystals. Assume the
large rank hypothesis is satisfied. Then, for any $\lambda \in \mathcal{P}
_{n}$
\begin{equation*}
\overline{X}_{\lambda ^{t},B^{t}}^{\diamondsuit ^{t}}(q)=q^{\frac{
2(||R||+|R|-|\lambda |)}{\left| \diamondsuit \right| }}\ \overline{X}
_{\lambda ,B}^{\diamondsuit }(q^{-1}).
\end{equation*}
\end{theorem}

\begin{proof}
For $\diamondsuit =\emptyset ,$ the equality $\frak{K}_{(R^{(1)})^{t},\ldots
,(R^{(p)})^{t}}^{\lambda ^{t},\emptyset }(q)=q^{\left\| B\right\| }\frak{K}
_{R^{(1)},\ldots ,R^{(p)}}^{\lambda ,\diamondsuit }(q^{-1})$ was proved in
\cite{KS}. By using Theorem \ref{th-shimo}, one obtains $\overline{X}
_{\lambda ^{t},B^{t}}^{\emptyset }(q)=q^{\left\| B\right\| }\overline{X}
_{\lambda ,B}^{\emptyset }(q^{-1}).$ Theorem \ref{Th_dec_X} gives
\begin{equation*}
\overline{X}_{\lambda ^{t},B^{t}}^{\diamondsuit ^{t}}(q)=q^{\frac{\left|
B\right| -\left| \lambda \right| }{\left| \diamondsuit \right| }}\sum_{\nu
\in \mathcal{P}_{n}}\sum_{\delta \in \mathcal{P}_{n}^{\diamondsuit
}}c_{\lambda ^{t},\delta ^{t}}^{\nu ^{t}}\ \overline{X}_{\nu
^{t},B^{t}}^{\emptyset }(q^{\frac{2}{\left| \diamondsuit \right| }}).
\end{equation*}
But $c_{\lambda ^{t},\delta ^{t}}^{\nu ^{t}}=c_{\lambda ,\delta }^{\nu }.$
Thus by using the previous identity for $\diamondsuit =\emptyset $
\begin{equation*}
\overline{X}_{\lambda ^{t},B^{t}}^{\diamondsuit ^{t}}(q)=q^{\frac{2\left\|
B\right\| }{\left| \diamondsuit \right| }+\frac{\left| B\right| -\left|
\lambda \right| }{\left| \diamondsuit \right| }}\sum_{\nu \in \mathcal{P}
_{n}}\sum_{\delta \in \mathcal{P}_{n}^{\diamondsuit }}c_{\lambda ,\delta
}^{\nu }\ \overline{X}_{\nu ,B}^{\emptyset }(q^{-\frac{2}{\left|
\diamondsuit \right| }})=q^{\frac{2(||R||+|R|-|\lambda |)}{\left|
\diamondsuit \right| }}\ \overline{X}_{\lambda ,B}^{\diamondsuit }(q^{-1}).
\end{equation*}
\end{proof}

\section{Splitting preserves energy}
\label{S:energy}

In this section we assume $\geh$ is of affine type with $\geh_0$ of type $B_n$, $C_n$, or $D_n$.

For $B\in\CC(\geh)$ we define the opposite grading $D:B\to\Z$ (the intrinsic energy)
to $\Db_B$. We show in Theorem \ref{T:Dsplit} that it is invariant under the
row-splitting map $S$. The normalization of $D$ is somewhat subtle.
For example, $\Db$ is nonnegative with minimum value zero, while $D$
may be negative. Also, if $B_1,B_2\in\CC$
are both tensor products of KR crystals, then the formula relating $H_{B_1,B_2}$
and $\Hb_{B_1,B_2}$, requires knowledge of all the KR tensor factors in $B_1$ and $B_2$.

For this reason, instead of an inductive definition analogous to that of $\Db_B$
we make the following definitions.

For $B_i=B^{R_i}=B^{r_i,s_i}\in\CCinf(\geh)$ for $i\in\{1,2\}$ we define
\begin{align}\label{E:H2}
  H_{B_1,B_2}(b_1\otimes b_2) = |R_1\cap R_2| - \Hb_{B_1,B_2}(b_1\otimes b_2)
\end{align}
for $b_1\in B_1$ and $b_2\in B_2$, where $|R_1\cap R_2| = \min(r_1,r_2)\min(s_1,s_2)$
is the number of cells in the rectangular partition given by the intersection of the Young
diagrams of the rectangular partitions $R_1$ and $R_2$. We define
\begin{align} \label{E:D1}
  D_{B^{R_1}}(b) = - \Db_{B^{R_1}}(b)\qquad\text{for $b\in B^{R_1}$.}
\end{align}
Analogous to \eqref{E:DbCC} we define
\begin{align} \label{E:DCC}
  D_{B^R}(b) = \sum_{i=1}^p D_{B^{R_i}}(b_i^{(1)}) +
  \sum_{1\le i<j\le p} H_{B_i,B_j}(b_i\otimes b_j^{(i+1)}).
\end{align}
We make the same definitions \eqref{E:H2}, \eqref{E:D1}, and \eqref{E:DCC}
for type $A_{n-1}^{(1)}$ also. Then \eqref{E:D1} reads $D_{B^{R_1}_A}=-\Db_{B^{R_1}_A}\equiv0$.
Using \eqref{E:DbCC} we deduce that
\begin{align}
\label{E:DDb}
  D_{B^R}(b) &= ||R|| - \Db_{B^R}(b) &\qquad&\text{for $b\in B^R$} \\
\label{E:DDbA}
  D_{B^R_A}(b) &= ||R|| - \Db_{B^R_A}(b) &\qquad&\text{for $b\in B^R_A$}
\end{align}
where $||R||$ is defined in \eqref{E:||R||}.
$D_{B^R_A}$ has nonnegative values with minimum value $0$ in the large rank case,
while $D_{B^R}$ has negative values in general.

\begin{proposition} \label{P:DAD} For any sequence of rectangles $R$ such that
$B^{R_i}\in\CCinf(\geh)$,
\begin{align}\label{E:DAD}
  D_{B_A^R} = D_{B^R} \circ i_A^R.
\end{align}
\end{proposition}
\begin{proof} As in the proof of Proposition \ref{P:DmaxB},
we reduce to checking the case $R=(R_1)$
and
\begin{align} \label{E:HAD}
  H_{B^{R_1}_A,B^{R_2}_A} = H_{B^{R_1},B^{R_2}} \circ i_A^{R_1,R_2}.
\end{align}
For \eqref{E:DAD} for $R=(R_1)$ we see that both sides yield zero by definition.
Equation \eqref{E:HAD} follows from Lemma \ref{L:embedH} and the definitions.
\end{proof}

Let $S_A:B^R_A \to B^{S(R)}_A$ be type $A_{n-1}^{(1)}$ row-splitting
of the first tensor factor and $\spl_A:B^R_A\to B^{\rows(R)}_A$
the type A complete splitting into rows (split first factor if possible
and use R-matrices).

\begin{proposition} \label{P:splitembed}
With $R$ such that $B^{R_i}\in\CCinf(\geh)$,
\begin{equation}
\label{E:splitembed}
\begin{split}
i_A^{S(R)} \circ S_A &= S \circ i_A^R \\
i_A^{\rows(R)} \circ \spl_A &= \spl \circ i_A^R.
\end{split}
\end{equation}
\end{proposition}
\begin{proof}
By Lemma \ref{L:embedR} and the definitions, we may reduce the statement on $\spl$
to that of $S$ and check $S$ only in the
single tensor factor case $B=B^R=B^{r,s}$. In this case $\tops(\max(B))$ is the $A_{n-1}$-component
of $b(r,s,(s^r))\in B(s^r)\subset B^{r,s}$, $\tops(\max(B))$ consists of type $B_n$, $C_n$, or $D_n$
KN tableaux of shape $(s^r)$ with no barred letters, and \eqref{E:splitembed} is easily verified.
\end{proof}

\begin{theorem} \label{T:Dsplit} For $B=B^R\in\CCinf(\geh^\di)$,
\begin{align}
\label{E:Dsplit}
D_{B^R} &= D_{B^{S(R)}} \circ S \\
\label{E:Dsplitall}
D_{B^R} &= D_{B^{\rows(R)}} \circ \spl.
\end{align}
\end{theorem}
\begin{proof} We need only prove
\eqref{E:Dsplit}. Since energy functions are constant on $I_0$-components, it suffices to
check \eqref{E:Dsplit} on $b\in \tops(B^R)$. We have
\begin{equation*}
D(b)=D(\sigma(b))-\frac{\abs{R}-\abs{\la(b)}}{|\di|}
\end{equation*}
by Theorem \ref{Th_SR} and \eqref{E:DDb}.
Since $S$ is an embedding of $I_0$-crystals, $S(b)\in\tops(B^{S(R)})$.
Applying the previous argument to $S(b)$ we have
\begin{equation*}
D(S(b))=D(\sigma (S(b)))-\frac{\abs{S(R)}-\abs{\la(S(b))}}{|\di|}.
\end{equation*}
But $\abs{S(R)}=\abs{R}$ and $\abs{\la(S(b))}=\abs{\la(b)}$
since $S$ is an embedding of $I_0$-crystals.
So it suffices to prove that $D(\sigma (b))=D(\sigma (S(b)))$. By
Proposition \ref{P:Ssigmacommute},
this is equivalent to $D(\sigma (b))=D(S(\sigma (b)))$. So
by Theorem \ref{Th_corespondence} we are reduced to prove the equality $D(c)=D(S(c))$ for any
$c\in\hmax(B^R)$. Since $D$ is constant on $I_0$-components we need only show
$D(c)=D(S(c))$ for $c\in \hw_{I_0}(\max(B^R))=\hw_{A_{n-1}}(\tops(\max(B^R)))$.
By Proposition \ref{P:DAD} applied for $R$ and $S(R)$, the desired equality reduces
to the identity $D_A(a)=D_A(S_A(a))$ for any $a\in B_A^R$ which
was established in \cite{Sh:crystal}.
\end{proof}

\begin{remark} In the statement of Theorem \ref{T:Dsplit}, it should be unnecessary
to assume that $\geh$ is reversible and $B^R\in\CCinf(\geh)$.
However for $S$ to make sense there cannot be spin nodes in the $R_i$.
\end{remark}

\appendix

\section{Proofs for Section \ref{S:hat}}
\label{A:hatcrystal}

\subsection{Proof of Proposition \ref{P:Bhatprop}}

\begin{proof}[Proof of Proposition \ref{P:Bhatprop}] Let $b\in\Bhp^\di(\nu,\delta)$ for $\delta\in\Pa_n^\di$.
Observe that the letters of the canonical subtableau $C_\delta^\di$
collectively do not affect any $A_{n-1}$-string. Now $b|^{\nu\setminus\delta}$
is a semistandard tableau in the alphabet $\{\bar{n},\dotsc,\bar{1}\}$. It is well-known
that the set of skew tableaux of a fixed shape, form an $A_{n-1}$-crystal. This proves 1.

For Assertion 2, based on the above observations,
$b$ is $A_{n-1}$-highest weight, if and only if $b|^{\nu/\delta}$
is $A_{n-1}$-highest weight as an element of the type $A_{n-1}$ skew tableau crystal. But it is well-known that
such a skew tableau is $A_{n-1}$-highest weight if and only if its row-reading word is Yamanouchi.
Finally, since the tableau has letters in $\{\bar{n},\dotsc,\bar{1}\}$, if it is $A_{n-1}$-highest weight,
then its weight must have the form $\bla$ for some $\la\in\Pa_n$.

For Assertion 3, suppose $b$ admits $f_n$.
\begin{enumerate}
\item $\di=(1,1)$: The application of $f_n$ to $b$, changes an $n$ (which by the signature
rule, must be in a corner cell of $\delta$) to a $\overline{n-1}$. Since every $n$ sits atop a $\bar{n}$,
Assertion 3 follows.
\item $\di=(1)$: The application of $f_n$ to $b$ changes some $0$ to $\bar{n}$
or some $n$ to $0$ (say in row $i$). The tableau $f_n(b)$ contains
$C_{\dm}^{(1)}$ where $\dm\in\Pa_n^{(1)}$ is obtained from $\delta$ by removing a cell in row $i$.
The only way that $f_n(b)$ is not in $\Bhp^{(1)}(\nu,\dm)$ is if
$f_n(b)|^{\nu/\dm}$ contains two letters $\bar{n}$ in the same column,
either because the changed letter became $\bar{n}$ and now lies beneath another $\bar{n}$,
or because in $b$ there was a pair of letters $\bar{n}$ atop each other but one was in $\delta$
and the other not in $\delta$, but now in $f_n(b)$ both are outside $\dm$.
However the assumption that $\delta_{i+1}\le\delta_i$
and the signature rule, imply that this cannot occur.
\item $\di=(2)$: The application of $f_n$ to $b$ changes some $n$ to $\bar{n}$
(say in the $i$-th row). The tableau $f_n(b)$ contains $C_{\dm}^{(2)}$ where
$\dm\in\Pa_n^{(2)}$ is obtained from $\delta$ by removing two cells in row $i$.
Similarly to the case $\di=(1)$, one may deduce Assertion 3.
\end{enumerate}

We prove Assertions 4 and 5 by induction on
$\abs{\delta}=\abs{\nu}-\abs{\la}$. Equivalently we find a sequence
$\mba$ of indices in $I_0$ such that $f_\mba(b)=\rowtab(b^\nu)$. By
Assertion 1 we may assume $b$ is a $A_{n-1}$-lowest weight vertex.

If $\abs{\delta}=0$ then $b=\rowtab(b^\nu)$ and the empty sequence
works. Suppose $\delta\ne\emptyset$.
Since $b$ is $A_{n-1}$-lowest weight and $\nu\in\Pa_n^\infty$
the skew tableau $b|^{\nu\setminus\delta}$ admits no
$A_{n-1}$-lowering operator and contains letters in $\{\overline{n-2},\dotsc,\bar{2},\bar{1}\}$.
So the letters outside $b|^\delta=C_\delta^\di$ are irrelevant for the $n$-signature.
In the various cases we see that $f_n(b)\in \Bhp^\di(\nu,\dm)$ where $\dm\in\Pa_n^\di$
is obtained from $\delta$ in the same way that $\lam$ is obtained from $\la$ in Lemma \ref{L:npath}.
Induction completes the proof.
\end{proof}

\begin{example} $\la=(4,3,3,1,1)\in\Pa(5,4)$ since $\delta=(3,3,1,1)\in\Pvd$.
\begin{equation*}
\rowtab(\ba(5,4,\la)) =
\begin{tabular}{|c|c|c|c|}
\hline
$\overline{n}$ & \lower.5mm\hbox{$\overline{n-2}$} & \lower.5mm\hbox{$\overline{n-2}$}
& \lower.5mm\hbox{$\overline{n-4}$} \\ \hline
$\mathbf{n}$ & \lower.5mm\hbox{$\overline{n-1}$} & \lower.5mm\hbox{$\overline{n-1}$}
& \lower.5mm\hbox{$\overline{n-3}$} \\ \hline
$\mathbf{\overline{n}}$ & $\overline{n}$ & $\overline{n}$ & \lower.5mm\hbox{$\overline{n-2}$} \\ \hline
$\mathbf{n}$ & $\mathbf{n}$ & $\mathbf{n}$ & \lower.5mm\hbox{$\overline{n-1}$} \\ \hline
$\mathbf{\overline{n}}$ & $\mathbf{\overline{n}}$ & $\mathbf{\overline{n}}$ & $\overline{n}$ \\ \hline
\end{tabular}
\ \text{ and }\rowtab(\bamin^{(1,1)}(5,4))=
\begin{tabular}{|c|c|c|c|}
\hline
$\overline{n}$ & $\overline{n}$ & $\overline{n}$ & $\overline{n}$ \\ \hline
$n$ & $n$ & $n$ & $n$ \\ \hline
$\overline{n}$ & $\overline{n}$ & $\overline{n}$ & $\overline{n}$ \\ \hline
$n$ & $n$ & $n$ & $n$ \\ \hline
$\overline{n}$ & $\overline{n}$ & $\overline{n}$ & $\overline{n}$ \\ \hline
\end{tabular}
\ .
\end{equation*}
\end{example}

\begin{proof}[Proof of Lemma \ref{L:npath}]
Let $\rowtab=\rowtab_{(s^r)}$, $b'=\rowtab(\ba(r,s,\lam))$ and $b=\rowtab(\ba(r,s,\la))$.
It suffices to show
\begin{align}
\label{E:holdhwv}
\veps_n(f_{\ta'(h)}(b')) &\ge \ell \\
\label{E:canlower}
\vphi_n(f_{\ta'(h)}(b')) &> 0 \\
\label{E:removedi}
  u_{\ell\La_n} \otimes b &= f_{\mathbf{\ta(h)}}(u_{\ell\La_n} \otimes b').
\end{align}
Let $\dm\in\Pvd$ be the complementary
partition to $\lam$ within $(s^r)$. We will need to keep track of certain letters that
may contribute to the $n$-signature.

Suppose $\di=(1,1)$. By Proposition \ref{P:Bhatprop}(2), the restriction of
$b^{\prime}$ to the skew shape $(s^{r})\setminus\dm$, has the letter $\bar{n}$
at the bottom of each column and a letter $\overline{n-1}$ atop the letter
$\bar{n}$ if it fits into $(s^{r})$. We may think that every $n$ not in the
top row is paired (in the $(n-1)$-signature) with the $\bar{n}$ sitting atop
it. The $n$'s in the top row (which may occur if $r$ is even) are unpaired and
occur at the end of the rowwise $n$-signature reading. There are $s$ unpaired
letters $\bar{n}$ in the bottom row, and an unpaired $\overline{n-1}$ in each
column of $\delta$ that is not of maximum height $2\lfloor r/2\rfloor$. We now
apply $f_{\ta^{\prime}(h)}$ to $b^{\prime}$; call the result $b^{\prime\prime
}$. It only changes letters at the top of the $p$-th column from the right,
from (reading down) $\overline{n-h+3},\dotsc,\overline{n-1},\overline{n}$, to
$\overline{n-h+1},\dotsc,\overline{n-2}$. The bottom row still consists of $s$
copies of $\overline{n}$ which occur at the beginning of the $n$-signature, so
\eqref{E:holdhwv} holds. The dominant elements in the $n$-signature of
$b^{\prime\prime}$ are the unpaired letters $n$ in the top row if $r$ is even,
and the copy of $n$ in the active column, since the relevant letters changed
from $\overline{n-1},\overline{n},n$ to $\overline{n-3},\overline{n-2},n$.
Therefore \eqref{E:canlower} holds. Applying $f_{n}$ to $b^{\prime\prime}$
changes the $n$ in the active column to the letter $\overline{n-1}$, with
final result $\rowtab(\ba(r,s,\la))$.

Now assume $\di=(2)$. Similarly, the restriction of $b^{\prime}$ to the skew
shape $(s^{r})\setminus\dm$, has the letter $\bar{n}$ at the bottom of each
column and a letter $\bar{n}$ or $\overline{n-1}$ atop the letter $\bar{n}$ if
it fits into $(s^{r})$. Moreover, each letter $n$ is paired with a letter
$\bar{n}$ in the $(n-1)$-signature of $b^{\prime}.\;$Therefore, $b^{\prime
\prime}=f_{\ta^{\prime}(h)}(b^{\prime})$ is obtained by changing the letters
at the top of the $p$-th and $(p+1)$-th columns from (reading down)
$\overline{n-h+2},\dotsc,\overline{n-1},\overline{n}$, to $\overline
{n-h+1},\dotsc,\overline{n-2},\overline{n-1}$. In $b^{\prime}$ and
$b^{\prime\prime}$ the bottom contains at least $\left\lceil
\genfrac{.}{.}{0pt}{}{s}{2}
\right\rceil $ letters $\bar{n}$ which are unpaired in the $n$-signature. Thus
\eqref{E:holdhwv} holds. In the columns $p$ and $p+1$, the letters
$\overline{n}$ are changed in $\overline{n-1}.$ Therefore \eqref{E:canlower}
holds. Applying $f_{n}$ to $b^{\prime\prime}$ changes the $n$ in the active
column $p$ to the letter $\overline{n}$, with final result $\rowtab
(\ba(r,s,\la))$ as desired.

The case $\di=(1)$ is similar.
\end{proof}

\section{Proofs for Section \ref{S:sigmasec}}
\label{A:sigma}

In this appendix we assume $\di=(1,1)$ and $\geh^\di=D_n^{(1)}$.

\subsection{Reduction to relation on automorphisms of $B^{r,s}$}
Our first reduction for proving Proposition \ref{P:KRsigmaall}
in the case $\di=(1,1)$ is to rephrase it in terms of a relation among
various automorphisms on $B^{r,s}$.
Recall the automorphism $\varsigma$ on $B^{r,s}$ from Section \ref{SS:KR}.

Let $\varsigma'\in\Aut(D_n^{(1)})$ be defined by the permutation of $I^s$ given by
$(n-1,n)$. $\varsigma'$ is also not a special automorphism. It coincides with
$*\in\Aut(D_n^{(1)})$ if $n$ is odd. There is a unique bijection $\varsigma':B^{r,s}\to B^{\varsigma'(r),s}$
\begin{align} \label{E:tauprimedef}
  \varsigma' e_i = e_{\varsigma'(i)} \varsigma'\qquad\text{for all $i\in I$.}
\end{align}
It is explicitly given by exchanging $n$'s with $\bar{n}$'s in
KN tableaux. For $r\in I_0$ nonspin, $\varsigma'$ is an involution on
$B^{r,s}$.

\begin{lemma} \label{L:reduce}
If
\begin{align} \label{E:suff}
  \varsigma' \sigma = \sigma \varsigma
\end{align}
holds on $B^{r,s}$, then Proposition \ref{P:KRsigmaall} holds.
\end{lemma}
\begin{proof} We have
\begin{align*}
  \varsigma' \sigma e_0 = \sigma \varsigma e_0 = \sigma e_1 \varsigma = e_{n-1} \sigma \varsigma = e_{n-1} \varsigma' \sigma = \varsigma' e_n \sigma.
\end{align*}
Applying the involution $\varsigma'$, we have $\sigma e_0 = e_n\sigma$ on
$B^{r,s}$. By Proposition \ref{P:sigma} it follows that $\sigma$
satisfies \eqref{E:sigmaprop} as required.
\end{proof}

\subsection{Rule for $\rowtab(\sigma(\Phi(P)))$ for a $\pm$-diagram $P$}
We give a rule for $\rowtab(\sigma(\Phi(P)))$ for any $\pm$-diagram $P$.

\medskip \noindent\textbf{Rule} \vspace{-2mm}

\begin{itemize}
\item[1.]  Rotate $P$ 180 degrees and place it in the $r\times s$ rectangle
so that the NE corners of the rotated $P$ and the rectangle coincide.

\item[2.]  Fill each column of the inner shape of $P$ by sequences of the form
$\overline{k},\ldots ,\overline{n-2},\overline{n-1}$ reading from the top, place
$\overline{n}$ in each node where $+$ is situated, and fill all columns from top to bottom in
the rest of the rectangle by sequences of the form $n,\overline{n},n,
\overline{n},\ldots$, always starting with $n$.

\item[3.]  In each row perform the following substitution. Suppose there are $
k_{+}$ $n$'s and $k_{-}$ $\overline{n}$'s in the row. Then replace them with
$(n-1)^{k_{-}}n^{k_{+}-k_{-}}(\overline{n-1})^{k_{-}}$ if $k_{+}\geq k_{-}$,
and $(n-1)^{k_{+}}\overline{n}^{k_{-}-k_{+}}(\overline{n-1})^{k_{+}}$
otherwise.
\end{itemize}

\begin{example} \label{X:Rule}
Let $n=9,r=6,s=7$.
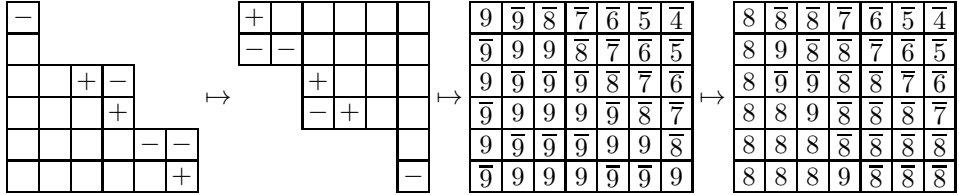
\begin{equation*}
\vcenter{
\tableau[sby]{ - \\ \kuu \\ \kuu &\kuu&+&- \\ \kuu&\kuu&\kuu&+ \\ \kuu&\kuu&\kuu&\kuu&-&- \\ \kuu&\kuu&\kuu&\kuu&\kuu&+}
}
\mapsto
\vcenter{
\tableau[sby]{ +&\kuu&\kuu&\kuu&\kuu&\kuu \\ -&-&\kuu&\kuu&\kuu&\kuu \\ \bl&\bl&+&\kuu&\kuu&\kuu \\
\bl&\bl&-&+&\kuu&\kuu \\ \bl&\bl&\bl&\bl&\bl&\kuu \\ \bl&\bl&\bl&\bl&\bl&-}
}
\mapsto
\vcenter{
\tableau[sby]{
9&\kyu&\hachi&\nana&\roku&\go&\yon\\
\kyu&9&9&\hachi&\nana&\roku&\go\\
9&\kyu&\kyu&\kyu&\hachi&\nana&\roku\\
\kyu&9&9&9&\kyu&\hachi&\nana\\
9&\kyu&\kyu&\kyu&9&9&\hachi\\
\kyu&9&9&9&\kyu&\kyu&9}
}
\mapsto
\vcenter{
\tableau[sby]{
8&\hachi&\hachi&\nana&\roku&\go&\yon \\
8&9&\hachi&\hachi&\nana&\roku&\go\\
8&\kyu&\kyu&\hachi&\hachi&\nana&\roku\\
8&8&9&\hachi&\hachi&\hachi&\nana\\
8&8&8&\hachi&\hachi&\hachi&\hachi\\
8&8&8&9&\hachi&\hachi&\hachi
}
}
\end{equation*}
\end{example}

\begin{proposition}
\label{prop:sigma tau} For any $\pm $-diagram $P$ for $B^{r,s}$, the above rule gives $\rowtab(\sigma (\Phi (P)))$.
\end{proposition}

The proof of this key technical result is given in the following subsection.
We use it to finish up the proofs of Section \ref{S:sigmasec}.

\begin{proposition}
For any $b\in \hw_J(B^{r,s})$, we have $\sigma \varsigma (b)=\varsigma' \sigma (b)$.
\end{proposition}

\begin{proof}
Compare $P$ and $\mathfrak{S}(P)$ where $\SSS$ is the involution on $\pm$-diagrams
corresponding to the automorphism $\varsigma$ on $B^{r,s}$.
The inner shapes are the same and in each
column, if there is $+$ in $P$, then there is no $+$ in $\mathfrak{S}(P)$,
and vice versa. Therefore, at the moment when Rule 2 is finished, the number
of $n$'s and $\overline{n}$'s in each row are switched for $P$ and $
\mathfrak{S}(P)$. Hence, we have $\sigma \Phi (\mathfrak{S}(P))=\varsigma
^{\prime }\sigma \Phi (P)$. This is what we needed to show.
\end{proof}

\begin{proof}[Proof of Proposition \ref{P:KRsigmaall}]
Let $b\in B^{r,s}$. Let $b^\circ = \hw_J(b)$ and let $\bs=(i_1,i_2,\dotsc)$ be a finite sequence in $J$
such that $b=f_{\bs}(b^\circ)$. Then
\begin{align*}
\sigma \varsigma (b)& =f_{\sigma \varsigma (\bs)} \sigma \varsigma (b^\circ), \\
\varsigma'\sigma (b)& =f_{\varsigma'\sigma (\bs)}\varsigma'\sigma (b^\circ),
\end{align*}
where the Dynkin automorphisms $\varsigma$, $\varsigma'$, and $\sigma$ act on sequences of Dynkin nodes in the obvious way.
Since $\sigma \varsigma (\bs)=\varsigma' \sigma (\bs)$
and $\sigma \varsigma (b^\circ)=\varsigma'\sigma (b^\circ)$ by Lemma \ref{L:sigma},
we obtain $\sigma \varsigma (b)=\varsigma'\sigma (b)$.
\end{proof}

\subsection{Proof of Proposition \ref{prop:sigma tau}}

We need some notation. Let $r'=\lfloor
r/2\rfloor$. For $\la\in \Pa(r,s)$ and $0\le j\le r'$, define $c_j$ by
\begin{equation*}
\la=\sum_{j=0}^{r'}(c_j-c_{j+1})\omega_{r-2j}
\end{equation*}
with $c_0=s$ and $c_{r'+1}=0$. Then a sequence $(c_1,c_2,\ldots,c_{r^{
\prime}})$ such that $s\ge c_1\ge c_2\ge \cdots\ge c_{r^{\prime}}\ge0$ is in
one-to-one correspondence with the $I_0$-highest element $b(r,s,\la)\in B^{r,s}$.

It remains to prove Proposition~\ref{prop:sigma tau}.
First we assume the $\pm$-diagram $P$ has no column for
which a + can be added. Let $\la$ be the outer shape of $P$, $c_i^-$ the
number of columns that has a $-$ at height $i$ in $P$. Set
$a_i=\sum_{j=1}^ic_j^-$. By \cite[Prop. 2.2]{OSa} we have
\begin{equation*}
\Phi(P)=f_{(1^{a_r},\dotsc,(n-1)^{a_r},n^{a_r},(n-2)^{a_r},\dotsc,
(r+1)^{a_r},r^{a_r},\dotsc,2^{a_2},1^{a_1})}b(r,s,\la).
\end{equation*}
(The notation $a_i$ in \cite[Prop. 2.2]{OSa}, is equal to $\sum_{j=1}^rc_j^-$.) Hence, by Lemma \ref{L:sigma}
and the definition of $\sigma$ we obtain
\begin{equation*}
\sigma(\Phi(P))=f_{((n-1)^{a_r},\dotsc,1^{a_r},0^{a_r},2^{a_r},\dotsc,(n-2)^{a_2},(n-1)^{a_1})}\ba(r,s,\la).
\end{equation*}

\begin{lemma}
\label{lem:sig1} The row tableau
\begin{equation*}
t_1 =
f_{(2^{a_r},\dotsc,(n-r-1)^{a_r},(n-r)^{a_r},\dotsc,(n-2)^{a_2},(n-1)^{a_1})}\ba(r,s,\la)
\end{equation*}
differs from $\ba(r,s,\la)$ only in the top row, which is given by
\begin{equation*}
n^{s-\la _{1}}\overline{n-1}^{\la _{1}-\la _{3}-c_{2}^{-}}
\overline{n-3}^{\la _{3}-\la _{5}-c_{4}^{-}}\cdots \overline{n-r+1}
^{\la _{r-1}-c_{r}^{-}}\overline{2}^{a_{r}}
\end{equation*}
for $r$ is even and
\begin{equation*}
\overline{n}^{s-\la _{2}-c_{1}^{-}}\overline{n-2}^{\la _{2}-\la
_{4}-c_{3}^{-}}\overline{n-4}^{\la _{4}-\la _{6}-c_{5}^{-}}\cdots
\overline{n-r+1}^{\la _{r-1}-c_{r}^{-}}\overline{2}^{a_{r}}
\end{equation*}
for $r$ is odd.
\end{lemma}

\begin{proof}
We consider the $r$ even case. Consider the $(n-1)$-signature. $+$'s in the $
(2j)$-th row and $-$'s in the $(2j+1)$-th row from bottom cancel out for any $
j=1,\ldots ,r-1$. Hence $f_{((n-1)^{a_{1}})}$ acts only on the top row. We
proceed similarly.
\end{proof}

Let $t$ be the row tableau constructed by the Rule 3. The following lemma
allows us to calculate the action of Kashiwara operators on $t$ before
applying Rule 3.

\begin{lemma}
\label{lem:sig2} Let $t_{-}$ be another row tableau obtained by putting $
n^{k_{+}}\overline{n}^{k_{-}}$ instead of applying Rule 3 in each row. One
can formally apply $e_n$ and $f_n$ on $t_{-}$. Then the action of
$e_n$ (resp. $f_n$) commutes with applying Rule 3. A similar fact hold also for $e_{n-1}$
and $f_{n-1}$ by replacing $n^{k_{+}}\overline{n}^{k_{-}}$ with $
\overline{n}^{k_{-}}n^{k_{+}}$.
\end{lemma}

\begin{proof}
It suffices to prove the statement for a one-row tableau. This is done easily.
\end{proof}

\begin{lemma}
\label{lem:sig3} The row tableau $t_{2}$ for
\begin{equation*}
e_{(1^{a_r},\dotsc,(n-2)^{a_r},(n-1)^{a_r})}t
\end{equation*}
differs from $t$ only in the bottom row, which is given by $1^{a_{r}}
\overline{n}^{s-a_{r}}$.
\end{lemma}

\begin{proof}
In view of the previous lemma we can replace $t$ with $t_{-}$. Note that the
lowest row of $t$ is given by $n^{a_{r}}\overline{n}^{s-a_{r}}$. Since $
e_{((n-1)^{a_{r}})}$ acts only on the lowest row, we get $(n-1)^{a_{r}}
\overline{n}^{s-a_{r}}$. The application of $e_{(1^{a_r},\dotsc,(n-2)^{a_r})}$ is easier.
\end{proof}

Next we want to show $f_{(0^{a_r})}t_1=t_2$. To do this we calculate the $
\{3,4,\ldots\}$-highest element of $t_1$ and $t_2$. Let $N(\alpha)$ (resp. $
N^{\prime }(\alpha)$) be the number of letter $\alpha$ in $t_1$ (resp. $
(t_2)_-$) (see Lemma \ref{lem:sig2} for the definition of $t_-$). Define a
sequence $\boldsymbol{a}$ by $\boldsymbol{a}= \boldsymbol{a}_r\sqcup\cdots\sqcup\boldsymbol{a}_2
\sqcup\boldsymbol{a}_1$ where
\begin{equation*}
\boldsymbol{a}_j=((j+2)^{s+N(\overline{3})},\dotsc,(n-2)^{s+N(\overline{n-j-1})},
(n-\delta_j^{(2)})^{s+N(\overline{n-j})})
\end{equation*}
for $j=1,2,\ldots,r-1$, and $\boldsymbol{a}_r=((r+2)^{s-N(\overline{2})},\dotsc,
(n-2)^{s-N(\overline{2})},(n-\delta_r^{(2)})^{s-N(\overline{2})})$.
Here $\delta_j^{(2)}=1$ if $j$ is even, $=0$ otherwise and $\sqcup$ means the
concatenation of sequences. We also define $
\boldsymbol{a}^{\prime }= \boldsymbol{a}^{\prime }_1\boldsymbol{a}^{\prime
}_2\cdots\boldsymbol{a}^{\prime }_r$ by replacing $N(\overline{k})$ with $
N^{\prime }(\overline{k})$ for $k=3,4,\ldots,n-1$ and $N(\overline{2})$ with
$N^{\prime }(1)$ in $\boldsymbol{a}$. Then we have the following lemma.

\begin{lemma}
\label{lem:sig4}

\begin{itemize}
\item[(1)]  $e_{\boldsymbol{a}}t_{1}$ is a $\{3,4,\ldots \}$-highest element
whose $j$-th row from bottom is given by $(j+2)^{s}$ for $j=1,\ldots ,r-1$
and $(r+2)^{s-a_{r}}\overline{2}^{a_{r}}$ for $j=r$.

\item[(2)]  $e_{\boldsymbol{a}^{\prime }}t_{2}$ is a $\{3,4,\ldots \}$
-highest element whose $j$-th row from bottom is given by $
1^{a_{r}}3^{s-a_{r}}$ for $j=1$ and $(j+1)^{a_{r}}(j+2)^{s-a_{r}}$ for $
j=2,\ldots ,r$.

\item[(3)]  $N(\overline{k})=N^{\prime }(\overline{k})$ for $k=3,4,\ldots
,n-1$ and $N(\overline{2})=N^{\prime }(1)$.
\end{itemize}
\end{lemma}

\begin{proof}
For (1) and (2) simply calculate the action of $e_{\boldsymbol{a}}$ and $e_{
\boldsymbol{a}^{\prime }}$ using Lemma \ref{lem:sig2}. For (3) note that for
$1\leq j\leq n-3$ $N(\overline{n-j})=\la _{j}-c_{j+1}^{-}$ if $j$ is
odd, $=\la _{j+1}$ otherwise when $r$ is even, and $N(\overline{n-j}
)=\la _{j+1}$ if $j$ is odd, $=\la _{j}-c_{j+1}^{-}$ otherwise when $
r$ is odd. We also have $N(\overline{2})=N^{\prime }(1)=a_{r}$. For the
definition of the partition $\la ,c_{j}^{-}$ or $a_{r}$ see the
paragraph before Lemma \ref{lem:sig1}.
\end{proof}

Now we can prove Proposition \ref{prop:sigma tau} under the assumption that $
P$ has no column for which a + can be added. Using Lemma \ref{lem:e0} with $
\alpha=0,\beta=a_r,\gamma=s-a_r$ and with applying $e_1^{a_r}$, the results
in Lemma \ref{lem:sig4} show that $f_{(0^{a_r})}e_{\boldsymbol{a}}t_1=e_{
\boldsymbol{a}}t_2$. Since $f_0$ commutes with $e_j$ for $3\le j\le n$, we
obtain $f_{(0^{a_r})}t_1=t_2$, but this equality is what we wanted to show.

Finally, we prove Proposition \ref{prop:sigma tau} for general $\pm $
-diagram $P$. We show by induction on the number of columns for which a +
can be added. If there is no such column, the statement is proven already.
Now let $P$ be a $\pm $-diagram with at least one column for which a + can
be added. Let $c$ be the rightmost such column. Let $P^{\prime }$ be the $
\pm $-diagram obtained from $P$ by adding a + in column $c$. Let $h$ be the
height of this added +. Then it is known \cite{Sc} that $\Phi
(P)=f_{(1,\dotsc,h-1,h)}\Phi (P^{\prime })$. Hence, we have
\begin{equation*}
\sigma (\Phi (P))=f_{(n-1,\dotsc,n-h)}\sigma (\Phi (P^{\prime })).
\end{equation*}
Since we know the row tableau of $\sigma (\Phi (P^{\prime }))$ is given by
the Rule by induction hypothesis, it suffices to calculate the right hand
side and see it agrees with the row tableau of $\sigma (\Phi (P))$ given by
the Rule. Careful calculation using Lemma \ref{lem:sig2} shows that the
application of $f_{(n-1,\dotsc,n-h)}$ changes the row tableau of $\sigma
(\Phi (P^{\prime }))$ only in the rightmost column with letters $\overline{
n-h+1},\overline{n-h+2},\ldots ,\overline{n},\ldots $ reading from top to $
\overline{n-h},\overline{n-h+1},\ldots ,\overline{n-1},\ldots $. This
completes the proof of Proposition \ref{prop:sigma tau}.

\end{document}